\documentclass[11pt,oneside]{amsart}
\usepackage[a4paper,margin=2.5cm]{geometry}
\usepackage{amsmath,amssymb,amsfonts,amsthm,mathrsfs}
\usepackage[abbrev,nobysame]{amsrefs}
\usepackage{epic}
\usepackage{amsopn,amscd,graphicx}
\usepackage{color,transparent} 
\usepackage[dvipsnames]{xcolor}
\usepackage
[colorlinks, breaklinks,
bookmarks = false,
linkcolor = NavyBlue,
urlcolor = ForestGreen,
citecolor = ForestGreen,
hyperfootnotes = false
]
{hyperref}
\usepackage{comment}

\usepackage{palatino,mathpazo}
\usepackage{dsfont}
\usepackage{enumerate}
\usepackage{bm}

\numberwithin{equation}{subsection}
\usepackage[dvipsnames]{xcolor}
 1
 1
 1

\newtheorem{theorem}{Theorem}[section]
\newtheorem{corollary}[theorem]{Corollary}

\newtheorem{lemma}[theorem]{Lemma}

\theoremstyle{definition}
\newtheorem{definition}[theorem]{Definition}

\newtheorem{remark}[theorem]{Remark}

\newtheoremstyle{named}{}{}{\itshape}{}
{\bfseries}{.}{.5em}{\thmnote{#3}#1}
\theoremstyle{named}

\DeclareMathOperator{\End}{End} 
\DeclareMathOperator{\Dom}{Dom}
\DeclareMathOperator{\Spec}{Spec} 
\DeclareMathOperator{\Ker}{Ker} 
\DeclareMathOperator{\supp}{supp}
\DeclareMathOperator{\Ph}{Ph}

\newcommand{\R}{\mathbb{R}}
\newcommand{\C}{\mathbb{C}}
\newcommand{\N}{\mathbb{N}}
\newcommand{\Z}{\mathbb{Z}}

\def\ker{{\rm Ker}~}
\def\Re{\mathrm{Re}}
\def\Im{\mathrm{Im}}
\def\cC{\mathscr{C}^\infty}

\def\cCc{\mathscr{C}^\infty_c}

\newcommand{\abs}[1]{\left\vert#1\right\vert}

\newcommand{\set}[1]{\left\{#1\right\}}
\newcommand{\To}{\rightarrow}
\newcommand{\pr}{\partial}
\newcommand{\ol}{\overline}

\newcommand{\Td}{\widetilde}

\newcommand{\OK}{O\left(k^{-\infty}\right)}

\newcommand{\im}{{\rm Im}~}

\newcommand{\range}{{\rm Range}}

\usepackage{scalerel}
\usepackage{stackengine,wasysym}

\title[Spectral asymptotics of semi-classical Levi elliptic Toeplitz operators]
{Spectral asymptotics of semi-classical Toeplitz operators on Levi non-degenerate CR manifolds}

\author[Wei-Chuan Shen]{Wei-Chuan Shen}
\address{CY Cergy Paris University, Site Saint-Martin, 2 avenue Adolphe Chauvin, 95300 Pontoise, France.}
\email{wei-chuan.shen@u-cergy.fr or wshen@uni-koeln.de}
\thanks{The author was partially supported by the DFG funded projects
SFB/TRR 191 "Symplectic Structures in Geometry, Algebra and Dynamics"
(281071066-TRR 191) and is supported by ANR funded project "Quantification des variétés de Caractères comme Modèle pour le chaos quantique"  (ANR-23-CE40-0021-01).
}

\date{\today}
\begin{document}

\begin{abstract}
We consider any compact CR manifold whose
Levi form is non-degenerate of constant signature $(n_-,n_+)$, $n_-+n_+=n$. For $\lambda>0$ and $q\in\{0,\cdots,n\}$, we let $\Pi_\lambda^{(q)}$ be the spectral projection of the Kohn Laplacian of $(0,q)$-forms
corresponding to the interval $[0,\lambda]$. For certain classical pseudodifferential operators $P$, we study
a class of generalized elliptic Toeplitz operators
$T_{P,\lambda}^{(q)}:=\Pi_\lambda^{(q)}\circ P\circ \Pi_\lambda^{(q)}$. For any cut-off $\chi\in\cCc(\R\setminus\{0\})$, we establish the full asymptotics of the semi-classical spectral projector  $\chi(k^{-1}T_{P,\lambda}^{(q)})$ as $k\to+\infty$. Our main result conclude that the smooth Schwartz kernel $\chi(k^{-1}T_{P,\lambda}^{(n_-)})(x,y)$ is the sum of two semi-classical oscillatory integrals with complex-valued phase functions.

\end{abstract}
\maketitle
\tableofcontents	
\section{Preliminaries}
\subsection{Introduction and main results}
	The theory of Toeplitz operators is a classical subject in several complex variables and has a deep relation to microlocal analysis: for a bounded strictly pseudoconvex domain $M\subset\C^{n+1}$ with the smooth boundary $\partial M$, $n\geq 1$, we let $H_b^0(\partial M)$ the closure in $L^2(X)$ of the space of boundary values of holomorphic functions on $M$. We call operators of the form $T_P:=\Pi\circ P\circ \Pi$ Toeplitz operators, where $\Pi$ is the orthogonal projection of $L^2(\partial M)$ onto $H_b^0(\partial M)$ and $P$ is a pseudodifferential operator on $\partial M$. Inspired by the earlier results of Melin--Sj\"ostrand \cites{MeSj74} and Boutet de Monvel--Sj\"ostrand \cite{BouSj75} on Fourier integral operators with complex phase and the off-diagonal asymptotic expansion of the singularities of the Szeg\H{o} projection, in \cite{Bou79} Boutet de Monvel proved that these operators admit symbolic calculus as pseudodifferential operators, and he gave a famous result about a variant of the Atiyah--Singer index theorem in this context. We refer the microlocal technique and spectral theory of Toeplitz operators to the monograph of Boutet de Monvel--Guillemin \cite{BG81}. We also mention some works linked to such point of view \cites{BMS94,Zel98,SZ02,ZZ19,CRa23,Shabtai24,Cat99,BPU98,Ch03,Ch24APDE,KS01,IKPS20,HM17CPDE,HHL20,GaHs23,GaHs24,HsMar24,DLM06,MaMar06IJM,MM07,MM12,HHMS23,HHMS24,Hs12AGAG,CHH24,DraLiuMar24,Finski25}. 

 The purpose of this paper is to generalize the semi-classical analysis of Toeplitz operators recently by Herrmann, Hsiao, Marinescu and the author to the case of Levi non-degenerate Cauchy--Riemann manifolds. Semi-classical analysis is a branch of microlocal analysis, and Levi non-degenerate CR manifolds play an important role in the context of microlocal analysis such as \cites{Bou74,Koh86,Hoe04,Hs10,HsHua21}. The foundation of this paper is the microlocal structure of the Szeg\H{o} projection on lower energy forms by \cite{HM17JDG} and the semi-classical approach developed in \cite{HHMS23} for the spectral theory of Toeplitz operators. Different from the calculus appeared in the main text of \cite{BG81}, we treat Toeplitz operators as Fourier integral operators of complex phases \cite{MeSj74}. The main results, appeared later at \eqref{eq:q not in n_- adn n_+ intro} and \eqref{eq:q=n_- expansion intro}, can be interpreted as a form of regularization of Toeplitz operators achieved through semi-classical spectral asymptotics via smooth cut-off functions. We can also see these results as the counterparts within the context of CR manifolds for two well-known theorems on complex manifolds: the Andreotti--Grauert vanishing theorem and the Bergman kernel expansion for high powers of line bundles associated with mixed curvature. We refer the semi-classical analysis of the related subject to \cites{BerSj07,HM14,HM17CPDE,MaMar06IJM,MM07,Ch24APDE}, to quote just a few. It is our hope that the results presented in this paper could contribute to the growing interest in semi-classical analysis in several complex variables, particularly in the study of Berezin–Toeplitz quantization.

From now on, we always consider the compact and non-degenerate CR manifold $(X,T^{1,0}X)$ of real dimension $2n+1$ and $n\geq 1$. We denote by $\bm\alpha$ the contact form on $X$ such that the complex-valued Hermitian form $\mathcal{L}:=\frac{i}{2}d\bm\alpha|_{T^{1,0}X}$, called Levi form, is non-degenerate everywhere. Then the numbers of the negative and positive eigenvalues of $\mathcal{L}$ are always the constant and we denote them by $n_-$ and $n_+$, respectively. The pair $(n_-,n_+)$ is called the signature of (the Levi form of) $X$. For any classical pseudodifferential operator $P\in L^1_{\rm cl}(X;T^{*0,q}X)$ of first order such that $P:\cC(X,T^{*0,q}X)\to \cC(X,T^{*0,q}X)$, we consider the Toeplitz operator
	\begin{equation}
    \label{eq:Toeplitz operator}
	T_{P,\lambda}^{(q)}:=\Pi_\lambda^{(q)}\circ P \circ \Pi_\lambda^{(q)},
	\end{equation}
	where $\lambda>0$ is any number, $\Pi_\lambda^{(q)}:L^2_{0,q}(X)\to E([0,\lambda])$ is the orthogonal projection called \textit{Szeg\H{o} projection on lower energy forms}, $L^2_{0,q}(X)$ is the square integrable $(0,q)$ forms on $X$, and the subspace of lower energy forms $E([0,\lambda]):=\range~\mathds 1_{[0,\lambda]}(\Box_b^{(q)})$ is the image of the spectral projection of the Kohn Laplacian $\Box_b^{(q)}$ (extended by Gaffney extension). We always assume that $P$ is formally self-adjoint. When $q=n_-$ we assume the following \textit{Levi-ellipticity conditions}. We let $\{W_j\}_{j=1}^{n}$
		be an orthonormal frame of $T^{1,0}X$ in a neighborhood of $x$ such that $\mathcal{L}_{x}(W_{j},\overline{W}_{s})=\delta_{j,s}\mu_{j}$, $j,s=1,\cdots,n$, and $\mu_1\leq\cdots\leq\mu_{n_-}<0<\mu_{1+n_-}\leq\cdots\leq\mu_n$. We take the dual basis $\{\omega_j\}_{j=1}^{n}$ of $T^{*0,1}X$ with respect to $\{\ol W_j\}^{n}_{j=1}$ and we consider the subspaces $\mathcal{N}_{x}^{n_-}:=\set{c\omega_{1}(x)\wedge\cdots\wedge \omega_{n_-}(x):c\in\C}\subset T^{*0,n_-}_x X$ and $\mathcal{N}_{x}^{n_+}:=\set{c\omega_{1+n_-}(x)\wedge\cdots\wedge \omega_{n}(x):c\in\C}\subset T^{*0,n_+}_x X$. By the given Hermitian metric on $\C TX$, we define the orthogonal projection
	\begin{align}
&\tau_{x_0}^{n_\mp}:T^{*0,n_\mp}_{x}X\To\mathcal{N}_{x_0}^{n_\mp}.\label{eq:frame indicator mp intro}
	\end{align}
Assuming $P$ is Levi elliptic means that, for the principal symbol $p_0\in\cC\left(T^*X,\End(T^{*0,q}X)\right)$ of $P$, for all $x\in X$ we require
	\begin{equation}
	\label{eq:Levi elliptic n_- intro}
	\tau_x^{n_-} p_0(-\bm\alpha_x)\tau_x^{n_-}> 0\ \ \text{when $q=n_-$},
	\end{equation}
	and we additionally suppose
	\begin{equation}
	\label{eq:Levi elliptic n_+ intro}
	\tau_x^{n_+} p_0(\bm\alpha_x)\tau_x^{n_+}< 0\ \ \text{when $q=n_-=n_+$}.
	\end{equation}
	We will see in Theorem \ref{thm:T_P is self-adjoint} that $T_{P,\lambda}^{(q)}$ has a self-adjoint $L^2$-extension through maximal extension. In fact, for $q\notin\{n_-,n_+\}$, $T_{P,\lambda}^{(q)}$ is a compact operator. When $q=n_-$,  we will see in Theorem \ref{thm:Spec(T_P,lambda^q, q=n_-)} that the set $\Spec T_{P,\lambda}^{(q)}\subset\R$ is also discrete and the accumulation points of the spectrum is the subset of $\{-\infty,+\infty\}$. Roughly speaking, the conditions \eqref{eq:Levi elliptic n_- intro} and \eqref{eq:Levi elliptic n_+ intro} are responsible for the part of eigenvalues accumulated at $+\infty$ and $-\infty$, respectively.
 
Next, for any function $\chi\in\cCc(\R\setminus\{0\})$, our main result describes the operator $\chi(k^{-1}T_{P,\lambda}^{(q)})$ as the sum of two semi-classical Fourier integral operators modulo some $k$-negligible operator when $q=n_-$ and $k\to+\infty$. From the spectral theorem of $T_{P,\lambda}^{(q)}$, with respect to $L^2$-inner product, we find an orthonormal system $\{f_j\}_j$ such that $T_{P,\lambda}^{(q)}f_j=\lambda_j f_j$, $\lambda_j\neq 0$, and
 \begin{equation}
 \label{eq:reproducing kernel}
     \chi(k^{-1}T_{P,\lambda}^{(q)})(x,y)=\sum_{j}\chi(k^{-1}\lambda_j)f_j(x)\otimes f_j^*(y).
 \end{equation}
However, to state the precise semi-classical spectral asymptotics, we need more detail of the fundamental theorem  \cite{HM17JDG}*{Theorem 4.1} about the microlocal structure of $\Pi_\lambda^{(q)}$, cf.~also \S 1.4. For $q=n_-$ and any coordinate patch $(\Omega,x)$ on $X$,
	we have some $\varphi_\mp(x,y)\in\cC(\Omega\times \Omega,\C)$ with
	\begin{align}
	&\operatorname{Im}\varphi_\mp(x,y)\geq 0,\ \ \varphi_\mp(x,y)=0\iff x=y,\ \ d_x\varphi_\mp(x,x)=-d_y\varphi_\mp(x,x)=\mp\bm\alpha(x),\label{eq:varphi intro}
	\end{align}
and some proplery supported H\"ormander symbol $s^\mp(x,y,t)$ with the asymptotic expansion
 \begin{align}
&s^\mp(x,y,t)\sim\sum_{j=0}^{+\infty}s^\mp_j(x,y)t^{n-j}~\text{in}~S^n_{1,0}\left(\Omega\times \Omega\times\R_+,\End(T^{*0,q}X)\right),\label{eq:full symbol intro I}
	\end{align}
 such that for the Fourier integral operator $S_\mp$ determined by
	\begin{equation}
	\label{eq:approximated Szego kernel intro}
	S_\mp(x,y)=\int_0^{+\infty} e^{it\varphi_\mp(x,y)}s^\mp(x,y,t)dt,
	\end{equation}
 we have $S_+=0$ when $n_-\neq n_+$ and
 \begin{equation}
     \Pi^{(q)}_\lambda= S_-+S_+ + F\ \ \text{on $\Omega$},\ \ F(x,y)\in\cC(\Omega\times\Omega).
 \end{equation}
 When $q\notin\{n_-,n_+\}$, $\Pi_\lambda^{(q)}$ is a smoothing operator. Our main result is the following.
	\begin{theorem}
		\label{thm:Main Semi-Classical Expansion}
	We let $(X,T^{1,0}X)$  be a compact, non-degenerate CR manifold, and $\dim_\R X:=2n+1$ for $n\geq 1$. We let $\bm\alpha$ be the contact form on $X$ such that the Levi form has the constant signature $(n_-,n_+)$. For $q\in\{0,\cdots,n\}$ and any formally self-adjoint $P\in L^1_{\rm cl}(X;T^{*0,q}X)$, we assume that when $q=n_-$ we have the Levi-ellipticity conditions \eqref{eq:Levi elliptic n_- intro} and \eqref{eq:Levi elliptic n_+ intro}. For any $\lambda>0$ and any $\chi\in\cCc(\R\setminus\{0\},\C)$, $\chi\not\equiv 0$, we have the semi-classical spectral asymptotics of the Toeplitz operator \eqref{eq:Toeplitz operator}:
		\begin{equation}
  \label{eq:q not in n_- adn n_+ intro}
		\chi(k^{-1}T_{P,\lambda}^{(q)})=0\ \ \text{on}\ X:\ \ q\notin\{n_-,n_+\},\ \ k\gg 1,
		\end{equation}
		and for each coordinate patch $(\Omega,x)$ in $X$ we have the off-diagonal asymptotic expansion of the Schwartz kernel in $\cC$-topology by
		\begin{equation}
          \label{eq:q=n_- expansion intro}
        \begin{split}
            \chi(k^{-1}T_{P,\lambda}^{(q)})(x,y)
		=&\int_0^{+\infty} e^{ikt\varphi_-(x,y)}A^-(x,y,t;k)dt
        +\int_0^{+\infty} e^{ikt\varphi_+(x,y)}A^+(x,y,t;k)dt\\
        +&\OK \text{on}\ \  \Omega\times\Omega:\ \  q=n_-,\ \ k\gg 1,
        \end{split}
		\end{equation}
where $\varphi_\mp(x,y)\in\cC(\Omega\times\Omega,\C)$ satisfies the property \eqref{eq:varphi intro} and we also have
		\begin{align}
&A^\mp(x,y,t;k)\sim\sum_{j=0}^{+\infty}A^\mp_j(x,y,t)k^{n+1-j}\ \ 
		\text{in}~S^{n+1}_{\rm loc}\left(1;\Omega\times\Omega\times\R_+,\End(T^{*0,q}X)\right),\label{eq:A(x,y,t,k) expansion intro}\\
  &\text{$A^+(x,y,t;k)=0$ when $n_-\neq n_+$;\ \ $A_0^-(x,x,t)\neq 0$;\ \ $A_0^+(x,x,t)\neq 0$ when $n_-=n_+$.}
		\end{align}
In fact, when $\supp\chi\cap\R_+\neq\emptyset$, there is an interval $I_-\Subset\R_+$ such that when $A_j^-(x,y,t)\neq 0$ and $A^-(x,y,t)\neq 0$ we have $t\in I_-$ for all $j\in\N_0$; when $n_-=n_+$ and $\supp\chi\cap\R_-\neq\emptyset$, there is also an interval $I_+\Subset\R_+$ such that when $A_j^+(x,y,t)\neq 0$ and $A^+(x,y,t)\neq 0$ we have $t\in I_+$ for all $j\in\N_0$.  Moreover, for any $\tau_1,\tau_2\in\cC(X)$ such that $\supp\tau_1\cap\supp\tau_2=\emptyset$, we have 
		\begin{equation}
		\tau_1\circ\chi_k(T_{P,\lambda}^{(q)})\circ\tau_2=O(k^{-\infty})\ \ \text{on $X$},
		\end{equation}
		where $\tau_1$ and $\tau_2$ are seen as the multiplication operator.
	\end{theorem}
 Our method heavily relies on microlocal analysis of \cites{MeSj74,BouSj75,Hs10,HM17JDG,GaHs23} and especially the semi-classical microlocal approaches introduced in \cite{HHMS23}. The study employs the approach of Melin--Sj\"ostrand, Boutet de Monvel--Sj\"ostrand, Hsiao--Marinescu and Galasso--Hsiao, utilizing a calculus of specific complex phase Fourier integral operators, cf.~\S 3. Additionally, it incorporates a semi-classical analysis on a distinct integral, as defined by the Helffer--Sj\"ostrand formula, cf.~\S 4.1. This kind of analysis was well-studied for order zero Toeplitz operators \cites{GaHs23,GaHs24} and for the order one situation \cite{HHMS23}. The main contribution in this work is the semi-classical microlocal analysis under the Levi ellipticity condition of $T_{P,\lambda}^{(n_-)}$, which is inspired by \cite{GaHs23}*{Lemma 4.1}. We notice that this condition is clearly a generalization of the concept of elliptic Toeplitz operators because we allow mild degeneracy of the principal symbol of the pseudodiffential operator $P$ used to define $T_{P,\lambda}^{(n_-)}$. On the other hand, our relatively general ellipticity assumption restrains us from arguing directly as in the case of $(0,0)$-forms by Boutet de Monvel--Guillemin. But we can still construct the parametrix of $T_{P,\lambda}^{(n_-)}$ in the space of lower energy CR harmonic forms in this context. In fact, such parametrix is also in the form of Toeplitz operators we consider, cf.~Theorem \ref{thm:parametetrix of Toeplitz operators for lower energy forms}. Another slight improvement of this work comparing to the existing one is that we do not use the estimates about the numbers of the eigenvalues of Toeplitz operators \cite{BG81}*{Proposition 12.1} as in \cite{HHMS23} to obtain our main semi-classical expansion.

	We have the description of leading term of our main result through the local picture \eqref{eq:frame indicator mp intro}. We let $m(x)dx$ be the given volume form on $X$ and $v(x)dx$ be the volume form induced by the Hermitian metric on $\C TX$ which is compatible with $\bm\alpha$. For the asymptotic expansion \eqref{eq:full symbol intro I} of $s^\mp(x,y,t)$, by \cite{HM17JDG}*{Theorem 3.5}, when $q=n_-$ we have
	\begin{align}
	&s_0^-(x, x)=\frac{\abs{\det\mathcal{L}_{x}}}{2\pi^{n+1}}\frac{v(x)}{m(x)}\tau_{x}^{n_-},\ \ x\in\Omega,\label{eq:s_0^- intro}
	\end{align}
 and when $n_-=n_+$ we also have
 \begin{equation}
     s_0^+(x, x)=\frac{\abs{\det\mathcal{L}_{x}}}{2\pi^{n+1}}\frac{v(x)}{m(x)}\tau_{x}^{n_+},\ \ x\in\Omega.\label{eq:s_0^+ intro}
 \end{equation}
Here $\det\mathcal{L}_x$ is the product $\mu_1(x)\cdots\mu_n(x)$ of the eigenvalues $\{\mu_j(x)\}_{j=1}^n$ of the Levi form $\mathcal{L}_x$. By assumption we have $\mu_j(x)<0$ for $1\leq j\leq n_-$ and $\mu_j(x)>0$ for $n_-+1\leq j\leq n$. We let
	\begin{align}
	&I_0:=\{1,\cdots,n_-\},~J_0:=\{n_-+1,\cdots,n\}.
        \end{align}
	With respect to the orthonormal basis $\{T_j\}_{j=1}^n$ of $T^{*0,1}X$, the principal symbol $p_0(x,\eta)$ of $P$ reads
	\begin{equation}
	\label{eq:system of the principal symbol intro}
	p_0(x,\eta)={\sum_{|{\bf I}|=|{\bf J}|=q}} p_{\bf I, J}(x,\eta) \omega_{\bf I}^\wedge\otimes \omega_{\bf J}^{\wedge,*},
	\end{equation}
	where $p_{\bf I,J}(x,\eta)\in\cC(T^*X,\C)$ and ${\bf I,J}$ are strictly increasing index sets. Then the Levi-ellipticity conditions \eqref{eq:Levi elliptic n_- intro} and \eqref{eq:Levi elliptic n_+ intro} now become
	\begin{equation}
	p_{I_0,I_0}(-\bm\alpha_x)> 0~\text{and}~p_{J_0,J_0}(\bm\alpha_x)< 0,
	\end{equation}
	respectively. We can now formulate the leading coefficient of our expansion.
	\begin{theorem}
		\label{thm:Main Leading Term}
		Following Theorem \ref{thm:Main Semi-Classical Expansion} and the above local picture, if $q=n_-$, for leading term $A_0^-(x,y,t)$ in the expansion \eqref{eq:A(x,y,t,k) expansion intro}  we have
		\begin{equation}
		\label{eq:leading term A_-}
		A_0^-(x, x, t)=t^n\chi(p_{I_0,I_0}(-\bm\alpha_x)t)\frac{\abs{\det\mathcal{L}_{x}}}{2\pi^{n+1}}\frac{v(x)}{m(x)}\tau_x^{n_-}.
		\end{equation}
		In addition, if $n_-=n_+$, for leading term $A_0^+(x,y,t)$ in the expansion \eqref{eq:A(x,y,t,k) expansion intro} we have
		\begin{equation}
		\label{eq:leading term A_+}
		A_0^+(x, x, t)=t^n\chi(p_{J_0,J_0}(\bm\alpha_x)t)\frac{\abs{\det\mathcal{L}_{x}}}{2\pi^{n+1}}\frac{v(x)}{m(x)}\tau_{x}^{n_+}.
		\end{equation}
	\end{theorem}
 \begin{corollary}
		\label{coro:Semi-Classical Expansion}
		Following Theorem \ref{thm:Main Semi-Classical Expansion}, we have the asymptotic expansion for $q=n_-$ that
		\begin{equation}
        \label{eq:on diagonal expansion}
		\chi(k^{-1}T_{P,\lambda}^{(q)})(x,x)\sim\sum_{j=0}^{+\infty}k^{n+1-j} \left(A_j^-(x)+A_j^+(x)\right)~\text{in}~S^{n+1}_{\rm loc}(1;X,\End(T^{*0,q}X)),
		\end{equation}    
		where for all $j\in\N_0$ we have $A_j^\mp(x)=\int_0^{+\infty}A_j^\mp(x,x,t)dt\in\cC(X,\End(T^{*0,q}X))$, and the local description of $A_0^\mp(x)$ is explicit through Theorem~\ref{thm:Main Leading Term}.
	\end{corollary}
By \cite{HM17JDG}*{Theorem 1.12} and \cite{Koh86}, we can apply Theorem~\ref{thm:Main Semi-Classical Expansion} to the case that $(X,T^{1,0}X)$ is compact, strictly pseudoconvex and CR embeddable. In this case,  $\Pi^{(q)}=\Pi_\lambda^{(q)}$ for some $\lambda>0$ at $q=n_-=0$. When $P=-iT$ and the Lie derivative $\mathscr{L}_T dV=0$ (e.g.~$dV=dV_\alpha$), the way of $A_j^-(x,x,t)$ depending on $\chi(t)$, $j\in\N$, and the precise formula of $A_1^-(x,x,t)$ are obtained in \cite{CHH24}. A particular case is the principal circle bundle given by the pair of a complex manifold and a Hermitian line bundle. We will follow the lines in \cites{HHMS23,HHMS24} to give some examples for the specific set-up of CR manifolds with transversal CR circle action. In \S 4.2, we discuss Corollary~\ref{coro:Scaled Spectral Measures} for the case where the circle action is free, and in \S 4.3 we revisit Theorem~\ref{thm:Main Semi-Classical Expansion}  for the case where the circle action is only locally free.
	  
  We have the following Szeg\H{o} type limit theorem, cf.~also \cite{BG81}*{\S 13} and \cite{HHMS23}*{Theorem 1.3}. With respect to \eqref{eq:reproducing kernel}, we consider the scaled spectral measures \(\mu_k\) given by 
\begin{equation}\label{eq:scaledspectraclmeasuredefinition}
\mu_k^{(q)}=k^{-n-1}\sum_{j\in J}\delta\big(t-k^{-1}\lambda_j\big).
\end{equation} 
\begin{corollary}
\label{coro:Scaled Spectral Measures}
In the situation of Theorem~\ref{thm:Main Semi-Classical Expansion}, for $q=n_-$
the scaled spectral measures \(\mu_k\) 
converges weakly as \(k\to+\infty\) to the continuous measure $\mu_\infty^{(q)}$ on $\R\setminus\{0\}$ given by
\begin{equation}
\mu_\infty^{(q)}=\mathcal{C}_P^{(q)}\,t^n dt,\ \ 
\mathcal{C}_P^{(q)}:=\frac{1}{2\pi^{n+1}}\left(
\int_X\frac{|\det\mathcal{L}_x| v(x)dx}{p^{n+1}_{I_0,I_0}(-\bm\alpha)}+\mathds 1_{\{n_+\}}(q)\int_X\frac{|\det\mathcal{L}_x| v(x)dx}{p^{n+1}_{J_0,J_0}(\bm\alpha)}\right),
\end{equation}
where $dt$ is the Lebesgue measure on $\R$.
\end{corollary} 
	\subsection{Elements of microlocal and semi-classical analysis}
We use the following notations and conventions throughout this article.
	$\mathbb{Z}$ is the set of integers, 
	$\N=\{1,2,3,\cdots\}$ is the set of natural numbers and we 
	put $\N_0=\N\bigcup\{0\}$; $\mathbb R$ is the set of
	real numbers, $\R_+:=\{x\in\R:x>0\}$ and 
	${\dot\R}:=\R\setminus\{0\}$; $\C$ is the set of complex numbers and $\dot\C:=\C\setminus\{0\}$.
For a multi-index $\alpha=(\alpha_1,\cdots,\alpha_n)\in\N^n_0$ 
	and $x=(x_1,\cdots,x_n)\in\mathbb R^n$, we set
	\begin{align}
	&x^\alpha=x_1^{\alpha_1}\cdots x^{\alpha_n}_n\,,\quad
	 \partial_{x_j}=\frac{\partial}{\partial x_j}\,,\quad
	\partial^\alpha_x=\partial^{\alpha_1}_{x_1}\cdots\partial^{\alpha_n}_{x_n}
	=\frac{\partial^{|\alpha|}}{\partial x^\alpha}\cdot
	\end{align}
We let $z=(z_1,\cdots,z_n)$, $z_j=x_{2j-1}+ix_{2j}$, $j=1,\cdots,n$, 
	be coordinates on $\C^n$. We write
	\begin{align}
	&z^\alpha=z_1^{\alpha_1}\cdots z^{\alpha_n}_n\,,
	\quad\ol z^\alpha=\ol z_1^{\alpha_1}\cdots\ol z^{\alpha_n}_n\,,\\
	&\partial_{z_j}=\frac{\partial}{\partial z_j}=
	\frac{1}{2}\Big(\frac{\partial}{\partial x_{2j-1}}-
	i\frac{\partial}{\partial x_{2j}}\Big)\,,
	\quad\partial_{\ol z_j}=\frac{\partial}{\partial\ol z_j}
	=\frac{1}{2}\Big(\frac{\partial}{\partial x_{2j-1}}+
	i\frac{\partial}{\partial x_{2j}}\Big),\\
	&\partial^\alpha_z=\partial^{\alpha_1}_{z_1}\cdots\partial^{\alpha_n}_{z_n}
	=\frac{\partial^{|\alpha|}}{\partial z^\alpha}\,,\quad
	\partial^\alpha_{\ol z}=\partial^{\alpha_1}_{\ol z_1}
	\cdots\partial^{\alpha_n}_{\ol z_n}
	=\frac{\partial^{|\alpha|}}{\partial\ol z^\alpha}\,.
	\end{align}
 For $j, s\in\mathbb Z$, we set $\delta_{js}=1$ if $j=s$, 
	$\delta_{js}=0$ if $j\neq s$. All the smooth manifolds in this work are assumed to be paracompact.
    
 In this section we recall basic notions of microlocal and semi-classical analysis, and we refer to \cites{GriSj94,DiSj99,MeSj74,Hoe03,Hoe71} for detail.
 
	For a $\cC$-orientable manifold $W$, we let $TW$ and $T^*W$ denote the tangent bundle of $W$ and the cotangent bundle of $W$ respectively.
	The complexified tangent bundle of $W$ and the complexified cotangent bundle of $W$
	will be denoted by $\C TW$ and $\C T^*W$ respectively. We write $\langle\,\cdot\,,\cdot\,\rangle$
	to denote the pointwise duality between $TW$ and $T^*W$.
	We extend $\langle\,\cdot\,,\cdot\,\rangle$ bilinearly to $\C TW\times\C T^*W$. We let $E$ be a $\cC$-vector bundle over $W$. The fiber of $E$ at $x\in W$ will be denoted by $E_x$. With respect to the base manifold $W$, the spaces of smooth sections of $E$ will be denoted by $\cC(W, E)$, and we let $\cCc(W,E)$ be the subspace of $\cC(W, E)$ whose elements have compact support in $W$; the spaces of distribution sections of $E$ will be denoted by $\mathscr D'(W, E)$, and we let $\mathscr E'(W,E)$ be the subspace of $\mathscr D'(W, E)$ whose elements have compact support in $W$. We denote $I$ to be the identity map on $W$. For an open set $V\subset W$, $f\in\cC(V\times V,E)$ and a number $N\in\N$, we write $f=O(|x-y|^{+\infty})$ if $f$ vanishes to infinite order at the diagonal; when $E=\C$ this means that for any $N\in\N$ we have $(\pr_x^\alpha\pr_y^\beta f)(x,x)=0$ for all $x\in V$ and $|\alpha|+|\beta|\leq N$.
 
	 We let $E$ and $F$ be $\cC$-vector
	bundles over orientable $\cC$-manifolds $W_1$ and $W_2$, respectively, equipped with smooth densities of integration. If
	$A:\cCc(W_2,F)\to\mathscr D'(W_1,E)$
	is continuous, we write $A(x,y)$ to denote the Schwartz kernel of $A$. The Schwartz kernel theorem implies that $A$ is continuous: $\mathscr E'(W_2,F)\to\cC(W_1,E)$ and $A(x,y)\in\cC\left(W_1\times W_2,\mathscr L(F,E)\right)$ are equivalent. Here we write $\mathscr L(F,E)$ to denote the vector bundle with fiber over $(x,y)\in W_1\times W_2$ consisting of the linear maps $\mathscr L(F_y,E_x)$ from $F_y$ to $E_x$, and we write $\End (E):=\mathscr L(E,E)$. If $A(x,y)\in\cC(W_1\times W_2,\mathscr{L}(F,E))$ we say that $A$ is smoothing on $W_1 \times W_2$. For  continuous operators $A,B:\cCc(W_2,F)\to\mathscr{D}'(W_1,E)$, we write 
	\begin{equation}
	\label{eq:up to a smoothing operator}
	\mbox{$A\equiv B$ on $W_1\times W_2$} 
	\end{equation}
	if $A-B$ is a smoothing operator. If \eqref{eq:up to a smoothing operator} holds when $W_1=W_2=W$, we simply write $A\equiv B$ on $W$ or just $A\equiv B$.  For an open set $V\subset W$, we say that a distributional section $A(x,y)\in \mathscr D'(V\times V,\mathscr L(E,E))$, which possibly smoothly depends on some other parameter, is properly supported (in the variables $(x,y)$) if the restrictions of the two projections 
	$(x,y)\mapsto x$, $(x,y)\mapsto y$ to ${\rm supp\,}A(x,y)$
	are proper maps, and we say an operator $A$ is properly supported (in $V$) if the Schwartz kernel $A(x,y)$ is properly supported. 

 For a $\cC$-vector bundle $E$ over a $\cC$-orientable compact manifold $W$ and any number $s\in\R$, with respect to the standard $L^2$-norm $\|\cdot\|$ for the section of $E$ we let $H^s(W,E)$ to be the standard Sobolev space of order $s$ for sections of $E$ with the Sobolev norm $\|\cdot\|_s$. We let $H_{\rm comp}^s(W, E)$ be the subspace of $H^s(W,E)$ whose elements have compact support in $W$. For a relatively compact open set $U\Subset W$, we put $H^s_{\rm loc\,}(U, E)=\{u\in\mathscr D'(U, E):\chi u\in H_{\rm comp}^s(U, E)\,,\forall\chi\in\cCc(U)\}$.
For smooth vector bundles $E,F$ over $W$ and an operator $F_z: H^{s_1}_{\rm comp}(W_1, E)\rightarrow H^{s_2}_{\rm loc}(W_2, F)$ smoothly depending on some parameter $z\in\C$, we write 
	\begin{equation}
	F_z=O\left(g(z)\right)~\text{in}~\mathscr L\left(H^{s_{1}}_{\rm comp}(W_{1}, E),H^{s_{2}}_{\rm loc}(W_{2}, F)\right)
	\end{equation}
	if for every $z\in\C$ the operator $F_{z}: H^{s_{1}}_{\rm comp}(W_{1}, E)\rightarrow H^{s_{2}}_{\rm loc}(W_{2}, F)$ is continuous and for any $\chi_j\in\cCc(W_{j})$, $j=1, 2$, $\tau_1\in\cCc(W_1)$, $\tau_1\equiv 1$ on ${\rm supp\,}\chi_1$, there is a constant $c>0$ independent of $z$ such that $\|\chi_2 F_z\chi_1 u\|_{s_{2}}\leq c \cdot|g(z)|\cdot\|\tau_1 u\|_{s_{1}}$ for all $u\in H^{s_{1}}_{\rm loc}(W_{1}, E)$.

For any $m\in\mathbb R$, $0\leq\delta\leq\rho\leq 1$, $N\in\N$ and the smooth vector bundle $E$ over an open set $V\subset \R^n$, we denote by $S^m_{\rho,\delta}\left(V\times\R^N,E\right)$ the H\"ormander symbol space of order $m$ with type $(\rho,\delta)$, and we denote by $S^{m}_{\rm cl}\,\left(V\times\R^N,E\right)\subset S^m_{1,0}\left(V\times\R^N,E\right)$ the classical symbol space. We always use the standard theory of asymptotic sum in this context throughout this paper. 

For open sets $V_1\subset\R^{n_1}$, $V_2\subset\R^{n_2}$, $V:=V_1\times V_2$, for any regular phase function $\varphi(x,\eta)\in\cC(V\times\dot\R^{N})$ and
$a(x,y,\eta)\in S^m_{\rho,\delta}(V\times\R^N,E)$, a continuous operator $A:\cCc(V_2)\to\cC(V_1)$ 
 determined by an oscillatory integral or a Fourier distribution $A(x,y)=\int e^{i\varphi(x,y,\eta)}a(x,y,\eta)d\eta$ is called a Fourier integral operator of order $\left(m+\frac{N}{2}-\frac{n_1+n_2}{4}\right)$. For $U\subset\R^n$ an open set and $E$ be a vector bundle over $U$, by $P\in L^m_{\rho,\delta}\left(U;E\right)$ we mean a pseudodifferential operator $P$ of order $m$ of type $(\rho,\delta)$ sending sections of $E$ to itself, where $\rho+\delta=1$. This means that the operator $P$ is given by the oscillatory integral  $P(x,y):=(2\pi)^{-n}\int_{\R^n}e^{i\langle x-y,\eta\rangle}p(x,y,\eta)d\eta$, where $p(x,y,\eta)\in S^m_{\rho,\delta}(U\times U\times\R^n,\End E)$. One can check that $F:\mathscr E'(U,E)\to\cC(U,E)$ is continuous if and only if $F\in L^{-\infty}(U;E)$, and from now on we also use the notation $L^{-\infty}(U;E)$ for the space of smoothing operator on $U$ acting on sections of $E$. When $P$ is properly supported and the type of $P$ satisfies $\rho>\delta$, by the asymptotic expansion of the complete symbol $\sigma_P(x,\eta)$ we define the principal symbol $p_0(x,\eta)$ by the image of $\sigma_P(x,\eta)$ in the quotient $S^m_{\rho,\delta}/S^{m-(\rho-\delta)}_{\rho,\delta}$. In fact, $p_0(x,\eta)\in\cC(T^*U)$. We denote $L^m_{\rm cl}(U,E)\subset L^m_{1,0}(U,E)$ to be the space of classical pseudodifferential operators, where for $P\in L^m_{\rm cl}(U,E)$ we have $\sigma_P(x,\eta)\in S^m_{\rm cl}(U\times\R^n,E)$ and we may assume that $p_0(x,\eta)$ satisfies $p_0(x,\lambda\eta)=\lambda^m p_0(x,\eta)$ for all $\lambda\geq 1$ in this situation. We will use the standard elliptic estimates and estimates on Sobolev spaces for such kind of pseudodifferential operators in this paper. For pseudodifferential operators of the type $\rho=\delta=\frac{1}{2}$, we will apply the classical theory of Calderon and Vaillancourt for the estimates on Sobolev spaces.

We let $W$ be an open set in $\mathbb{R}^{N}$ and $E=\C$ be a vector bundle over $W$. We consider the space $S^0_{\operatorname{loc}}(1;W,E)$ containing all $a(x,k)\in\cC(W,E)$ with
	real parameter $k$ such that for all multi-index $\alpha\in\mathbb{N}_0^N$, any
	cut-off function $\chi\in\cCc(W)$, we have $\sup_{\substack{k\geq 1}}\sup_{x\in W}|\partial^\alpha_x(\chi(x)a(x,k))|<+\infty$. For general $m\in\mathbb{R}$, we can also consider
	\begin{equation}
	S^m_{\operatorname{loc}}(1;W,E):=\{a(x,k):k^{-m}a(x,k)\in\ S^0_{\operatorname{loc}}(1;W,E)\}.  
	\end{equation}
	For a sequence of $a_j\in S^{m_j}_{\operatorname{loc}}(1;W,E)$ with $m_j\searrow-\infty$ and $a\in S^{m_0}_{\operatorname{loc}}(1;W,E)$, we denote 
	\begin{equation}
	a(x,k)\sim\sum_{j=0}^{+\infty} a_j(x,k)~\text{in}~S^{m_0}_{\operatorname{loc}}(1;W,E)
	\end{equation}
	if for all $\ell\in\mathbb{N}$ we have $a-\sum_{j=0}^{\ell-1} a_j\in S^{m_\ell}_{\operatorname{loc}}(1;W,E)$.	In fact, for all sequence $a_j$ above, there always exists an element $a$ as the asymptotic sum, which is unique up to the elements in $S^{-\infty}_{\operatorname{loc}}(1;W,E):=\bigcap_{m} S^m_{\operatorname{loc}}(1;W,E)$. The above notations can be generalized to any smooth vector bundle $E$.

We recall the concept of $k$-negligible operators.  We let $W_1, W_2$ be bounded open subsets of $\R^{n_1}$ and $\R^{n_2}$, respectively. Let $E$ and $F$ be smooth complex vector bundles over $W_1$ and $W_2$, respectively. Let $s_1, s_2\in\R$ and $n_0\in\Z$. We say a kernel $F_k(x,y)$ is $k$-negligible and write
 \begin{equation}
     F_k(x,y)=O(k^{-\infty})~\text{on}~W_1\times W_2
 \end{equation}
 or just $F_k=O(k^{-\infty})~\text{on}~W_1\times W_2$ if for all $k>0$ large enough, $F_k$ is a smoothing operator,
		and for any compact set $K$ in $W_1\times W_2$, for all multi-index
		$\alpha\in\N_0^{n_1}$, $\beta\in\N_0^{n_2}$ and $N\in\mathbb{N}_0$, 
		there exists a constant $C_{K,\alpha,\beta,N}>0$ such that 
        \begin{equation}
            \left|\partial^\alpha_x\partial^\beta_y F_k(x,y)\right|\leq 
		C_{K,\alpha,\beta,N} k^{-N}:\ \ x,y\in K.
        \end{equation}
         For $k$-dependent operators $F_k$ and $G_k$, sometimes we also write 
  \begin{equation}
      \mbox{$F_k=G_k$\ \ on $W_1\times W_2$}
  \end{equation}
if $F_k-G_k=\OK$ on $W_1\times W_2$.
  
All the notations introduced above can be generalized to the case on smooth manifolds.
\subsection{Non-degenerate Cauchy--Riemann manifolds}
 We let $X$ be a connected, smooth and  orientable manifold of real dimension $2n+1,~n\geq 1$. We say a pair $(X,T^{1,0}X)$ is a hypersurface type Cauchy--Riemann manifold if there is a subbundle
	$T^{1,0}X\subset\mathbb{C}TX$ so that 
 \begin{equation}
\dim_{\mathbb{C}}T^{1,0}_{p}X=n,\ \ T^{1,0}_p X\cap T^{0,1}_p X=\{0\},\ \ [V_1,V_2]\in\mathscr{C}^{\infty}(X,T^{1,0}X),
\end{equation}
 where $p\in X$ is arbitrary, $T^{0,1}_p X:=\overline{T^{1,0}_p X}$,  $V_1, V_2\in \mathscr{C}^{\infty}(X,T^{1,0}X)$ are also arbitrary and
		$[\cdot,\cdot]$ stands for the Lie bracket between vector fields. We will use the phrase \textit{CR manifold} in this work to abbreviate the hypersurface type Cauchy--Riemann manifold. For the above subbundle $T^{1,0}X$, we call it a \textit{CR structure} of the CR manifold $X$. 
  
  From now on, we always discuss on a CR manifold $(X,T^{1,0}X)$ of real dimension $2n+1$, $n\geq 1$. We denote by $T^{*1,0}X$ and $T^{*0,1}X$ the dual bundles of
$T^{1,0}X$ and $T^{0,1}X$, respectively. We define the vector bundle of $(0,q)$-forms by $T^{*0,q}X := \Lambda^q\,T^{*0,1}X$. The Levi distribution $HX$ of the CR manifold $X$ is the real part of 
$T^{1,0}X \oplus T^{0,1}X$ as the unique sub-bundle of $TX$ such that
\begin{equation}
    \C HX=T^{1,0}X \oplus T^{0,1}X.
\end{equation}
We let $J:HX\To HX$ be the complex structure given by 
$J(u+\ol u)=iu-i\ol u$, for every $u\in T^{1,0}X$. 
If we extend $J$ complex linearly to $\C HX$ we have $T^{1,0}X \, = \, \left\{ V \in \C HX \,;\, \, JV \, 
=  \,  iV  \right\}$. Given any auxiliary Riemannian metric on $X$, we have $TX=HX\bigoplus (HX)^\perp$ with respect to $g$. Since $TX$ and $HX$ are orientable,  we have some real and non-vanishing $v\in\cC(X,(HX)^\perp)$ and we can define a real and non-vanishing 1-form $\bm\alpha\in\cC(X,T^*X)$ by $\bm\alpha(u):=g(u,v)$ for all $v\in\cC(X,TX)$. It is clear that 
\begin{equation}
    HX=\ker\bm\alpha.
\end{equation}
 We call such $\bm\alpha$ a \textit{characteristic $1$-form}. It turns out that the restriction of $d\bm\alpha$ on $HX$ is a $(1,1)$-form and we have a symmetric bilinear map $\mathcal{L}_x:H_x X\times H_x X\to\R$,  $\mathcal{L}_x(u,v)
=\frac12d\bm\alpha(u,Jv)$, for all $u, v\in H_x X$. It induces a Hermitian form also denoted by $\mathcal{L}_x$ and called \textit{Levi form} by
\begin{equation}\label{eq:2.12b}
\mathcal{L}_x:T^{1,0}_xX\times T^{1,0}_xX\to\C,
\:\: \mathcal{L}_x(U,V)=\frac{i}{2}d\bm\alpha(U, \ol V) ,
\:\:U, V\in T^{1,0}_xX.
\end{equation}

 A CR manifold $X$ is said to be \textit{non-degenerate} if 
for every $x\in X$ the Levi form
$\mathcal L_x$ is a non-degenerate Hermitian form. One can check that this definition does not depend 
on the choice of the characteristic $1$-form $\bm\alpha$. If $X$ is non-degenerate, one can also check that $\bm\alpha$
is a contact form and the Levi distribution $HX$
is a contact structure. 

Locally, there exists an orthonormal basis $\{\mathcal{Z}_1,\cdots,\mathcal{Z}_n\}$
	of $T^{1,0}X$ with respect to the Hermitian metric $\langle\,\cdot\,|\,\cdot\,\rangle$ such that $\mathcal{L}_p$ is 
	diagonal in this basis, $\mathcal{L}_p(\mathcal{Z}_j,\ol{\mathcal{Z}}_\ell)=\delta_{j,\ell}\mu_j(p)$.
	The entries $\mu_1(p) \cdots, \mu_n(p)$ are called the \textit{eigenvalues of the Levi form}
	at $p\in X$ with respect to $\langle\,\cdot\,|\,\cdot\,\rangle$. We notice that the sign of the eigenvalues does not depend on the choice of the metric 
	$\langle\,\cdot\,|\,\cdot\,\rangle$. From now on, we use $n_-$ to denote the number of negative eigenvalues and $n_+$ for the number of positive eigenvalues of the Levi form on $X$, respectively. In our context, $n_-+n_+=n$ and the pair $(n_-,n_+)$ is called the signature (of the Levi form) of the CR manifold $(X,T^{1,0}X)$, which is also independent of the choice of Hermitian metric on $\C TX$. A strongly pseudoconvex CR manifold of real dimension $2n+1$ has a constant signature $(n_-,n_+)=(0,n)$. It is known that for each $j\in\{1,\cdots,n\}$ the function $p\mapsto \mu_j(p)$ is a continuous function on $X$, so by intermediate value theorem we know that the Levi form on a non-degenerate CR manifold must have the constant signature. 
 
From now on, we always assume our CR manifold $(X,T^{1,0}X)$ is non-degenerate with respect to some characteristic form $\bm\alpha$. Then $d\bm\alpha$ is non-degenerate on $HX$, and since $\dim_\R X$ is odd, by linear algebra we know that the set $W_x:=\{w\in T_x X:d\bm\alpha(w,v)=0,\ \ \forall v\in T_x X\}$ satisfies $\dim_\R W_x=1$. Again by the non-degenerate assumption, we have $T_x X=H_x X\bigoplus W_x$, and we can find a real and non-vanishing  $T\in \cC(X, TX)$ such that 
\begin{equation}
 \iota_T\,d\bm\alpha=0,\ \ \iota_T\,\bm\alpha=-1.  
\end{equation}
Such vector field is called the \textit{Reeb vector field} associated by $\bm\alpha$. From now on, we also let $\langle\cdot|\cdot\rangle$ be a Hermitian metric on $\C TX$
such that the decomposition $\C TX = T^{1,0}X \oplus T^{0,1}X \oplus \C T$ is orthogonal.
	\subsection{Szeg\H{o} projections for lower energy forms}
	We recall some essential material about microlocal analysis on Cauchy--Riemann manifolds. By exterior algebra, the Hermitian metric $\langle\cdot|\cdot\rangle$ on $\C TX$ induces a Hermitian metric on $\Lambda^r\mathbb{C}T^*X$. We take the corresponding orthogonal projection  $\pi^{0,q}:\Lambda^q\mathbb{C}T^*X\to T^{*0,q}X:=\Lambda^q(T^{*0,1}X)$. The \textit{tangential Cauchy--Riemann operator} is defined by $\overline{\partial}_b:=\pi^{0,q+1}\circ d:\mathscr{C}^{\infty}(X,T^{*0,q}X)\to\mathscr{C}^{\infty}(X,T^{*0,q+1}X)$. By exterior algebra, we can check that $\overline{\partial}_b^2=0$. We take the $L^2$-inner product $(\cdot|\cdot)$ on $\mathscr{C}^\infty(X,T^{*0,q}X)$ induced by $\langle\cdot|\cdot\rangle$ via $(f|g):=\int_X\langle f|g\rangle dm(x)$,
	where $f,g\in\mathscr{C}^\infty(X,T^{*0,q}X)$, $dm(x):=m(x)dx$	is the given volume form on $X$. We also recall that there is another volume form $dv(x):=v(x)dx$	induced by the Hermitian metric compatible with $\bm\alpha$ such that $v(x):=\sqrt{\det g}$, $g:=(g_{jk})_{j,k=1}^{2n+1}$, and $g_{jk}:=\langle\frac{\partial}{\partial x_j}|
	\frac{\partial}{\partial x_k}\rangle$. We let $L^2_{0,q}(X):=L^2(X,T^{*0,q}X)$ be the completion of the space $\Omega^{0,q}(X):=\mathscr{C}^\infty(X,T^{*0,q}X)$ with respect to $(\cdot|\cdot)$. We extend the closed and densely-defined operator
	$\bar\partial_{b}$ to $L^2_{0,q}(X)$, $q\in\{0,1,\cdots,n\}$, in the sense of current, and we denote by $\bar\partial_{b,H}^*$ the Hilbert space adjoint of $\bar\partial_{b}$ with respect to $(\cdot|\cdot)$. We let $\Box^{(q)}_{b}$ denote the \textit{Kohn Laplacian} (extended by Gaffney extension) such that $ {\rm Dom\,}\Box^{(q)}_{b}=\{s\in{\rm Dom\,}\bar\partial_{b}\cap{\rm Dom\,}\bar\partial_{b,H}^*:\,
	\bar\partial_{b}s\in{\rm Dom\,}\bar\partial_{b,H}^*,\ \bar\partial_{b,H}^*s\in{\rm Dom\,}\bar\partial_{b}\}$ and $\Box^{(q)}_{b}s=\bar\partial_{b}\bar\partial_{b,H}^*s+\bar\partial_{b,H}^*\bar\partial_{b}s$ for $s\in {\rm Dom\,}\Box^{(q)}_{b}$. For every $q\in\{0,1,\cdots,n\}$, $\Box^{(q)}_{b}$ is a positive self-adjoint operator. We refer this fact to the functional analysis argument \cite{MM07}*{Proposition 3.1.2}. We also notice that $\Box_b^{(q)}$ is never an elliptic differential operator for its principal symbol vanishes on the set $\Sigma$, where $\Sigma:=\Sigma^-\cup\Sigma^+$, $\Sigma^\mp:=\set{(x,\eta)\in T^*X:\sum\eta_j(x)dx_j=c\bm\alpha(x),~c\lessgtr 0}$. From now on, we also assume $X$ is compact. In our context, when $q\notin\{n_-,n_+\}$, $\Box^{(q)}_b$ is hypoelliptic with loss of one derivative and has $L^2$-closed range \cite{Hs10}*{Part I, \S 6}. For the concerning  result in a more general set-up called $Y(q)$ condition, we consult to the \cite{ChSh01}. When $q\in\{n_-,n_+\}$, $\Box^{(q)}_{b}$ may not even be hypoelliptic, i.e, $\Box^{(q)}_{b}u\in\mathscr{C}^\infty(X,T^{*0,q}X)$ may not imply that 
	$u\in\mathscr{C}^\infty(X,T^{*0,q}X)$. When $q\in\{n_-,n_+\}$ and $\Box_b^{(q)}$ has $L^2$-closed range in $L^2_{0,q}(X)$, it is proved in \cite{Hs10}*{Part I, Theorem 1.2} that the Szeg\H{o} projection $\Pi^{(q)}$ on $(0,q)$-forms, which is the orthogonal projection $\Pi^{(q)}:L^2_{0,q}(X)\to\ker\Box_b^{(q)}$,
	is the sum of two Fourier integral operators of order zero with complex-valued phase functions. Moreover, the Schwartz kernel $\Pi^{(q)}(x,y)\in\mathscr D'(X\times X,\End(T^{*0,q}X))$ called Szeg\H{o} kernel has the singularities described by H\"ormander's wavefront set ${\rm WF}(\Pi^{(q)})=\{(x,\eta,x,-\eta):(x,\eta)\in\widehat\Sigma\}$, where $\widehat{\Sigma}:=\Sigma^\mp$ when $q=n_\mp$ and $n_-\neq n_+$, and $\widehat{\Sigma}:=\Sigma$ when $q=n_-=n_+$. This kind of microlocal analysis for Szeg\H{o} projecition was first introduced in Boutet de Monvel--Sj\"ostrand \cite{BouSj75} when $(n_-,n_+)=(0,n)$. Hsiao \cite{Hs10} uses a different approach than \cite{BouSj75}
by developing the microlocal heat equation method (see also \cite{MenSj78}) together with Witten's trick (see also \cite{BerSj07}). We follow the formulation of Hsiao and especially the statement of Hsiao--Marinescu \cite{HM17JDG}*{\S 4}. The following theorem is the core ingredient of our paper.
\begin{theorem}
		\label{thm:Hsiao-Marinescu 17 LEF}
		We let $(Y,T^{1,0}Y)$ be a CR manifold with real dimension $2n+1$, $n\geq 1$, and assume that the Levi form of $Y$ is of constant signature $(n_-,n_+)$ on a relatively compact set $\Omega\Subset Y$ with respect to some characteristic form $\bm\alpha$. We take the orthogonal projection $\Pi^{(q)}_\lambda:L^2_{0,q}(X)\to E([0,\lambda])$, where $E([0,\lambda]):=\range~\mathds 1_{[0,\lambda]}(\Box_b^{(q)})$ is the image of the spectral projection of the self-adjoint and positive operator $\Box_b^{(q)}$. Then when $q=n_-$, there are properly supported operators  $S_-,S_+\in L^{0}_{\frac{1}{2},\frac{1}{2}}(\Omega;T^{*0,q}Y)$ given by the oscillatory integrals $S_-(x,y),S_+(x,y)$ in \eqref{eq:approximated Szego kernel intro} such that
  \begin{equation}
    \mbox{$\Pi^{(q)}_\lambda\equiv S_-+S_+$\ \ on $\Omega$}.
  \end{equation}
When $q\notin\{n_-,n_+\}$, we have $\Pi_\lambda^{(q)}\equiv$ 0 on $\Omega$.
	\end{theorem}
In \eqref{eq:approximated Szego kernel intro}, when $q=0$, for any $m\in\mathbb R$ the symbol space $S^m_{1,0}\left(\Omega\times \Omega\times\mathbb{R}_+,\End(T^{*0,q}X)\right)$ collects all $a(x,y,t)$ in $\cC(\Omega\times \Omega\times\mathbb{R}_+,\End(T^{*0,q}X))$ 
		such that for all compact sets $K\Subset \Omega\times \Omega$, all $\alpha, 
		\beta\in\N^{2n+1}_0$ and $\gamma\in\N_0$, 
		there is a constant $C_{K,\alpha,\beta,\gamma}>0$ satisfying the estimate 
        \begin{equation}
            \left|\partial^\alpha_x\partial^\beta_y\partial^\gamma_t a(x,y,t)\right|\leq 
		C_{K,\alpha,\beta,\gamma}(1+t)^{m-|\gamma|}:\ \ (x,y,t)\in K\times\mathbb R_+,\ \ t\geq 1.
        \end{equation}
        We let $S^{-\infty}(\Omega\times \Omega\times\mathbb{R}_+,\End(T^{*0,q}X))
		:=\bigcap_{m\in\mathbb R}S^m_{1,0}(\Omega\times \Omega\times\mathbb{R}_+,\End(T^{*0,q}X))$. We define the classical symbol space $S^m_{\rm cl}(\Omega\times\Omega\times\R_+,\End(T^{*0,q}X))$  by collecting all the elements of $a(x,y,t)\in\cC(\Omega\times\Omega\times\R_+,\End(T^{*0,q}X))$ with the asymptotic expansion
  \begin{equation}
      a(x,y,t)\sim\sum_{j=0}^{+\infty}a_j(x,y)t^{m-j}\ \ \text{in}~S^m_{1,0}(\Omega\times\Omega\times\R_+,\End(T^{*0,q}X)),
  \end{equation}
where $a_j(x,y)\in\cC(\Omega\times\Omega,\End(T^{*0,q}X))$ for all $j\in\N_0$.  The above definition can be naturally generalized to the case $q\geq 0$. 

We can check that, for Szeg\H{o} phase functions $\varphi_\mp(x,y)$, the complex-valued phase functions $\varphi_\mp(x,y)t$ are regular, and for any $m\in\R$ and the symbol $a(x,y,t)\in S^m_{\rm cl}(\Omega\times \Omega\times\mathbb{R}_+,\End(T^{*0,q}X))$ we also have $ \int_0^{+\infty}e^{it\varphi_\mp(x,y)}a(x,y,t)dt
=\lim_{\epsilon\to 0}\int_0^{+\infty}e^{it(\varphi_\mp(x,y)+i\epsilon)}a(x,y,t)dt$ as the regularization of the oscillatory integral.
	
By \cite{Koh86}, $\Box_b^{(q)}$ automatically has $L^2$-closed range when $|n_- -n_+|\neq 1$ but this is not necessarily true for $|n_- - n_+| = 1$, cf.~the three dimensional counter example \cite{Bou75} and \cite{ChSh01}*{\S 12}. We also recall that when $\Box_b^{(q)}$ has $L^2$-closed range, from the spectral theory for self-adjoint operators \cite{Dav95} and the result \cite{HM17JDG}*{Theorem 1.7} there is some $\lambda>0$ such that $\Pi^{(q)}=\Pi_\lambda^{(q)}$.

By the finite partition of unity of the compact manifold $X$ and a slight modification of the proof of \cite{HM17JDG}*{Theorems 4.6 and 4.7}, we have the following statement which can help us to localize the calculation later.
	\begin{theorem}
		\label{thm:WF of Szego projection on lower energy forms}
		With the notations and assumptions in Theorem \ref{thm:Main Semi-Classical Expansion}, for $q=n_-$ we have ${\rm WF}(\Pi^{(q)}_\lambda)=\{(x,\eta,x,-\eta):(x,\eta)\in\widehat\Sigma\}$. In particular, we get
  \begin{equation}
      \Pi_\lambda^{(q)}(x,y)\in\cC(X\times X\setminus{\rm diag}\left(X\times X),\End(T^{*0,q}X)\right).   \label{eq:Pi_lambda^q is smoothing away from diagonal}   
  \end{equation}
  For $q\notin\{n_-,n_+\}$, we have $\Pi_\lambda^{(q)}\equiv 0$ on $X$.
	\end{theorem}
	
	In the last, we recall some fine results of phase functions. The first is from \cite{HM17JDG}*{Theorem 3.4}.
	\begin{theorem}
		\label{thm:tangential Hessian of phase functions varphi_-}
		With the notations and assumptions of Theorem \ref{thm:Main Semi-Classical Expansion}, for a given point $x_0\in \Omega$, let $\{W_j\}_{j=1}^{n}$
		be an orthonormal frame with respect to $\langle\,\cdot\,|\,\cdot\,\rangle$ of $T^{1, 0}X$ in a neighborhood of $x_0$
		such that
		the Levi form is diagonal at $x_0$, i.e., $\mathcal{L}_{x_{0}}(W_{j},\overline{W}_{s})=\delta_{j,s}\mu_{j}$, $j,s=1,\cdots,n$.
		We can take local coordinates
		$x=(x_1,\cdots,x_{2n+1})$, $z_j=x_{2j-1}+ix_{2j}$, $j=1,\cdots,n$,
		defined on some neighborhood of $x_0$ such that  $x(x_0)=0$, ${\bm\alpha}(x_0)=dx_{2n+1}$, $T=-\pr_{x_{2n+1}}$, and 
  \begin{equation}
      W_j=\frac{\pr}{\pr z_j}-i\mu_j\ol z_j\frac{\pr}{\pr x_{2n+1}}-
		c_jx_{2n+1}\frac{\pr}{\pr x_{2n+1}}+O(\abs{x}^2),\ c_j\in\C,\ j=1,\cdots,n.
  \end{equation}
		We set
		$y=(y_1,\cdots,y_{2n+1})$, $w_j=y_{2j-1}+iy_{2j}$, $j=1,\cdots,n$, then under the above coordinates we have	in some neighborhood of $(0,0)$ that 
  \begin{multline}
      \varphi_-(x, y)
		=-x_{2n+1}+y_{2n+1}+i\sum^{n}_{j=1}\abs{\mu_j}\abs{z_j-w_j}^2+(x_{2n+1}-y_{2n+1})f(x, y)\\
		+\sum^{n}_{j=1}\Bigr(i\mu_j(\ol z_jw_j-z_j\ol w_j)+c_j(-z_jx_{2n+1}+w_jy_{2n+1})+\ol c_j(-\ol z_jx_{2n+1}+\ol w_jy_{2n+1})\Bigr)+O(\abs{(x, y)}^3),
  \end{multline}
  where $f$ is smooth and satisfies $f(0,0)=0$, $f(x, y)=\ol f(y, x)$. 
	\end{theorem}
	We remark that in the above theorem, when $q=n_-=n_+$ the function $\varphi_+(x,y)$ has the similar formula, cf.~\cite{Hs10}*{Theorem 1.4}. Also, at $x=x_0$ the volume form $v(x)dx$ induced by Hermitian metric satisfies $v(x_0)=2^n$ and compatible with the normalization of \eqref{eq:s_0^- intro} and \eqref{eq:s_0^+ intro}, cf.~also \cite{Hs10}*{pp.~76-77}. In this context we can check that $|\bm\alpha\wedge(d\bm\alpha)^n|=2^n n!|\det\mathcal{L}_x| v(x)dx$. Moreover, for such small enough coordinate patch, there exist a constant $C>0$ such that 
	\begin{equation}
	\label{eq:Im varphi>C|z-w|^2}
	{\rm Im\,}\varphi_-(x,y)\geq C\sum^{2n}_{j=1}\abs{x_j-y_j}^2,
	\end{equation}
 and so does $\varphi_+(x,y)$ when $q=n_-=n_+$. We refer to \cite{Hs10}*{Part I, Proposition 7.16} for a proof of 
	\eqref{eq:Im varphi>C|z-w|^2}.
	
With the theory of Melin--Sj\"ostrand \cite{MeSj74}*{\S 4} and the proof of \cite{HM17JDG}*{Theorem 5.4 \& \S 8}, we have the following theorem about equivalence class of Szeg\H{o} phase functions. 
	\begin{theorem}
		\label{thm:division theorem for equivalent Szego phase function}
		With the same notations and assumptions in Theorem \ref{thm:Hsiao-Marinescu 17 LEF}, for any $\psi_\mp(x,y)$ of $\cC(\Omega\times\Omega,\C)$ satisfying \eqref{eq:varphi intro}, we have  $s^{\psi_\mp}(x,y,t)\in S^n_{\rm cl}\left(\Omega\times \Omega\times\R_+, \End(T^{*0,q}X)\right)$ such that 
\begin{equation}
    S_\mp(x,y)\equiv\int_0^{+\infty} e^{it\psi_\mp(x,y)} s^{\psi_\mp}(x,y,t)dt
\end{equation}
		on $\Omega\times\Omega$. Moreover, when $S_\mp$ is not smoothing there are some $f_\mp(x,y)\in\cC(\Omega\times\Omega)$ satisfying $f_\mp(x,x)\neq 0$ such that 
  \begin{equation}
      \varphi_\mp(x,y)-f_\mp(x,y)\psi_\mp(x,y)=O(|x-y|^{+\infty}).
  \end{equation}
	\end{theorem}
 \begin{remark}
		\label{rmk:varphi(x,y)=x_2n+1+g(x',y)}
		We have some special choice of Szeg\H{o} phase function which can help us simplify the later calculation. For $\varphi_\mp$ of \eqref{eq:approximated Szego kernel intro} and a coordinate patch $(\Omega,x)$  described in Theorem \ref{thm:tangential Hessian of phase functions varphi_-}, by the Malgrange preparation theorem \cite{Hoe03}*{Theorem 7.5.5} and  Melin--Sj\"ostrand equivalence of phase functions \cite{MeSj74}*{Definition 4.1 \& Theorem 4.2}, we may assume that
		\begin{align}
		&\varphi_\mp(x,y)=\mp x_{2n+1}+g_\mp(x',y),\label{eq:varphi_mp(x,y)=x_2n+1+g(x',y)}
		\end{align}	
where $g_\mp(x',y)\in\cC(\Omega\times\Omega;\C)$, $\im g_\mp(x',y)\geq 0$ and $x'=(x_1,\cdots,x_{2n})$, so that 
\begin{equation}
S_\mp(x,y)\equiv\int_0^{+\infty}e^{it\varphi_\mp(x,y)}s^{\varphi_\mp}(x,y,t)dt, \ \ s^{\varphi_\mp}(x,y,t)\in S^n_{\rm cl}\left(\Omega\times \Omega\times\R_+,\End(T^{*0,q}X)\right).
\end{equation}
for some $s^{\varphi_\mp}(x,y,t)\in S^n_{\rm cl}\left(\Omega\times \Omega\times\R_+,\End(T^{*0,q}X)\right)$. If two Szeg\H{o} phase functions $\varphi_1^\mp$ and $\varphi_2^\mp$ 
		satisfy \eqref{eq:varphi_mp(x,y)=x_2n+1+g(x',y)}, we can apply the proof of \cite{HM17JDG}*{Theorem 5.4} and deduce that $\varphi_1^\mp-\varphi_2^\mp=O(|x-y|^{+\infty})$. 
		From now on, if we do not specify, we write $\varphi_\mp$ 
		to denote the Szeg\H{o} phase function $\varphi_\mp(x,y)$ satisfying \eqref{eq:varphi_mp(x,y)=x_2n+1+g(x',y)} up to an error of size $O(|x-y|^{+\infty})$. We also notice that $(x,y,t)=O(|x-y|^{+\infty})$  implies that $\int_0^{+\infty} e^{it\varphi_\mp(x,y)}r(x,y,t)dt\equiv 0$, and one can refer the proof to \cite{BouSj75}*{Proposition 1.11} for example.
\end{remark}
	\section{Toeplitz operators for lower energy forms}
The goal of this part is to study the Toeplitz operator $T_{P,\lambda}^{(q)}:=\Pi_\lambda^{(q)}\circ P\circ\Pi_\lambda^{(q)}$ associated by a formally self-adjoint $P\in L^1_\mathrm{cl}(X;T^{*0,q}X)$. We will first recall the notion of Fourier integral operators of Szeg\H{o} type and systematically establish the elementary spectrum results for $T_{P,\lambda}^{(q)}$.
\subsection{Fourier integral operators of Szeg\H{o} type}
From \cite{MeSj74}*{Definition 4.1 \& Theorem 4.2}, we also have the following more general class of equivalent Szeg\H{o} phase functions.
\begin{definition}
		\label{def:Szego Phase functions}
		With the same notations and assumptions in Theorem \ref{thm:Main Semi-Classical Expansion}, for $q=n_-$ and any $\Lambda\in\cC(X,\R_+)$, we let ${\rm Ph}(\mp\Lambda\bm\alpha,\Omega)$, respectively, be the set collecting all functions $\psi_
		\mp(x,y)\in\cC(\Omega\times \Omega)$ with the following effects:
		\begin{align}
		&\mbox{${\rm Im\,}\psi_
			\mp(x,y)\geq 0$},\ \ \mbox{$\psi_
			\mp(x,y)=0\iff y=x$},\ \ 
		\mbox{$d_x\psi_
			\mp(x,x)=-d_y\psi_
			\mp(x,x)=\mp\Lambda(x)\bm\alpha(x)$},
		\end{align}		
		and for $S_\mp$ of \eqref{eq:approximated Szego kernel intro} there is a symbol $s^{\psi_
			\mp}(x,y,t)\in S^n_{\rm cl}\left(\Omega\times \Omega\times\R_+,\End(T^{*0,q}X)\right)$
		with the property that
		\begin{equation}
		S_\mp(x,y)\equiv\int_0^{+\infty}e^{it\psi_
			\mp(x,y)}s^{\psi_
			\mp}(x,y,t)dt.
		\end{equation}
	\end{definition} 
 
We need the following analogue of approximated Szeg\H{o} kernels.
\begin{definition}
		\label{def:FIO Szego type}
		Following Theorem \ref{thm:Main Semi-Classical Expansion}, for $q=n_-$ we let $H:\cCc(\Omega,T^{*0,q}X)\to\cC(\Omega,T^{*0,q}X)$ be a continuous operator. For any $m\in\R$, we say that $H$ is a Fourier integral operator of Szeg\H{o} type of weight $m$ (or order $m-n$) if on $\Omega\times\Omega$ we have
		\begin{align}
		&H(x,y)\equiv H_-(x,y)+H_+(x,y),\ \ H_\mp(x,y)\equiv\int^{+\infty}_0 e^{it\varphi_\mp(x, y)} h^\mp(x, y, t)dt,
		\end{align}
		where $\varphi_\mp(x,y)\in{\rm Ph}(\mp\Lambda\bm\alpha,\Omega)$ for some $\Lambda\in\cC(X,\R_+)$ and we have the following data properly supported in $(x,y)$: $h^\mp(x,y,t)\sim\sum_{j=0}^{+\infty}h_j^\mp(x,y)t^{m-j}$ in $S^{m}_{1,0}\left(\Omega\times \Omega\times\R_+,\End(T^{*0,q}X)\right)$ and $h^+(x,y,t)=0$ if $n_-\neq n_+$. 
  
  We denote the space of Fourier integral operators of Szeg\H{o} type of weight $m$ by $I_\Sigma^m(\Omega;T^{*0,q}X)$ and $I_\Sigma(\Omega;T^{*0,q}X):=\bigcup_{m\in\R}I^m_\Sigma(\Omega;T^{*0,q}X)$. 
	\end{definition}

By the classical formula \cite{BouSj75}*{(1.6)}, the following formula for $s_0^{\psi_
		\mp}(x,x)$ for different choice of Szeg\H{o} phase functions is known, cf.~\cite{HHMS23}*{Theorem 2.13} for example.
	\begin{theorem}
		\label{thm:s_0^psi(x,x)}
		In the situation of Definition \ref{def:Szego Phase functions}, for any $\psi_\mp\in {\rm Ph}(\mp\Lambda\bm\alpha,\Omega)$ we have the transformation rule $s^{\psi_\mp}_0(x,x)=
		\Lambda(x)^{n+1}s_0^\mp(x,x)$.
	\end{theorem}

We will later frequently apply the following variant of \cite{GaHs23}*{Lemma 4.1}.
	\begin{theorem}
		\label{thm:division theorem of principal symbol for Szego FIO}   
		In the situation of Definition \ref{def:FIO Szego type}, we consider the operator $H\in I^m_\Sigma(\Omega;T^{*0,q}X)$ and we assume that
		\begin{align}
		& H\equiv (S_-+S_+)\circ H\equiv H\circ (S_-+S_+)\ \ \text{on $\Omega$},\label{eq:assumption microlocal self adjoint by S}\\
		&  \tau_x^{n_-}h_0^-(x,x)\tau_x^{n_-}=0,\ \ \text{$\forall x\in\Omega$},\label{eq:h_0 vanishes at I_0}\\
		&  \tau_x^{n_+} h_0^+(x,x)\tau_x^{n_+}=0\ \ \text{additionally when $n_-=n_+$,\ \ $\forall x\in\Omega$}.\label{eq:h_0 vanishes at J_0}
		\end{align}
		If we write $h_0^\mp(x,y)={\sum_{|{\bf I}|=|{\bf J}|=q}} h^\mp_{\bf I, J}(x,y) \omega_{\bf I}^\wedge(x)\otimes \omega_{\bf J}^{\wedge,*}(y)$ in the strictly increasing index sets, cf.~\eqref{eq:system of the principal symbol intro}, then we have some $\rho^\mp_{\bf I, J}(x,y)\in\cC(\Omega\times\Omega)$ such that
		\begin{equation}
		h^\mp_{\bf I, J}(x,y)-\rho^\mp_{\bf I, J}(x,y) \varphi_\mp(x,y)=O(|x-y|^{+\infty}).
		\end{equation}
	\end{theorem}
\subsection{Microlocal analysis of Toeplitz operators}
We have the following microloal structure of Toeplitz operators on lower energy forms, which can be deduced from  Melin--Sj\"ostrand complex stationary phase formula \cite{MeSj74}*{Theorem 2.3 \& p.~156}, and Theorems \ref{thm:Hsiao-Marinescu 17 LEF} and \ref{thm:WF of Szego projection on lower energy forms}. We also refer to \cite{GaHs23}*{Theorem 4.4} for the calculation.
	\begin{theorem}
		\label{thm:BdM-Sj for Toeplitz operators}
		With the same notations and assumptions of Theorem \ref{thm:Main Semi-Classical Expansion}, for $q=n_-$ the operator $T_{P,\lambda}^{(q)}$ is the sum of Fourier integral operators: on $\Omega$ we have $T_{P,\lambda}^{(q)}=T_{\varphi_-}+T_{\varphi_+}+F$, where $F:\mathscr E'(\Omega,T^{*0,q}X)\to \cC(X,T^{*0,q}X)$ is continuous and the Schwartz kernel
  \begin{equation}
      T_{\varphi_\mp}(x,y)=\int_0^{+\infty} e^{it\varphi_\mp(x,y)} t a^\mp(x,y,t)dt\label{eq:Toeplitz FIO}
  \end{equation}
is given by the Szeg\H{o} phase function $\varphi_\mp(x,y)$ in Remark \ref{rmk:varphi(x,y)=x_2n+1+g(x',y)} and the symbol 
\begin{equation}
  \text{$a^\mp(x,y,t)\sim\sum^{+\infty}_{j=0}a_j^\mp(x,y)t^{n-j}$
			in $S^{n}_{1,0}(\Omega\times \Omega\times\mathbb R_+,\End(T^{*0,q}X))$}.
\end{equation}
In fact, $a^\mp(x,y,t)$ is properly supported in the variables $(x,y)$ and we have $a_0^+(x,y,t)=0$ when $n_-\neq n_+$. Moreover, we have
		\begin{align}
		&{a_0^-(x,x)=\frac{\abs{\det\mathcal{L}_{x}}}{2\pi^{n+1}}\frac{v(x)}{m(x)}\tau_x^{n_-}p_0({-\bm\alpha}_x)\tau_x^{n _-}},
		\end{align}
and when $n_-=n_+$ we also have
\begin{equation}
    a_0^+(x,x)=\frac{\abs{\det\mathcal{L}_{x}}}{2\pi^{n+1}}\frac{v(x)}{m(x)}\tau_x^{n_+}p_0({\bm\alpha}_x)\tau_x^{n _+}.
\end{equation}
	\end{theorem}
	We notice that the continuity properties of $P$ and $\Pi_\lambda^{(q)}$ on Sobolev spaces imply that
	$T_{P,\lambda}^{(q)}$ is also a bounded operator between the Sobolev spaces
	$H^{s+1}_{0,q}(X)$ and $H^s_{0,q}(X)$, where $H^s_{0,q}(X):=H^s(X,T^{*0,q}X)$, for all $s\in\mathbb{R}$. We denote this fact by 
	\begin{equation}
	\label{eq:Sobolev boundedness of T_P}
T_{P,\lambda}^{(q)}=O(1)~\text{in}~\mathscr{L}(H^{s+1}_{0,q}(X),H^s_{0,q}(X)),~\forall~s\in\R.
	\end{equation} 

 	Moreover, we still have parametrix type theorem for Levi-elliptic Toeplitz operators although they are not necessarily defined by elliptic pseudodifferential operators.
	\begin{theorem}
		\label{thm:parametetrix of Toeplitz operators for lower energy forms}
		For the same notations and assumptions in Theorem \ref{thm:Main Semi-Classical Expansion} and $q=n_-$, we can always find a formally self-adjoint $Q\in L^{-1}_{\rm cl}(X;T^{*0,q}X)$ such that
		\begin{equation}
		T_{Q,\lambda}^{(q)}\circ T_{P,\lambda}^{(q)}\equiv T_{P,\lambda}^{(q)}\circ T_{Q,\lambda}^{(q)}\equiv \Pi_\lambda^{(q)}~\text{on}~X.
		\end{equation}
	\end{theorem}
	\begin{proof}
First of all, we notice that if we can find some $Q\in L^{-1}_{\rm cl}(X;T^{*0,q}X)$ such that $T_{Q,\lambda}^{(q)}\circ T_{P,\lambda}^{(q)}\equiv T_{P,\lambda}^{(q)}\circ T_{Q,\lambda}^{(q)}\equiv \Pi_\lambda^{(q)}$, then we can replace $Q$ be $\frac{1}{2}(Q+Q^*)$ which is clearly formally self-adjoint.

For the generality of our argument, we demonstrate the case $n_-=n_+$, and the case $n_-\neq n_+$ follows from the same argument with some minor change. In the following we always use the convention \eqref{eq:system of the principal symbol intro}. When $q=n_-$, we have $p_{I_0,I_0}(-\bm\alpha_x)>0$ and we can find a conic neighborhood $\mathcal{C}_1^-$ of $\Sigma^-$ such that $p_{I_0,I_0}(-\bm\alpha_x)>$ on the closure of $\mathcal{C}_1^-$. We take a function $\rho(x,\eta)\in\cC(T^*X)$ that $\rho$ vanishes for small $|\eta|$, $\rho$ is positively homogeneous in $\eta$ of degree zero when $|\eta|\geq 1$, $\rho$ equals to one when $(x,\eta)$ in a conic neighborhood of $\Sigma^-$ and $\rho$ has support in the closure of $\mathcal{C}_1^-$. Then for any $r(x,\eta)\in S^{-1}_{\rm cl}(T^*X\setminus \mathcal C_1^-)$, 
		\begin{align}
		&\ell_{I_0,I_0}:=\rho p_{I_0,I_0}+(1-\rho)r\implies\ell_{I_0,I_0}\in S^{-1}_{\rm cl}(T^*X),\ \ \ell_{I_0,I_0}(-\bm\alpha_x)p_{I_0,I_0}(-\bm\alpha_x)= 1.\label{eq:q_I_0 II}
		\end{align}
		From the similar construction, we have
		\begin{align}
		&\ell_{J_0,J_0}\in S^{-1}_{\rm cl}(T^*X),\ \ \ell_{J_0,J_0}(\bm\alpha_x)p_{J_0,J_0}(\bm\alpha_x)=1.\label{eq:q_J_0 II}
		\end{align}
From the above arguments, we can find $L^{(0)}\in L^{-1}_{\rm cl}(X;T^{*0,q}X)$ with the principal symbol $\ell^{(0)}_0=\sum_{\bf I,J}'\ell_{\bf I,J}\omega_{\bf I}^\wedge\otimes \omega_{\bf J}^{\wedge,*}\in S^{-1}_{\rm cl}(T^*X,\End(T^{*0,q}X))$ such that
  \eqref{eq:q_I_0 II} and \eqref{eq:q_J_0 II} holds. Then for any coordinate patch $\Omega\subset X$, by combining Theorem \ref{thm:BdM-Sj for Toeplitz operators}, Melin--Sj\"ostrand stationary phase theorem, Theorem \ref{thm:tangential Hessian of phase functions varphi_-}, \eqref{eq:s_0^- intro}, \eqref{eq:s_0^+ intro} and Theorem \ref{thm:WF of Szego projection on lower energy forms} we can check that 
  \begin{equation}
      T_{L^{(0)},\lambda}^{(q)}\circ T_{P,\lambda}^{(q)}=I_0^- + I_0^+ + R_0\ \ \text{ on $\Omega$},
  \end{equation}
  where $R_0:\mathscr E'(\Omega,T^{*0,q}X)\to \cC(X,T^{*0,q}X)$ is continuous and $I_0^\mp\in I^n_{\Sigma}(\Omega;T^{*0,q}X)$. In fact, we have
  \begin{equation}
I_0^\mp(x,y)=\int_0^{+\infty}e^{it\varphi_\mp(x,y)}\mathscr S^\mp(x,y,t)dt,
  \end{equation}
  where  $\varphi_\mp(x,y)\in{\rm Ph}(\mp\bm\alpha,\Omega)$, $\mathscr S^\mp(x,y,t)\sim\sum_{j=0}^{+\infty}\mathscr S_j(x,y)t^{n-j}$ in $S^n_{1,0}\left(\Omega\times\Omega\times\R_+,\End(T^{*0,q}X)\right)$. If we write $\mathscr S_0^\mp(x,y)=\sum_{|{\bf I}|=|{\bf J}|=q}' \mathscr S_{\bf I, J}^\mp(x,y) \omega_{\bf I}^\wedge\otimes \omega_{\bf J}^{\wedge,*}$, then we also have
  \begin{equation}
  \mathscr S_{I_0,I_0}^-(x,x)=\mathscr S_{J_0,J_0}^+(x,x)=\frac{\abs{\det\mathcal{L}_{x}}}{2\pi^{n+1}}\frac{v(x)}{m(x)}.  
  \end{equation}
  By Theorem \ref{thm:division theorem of principal symbol for Szego FIO} and the above relations, we can deduce that $T_{L^{(0)},\lambda}^{(q)}\circ T_{P,\lambda}^{(q)}-\Pi_\lambda^{(q)}=H_0+G_0$ on $\Omega$, where $H_0\in I_\Sigma^{n-1}(\Omega;T^{*0,q}X)$ and $G_0:\mathscr E'(\Omega,T^{*0,q}X)\to\cC(X,T^{*0,q}X)$ is continuous.	Then for any $N\in\N$ and $j=1,\cdots,N-1$, with the same method we can construct $L^{(j)}\in L^{-1-j}_{\rm cl}(X;T^{*0,q}X)$ such that
		\begin{equation}
		\sum_{j=0}^{N-1} T^{(q)}_{L^{(j)},\lambda}\circ T_{P,\lambda}^{(q)}-\Pi_\lambda^{(q)}=H_N+G_N\ \ \text{on}~\Omega,
		\end{equation}
  where $H_N\in I_\Sigma^{n-N}(\Omega;T^{*0,q}X)$ and $G_N:\mathscr E'(\Omega,T^{*0,q}X)\to\cC(X,T^{*0,q}X)$ is continuous.
		We can then construct the symbol $\ell\in S^{-1}_{\rm cl}(T^*X,\End(T^{*0,q}X))$ from the asymptotic sums of the complete symbol of $L^{(j)}$, $j=0,1,\cdots$, and we can define $L\in L^{-1}_{\rm cl}(X;T^{*0,q}X)$ by the symbol $\ell$. Since the above argument holds for arbitrary $\Omega$ and $X$ is compact, combining with Theorem \ref{thm:WF of Szego projection on lower energy forms} we can check that $T^{(q)}_{L,\lambda}\circ T_{P,\lambda}^{(q)}-\Pi_\lambda^{(q)}\equiv 0~\text{on}~X$. By the same method above, we also have an $R\in L^{-1}_{\rm cl}(X;T^{*0,q}X)$ such that $T_{P,\lambda}^{(q)}\circ T_{R,\lambda}^{(q)}-\Pi_\lambda^{(q)}\equiv 0\ \ \text{on}~X$. Then we have that $T_{L,\lambda}^{(q)}= T_{L,\lambda}^{(q)}\circ \Pi_\lambda^{(q)}\equiv T_{L,\lambda}^{(q)}\circ\left( T_{P,\lambda}^{(q)}\circ T_{R,\lambda}^{(q)}\right)
		\equiv\left(T_{L,\lambda}^{(q)}\circ T_{P,\lambda}^{(q)}\right)\circ T_{R,\lambda}^{(q)}\equiv \Pi_\lambda^{(q)}\circ T_{R,\lambda}^{(q)}=T_{R,\lambda}^{(q)}$
	on $X$ and we conclude our theorem.
	\end{proof}
	We have the following type of elliptic estimates which now easily follows from Theorem \ref{thm:parametetrix of Toeplitz operators for lower energy forms}. For convenience, we denote $\mathcal H_{b,\lambda}^{(q)}(X):=\ker(I-\Pi^{(q)}_\lambda)$.
	\begin{theorem}
		\label{thm:elliptic estimate for T_P,lambda^q,q=n_-}
		Following Theorem \ref{thm:Main Semi-Classical Expansion}, for $q=n_-$ and every $s\in\R$ we have a constant $C_s>0$ such that $\|u\|_{s+1}\leq C_s\left(\|T_{P,\lambda}^{(q)} u\|_{s}+\|u\|_{s}\right)$, $\forall u\in\mathcal H_{b,\lambda}^{(q)}(X).$ In particular, given a non-zero eigenvalue $\mu\in\dot\R$ 
		of $T_{P,\lambda}^{(q)}$ and $s\in\N_0$, we have a constant
		$c_s>0$ such that 
		\begin{equation}
		\label{eq:eigenform estimates}
		\|u\|_{s}\leq c_s (1+|\mu|)^s\|u\|,~\forall u\in{\rm Ker\,}(T_{P,\lambda}^{(q)}-\mu I).
		\end{equation} 
  In other words, $\Ker(T_{P,\lambda}^{(q)}-\mu I)\subset\Omega^{0,q}(X)$ for all $\mu\in\dot\R$. Moreover,~\eqref{eq:eigenform estimates} 
		holds with \(\mu=0\) for all \(u\in \ker T_{P,\lambda}^{(q)}\cap\mathcal H_{b,\lambda}^{(q)}(X)\).
	\end{theorem}

We immediately have the following self-adjoint extension of $T_{P,\lambda}^{(q)}$ in $L^2_{0,q}(X):=L^2(X,T^{*0,q}X)$.
	\begin{theorem}
		\label{thm:T_P is self-adjoint}
		In the context of Theorem \ref{thm:Main Semi-Classical Expansion}, the maximal extension $T_{P,\lambda}^{(q)}:\Dom T_{P,\lambda}^{(q)}\subset L^2_{0,q}(X)\to L^2_{0,q}(X)$, $\Dom T_{P,\lambda}^{(q)}:=\left\{u\in L^2_{0,q}(X):~T_{P,\lambda}^{(q)} u\in L^2_{0,q}(X)\right\}$, is a self-adjoint extension of $T_{P,\lambda}^{(q)}$.	In particular, we have $\Spec T_{P,\lambda}^{(q)}\subset\R$.
	\end{theorem} 
\begin{proof}
    For $q=n_-$, this statement follows from Theorem \ref{thm:elliptic estimate for T_P,lambda^q,q=n_-} and the standard argument of self-adjoint $L^2$-extension of   
    elliptic and formally self-adjoint operators. For $q\notin\{n_-,n_+\}$, since $\Pi_\lambda^{(q)}\equiv 0$ on the compact manifold $X$, it is a compact operator and we can verify this statement by standard spectral theory. We notice that in both cases our argument relies on the compactness of $X$.
\end{proof}

 We have the following analogue of \cite{BG81}*{Proposition 2.14}. The proof can be deduced from standard technique of elliptic estimate and Rellich compact embedding lemma.
	\begin{theorem}
		\label{thm:Spec(T_P,lambda^q, q=n_-)}
		With the same notations and assumptions in Theorem \ref{thm:Main Semi-Classical Expansion}, for $q=n_-$, the set $\operatorname{Spec}(T_{P,\lambda}^{(q)})$ consists only by eigenvalues, where the non-zero eigenvalues all have finite multiplicity. For
   any $c>0$, the set ${\rm Spec\,}T^{(q)}_{P,\lambda}\cap[c,+\infty)\cap (-\infty,-c]$ is a discrete subset of $\R$. Also, the accumulation points of $\operatorname{Spec}(T_{P,\lambda}^{(q)})$ is the subset of $\{-\infty,+\infty\}$.
	\end{theorem}
From the above spectral theorems, we have the following formula.
 \begin{theorem}
 \label{thm:formula of chi(k^-1 T_P)}
 With the same notations and assumptions in Theorem \ref{thm:Main Semi-Classical Expansion}, for $q=n_-$ we can find an $L^2$-orthonormal system $\{f_j\}_{j\in J}$ such that $T_{P,\lambda}^{(q)}f_j=\lambda_j f_j$ and
     \begin{equation}
             \chi(k^{-1}T_{P,\lambda}^{(q)})(x,y)=\sum_{k^{-1}\lambda_j\in \supp\chi}\chi(k^{-1}\lambda_j)f_j(x)\otimes f_j^*(y)\in T_x^{*0,q}X\otimes(T_y^{*0,q}X)^*.
     \end{equation}
 \end{theorem}
  We remark that, using different calculus, we obtain \cite{BG81}*{Proposition 2.14} by combining Theorem \ref{thm:Spec(T_P,lambda^q, q=n_-)} and the following result.
\begin{theorem}
\label{thm:bdd from below}
Following Theorem \ref{thm:Main Semi-Classical Expansion}, when $q=n_-\neq n_+$, if we further assume that $p_0|_\Sigma$ is positive definite, then the set ${\rm Spec}(T_{P,\lambda}^{(q)})$ is bounded from below.
\end{theorem}
   \begin{proof}
We can find a conic neighborhood $\mathcal{C}_1^-$ of $\Sigma^-$ such that the principal symbol of $P$ is also positive definite on the closure of $\mathcal{C}_1^-$. We let $\mathcal{C}_2$ be another conic neighborhood of $\Sigma^-$ such that $\mathcal{C}_2\Subset\mathcal{C}_1$ and take a suitable $F\in L^0_{\rm cl}(X;T^{*0,q}X)$ such that $F\equiv 0$ outside $\mathcal{C}_1$ and $F\equiv I$ on $\mathcal{C}_2$. By choosing a suitable $\mathcal P\in L^1_{\rm cl}(X;T^{*0,q}X)$ which has the principal symbol positive definite on $T^*X$, it is not difficult to find an operator $\mathscr P$ given by $\mathscr P:=F\circ P+(I-F)\circ \mathcal P$ such that the principal symbol of $\mathscr P$ is also positive definite on $T^*X$ and $T_{P,\lambda}^{(q)}= T_{\mathscr P,\lambda}^{(q)}+F$ on $X$, where $F\equiv 0$ on $X$. We have a constant $c_0>0$ such that $|(Fu |u)|\leq \|F u\|\cdot\|u\|\leq c_0\|u\|^2$ holds for all $u\in \Omega^{0,q}(X)$. Also, because the principal symbol of $\mathscr P$ is positive definite on whole $T^*X$, we can apply weak G{\aa}rding inequality and find some constants $c_1,C>0$ such that
		\begin{equation}
		(\mathscr P u|u)\geq \frac{1}{C}\|u\|_{\frac{1}{2}}-C\|u\|\geq -c_1\|u\|,\ \ \forall u\in\Omega^{0,q}(X).
		\end{equation}
		 Now for $\mu\neq 0$ and $u\in\ker(T_{P,\lambda}^{(q)}-\mu I)\cap\Omega^{0,q}(X)$, by $u=\Pi_\lambda^{(q)} u$ and the above discussions, We have a constant $c_2>0$ such that
		\begin{equation}
		(T_{P,\lambda}^{(q)} u| u)=(\Pi_\lambda^{(q)}\circ \mathscr{P}\circ \Pi_\lambda^{(q)} u|u)+(F u|u)\geq -c_2 \|u\|^2
		\end{equation}
		for all $u\in\Omega^{0,q}(X)$, which implies that $\mu\geq -c_2>-\infty$ in this context.
   \end{proof}


	\subsection{Expansion of resolvent type Toeplitz operators}
	In this section we always assume $q=n_-$. With respect to \eqref{eq:system of the principal symbol intro}, we recall that we use the convention
	\begin{align}
	&p_0(x,\eta)=\sum_{|I|=|J|=q} p_{I,J}(x,\eta) \omega_I^\wedge\otimes \omega_J^{\wedge,*}~\text{for strictly increasing}~I,J~,\\
	&I_0:=\{1,\cdots,q\}\leftrightarrow\mu_1<0,\cdots,\mu_q<0;\ \ 
	J_0:=\{q+1,\cdots,n\}\leftrightarrow\mu_{q+1}>0,\cdots,\mu_n>0,\\
	&\{\mu_1,\cdots,\mu_n\}~\text{is the set of eigenvalues of the Levi form}~\mathcal{L}:=
	\left.-\frac{d\bm\alpha}{2i}\middle|\right._{T^{1,0}X}.
	\end{align}
	We also recall that we assume
\begin{equation}
    p_{I_0,I_0}(-\bm\alpha)>0;\ \ \text{and additionally}~p_{J_0,J_0}(\bm\alpha)<0~\text{when}~n_-=n_+.
\end{equation}
 
In the expansion of $(z-T_{P,\lambda}^{(q)})^{-1}\circ\Pi_\lambda^{(q)}$, we will come across various types of smoothing operators that are dependent on $z$. These operators will appear as part of the remainder of the expansion. Subsequently, we will demonstrate that when the expansion of $(z-T_{P,\lambda}^{(q)})^{-1}\circ\Pi_\lambda^{(q)}$
 is incorporated into the Helffer–Sj\"ostrand formula, the terms that involve these operators contribute solely as $k$
-negligible operators.

 From now on, we let 
 \begin{equation}
 \label{eq: cut off tau}
 \tau\in\cC(\R),~\tau(t)=0\;\;\text{for $|t|\leq 1$},
		~\tau(t)=1\;\;\text{for $|t|\geq 2$}.
		\end{equation}
	\begin{definition}
		\label{def:z-depnednt smoothing operator type 1}
		In the situation of Theorem \ref{thm:Main Semi-Classical Expansion}, for $q=n_-$ we denote by $\mathcal E_z(\Omega;T^{*0,q}X)$ the set
		of finite linear combinations of the operators with kernels
		\begin{equation}
		\int_0^{+\infty} e(x,y,t)
		\frac{z^{M_2}}{(z-tp(x))^{M_1}}\tau(\varepsilon t)dt
		\end{equation}	
		over $\C$, where the symbol $e(x,y,t)\in 
		S^{-\infty}(\Omega\times \Omega\times{\R}_+,\End(T^{*0,q}X))$ is
		properly supported in the variables $(x,y)$, $p(x)\in\cC(X,\dot\R)$, and $M_1,M_2\in\N_0$.
	\end{definition} 
	\begin{definition}
		\label{def:z-depnednt smoothing operator type 2}
		In the situation of Theorem \ref{thm:Main Semi-Classical Expansion}, for $q=n_-$ we denote by $\mathcal F_z(\Omega;T^{*0,q}X)$
		the set of finite linear combinations of the operators with kernels
		\begin{equation}
		\int_0^{+\infty} f(x,y,t)\frac{z^{M_2}}{(z-tp(x))^{M_1}}
		\tau(\varepsilon t)dt
		\end{equation}
		over $\C$, where $f(x,y,t)\in 
		S^{-\infty}(X\times \Omega\times{\R}_+,\End(T^{*0,q}X))$, $p(x)\in\cC(X,\dot\R)$, and $M_1,M_2\in\N_0$.
	\end{definition} 
	\begin{definition}
		\label{def:z-depnednt smoothing operator type 3}
		In the situation of Theorem \ref{thm:Main Semi-Classical Expansion}, for $q=n_-$ we denote by $\mathcal{G}_z(\Omega;T^{*0,q}X)$ the set of finite linear combination of the operators with kernels
		\begin{equation}
		\label{eq:mathcal G_z}
		\int_0^{+\infty} e^{it\psi(x,y)}g(x,y,t)
		\frac{z^{M_2}}{(z-tp(x))^{M_1}}\tau(\varepsilon t)dt,
		\end{equation}
		where $g(x,y,t)=O(|x-y|^{+\infty})$, $g(x,y,t)\in S^m_{\rm cl}(\Omega\times \Omega\times{\R}_+,\End(T^{*0,q}X))$ for some $m\in\mathbb R$, 
		$g(x,y,t)$ is properly supported in the variables $(x,y)$, $p(x)\in\cC(X,\dot\R)$, $M_1,M_2\in\N_0$, 
		and $\psi\in {\Ph}(\Lambda\bm\alpha,\Omega)$ for some $\Lambda\in\cC(X,{\dot\R})$.
	\end{definition} 
	\begin{definition}
 \label{def:z-depnednt smoothing operator type 4}
		In the situation of Theorem \ref{thm:Main Semi-Classical Expansion}, for $q=n_-$ we denote by $\mathcal{R}_z(\Omega;T^{*0,q}X)$ the set of finite linear combination of the operators with kernels
		\begin{equation}
		\label{eq:mathcal R_z}
		\int_0^{+\infty}\int_0^{+\infty}\int_\Omega e^{it\psi_\mp(x,w)+i\sigma\psi_\pm(w,y)}r_1(x,w,t)\circ r_2(w,y,\sigma)
		\frac{z^{M_2}}{(z-tp(x))^{M_1}}\tau(\varepsilon t)m(w) dw d\sigma dt,
		\end{equation}
  or
  \begin{equation}
      \int_0^{+\infty}\int_0^{+\infty}\int_\Omega e^{it\psi_\mp(x,w)+i\sigma\psi_\pm(w,y)}r_1(x,w,t)\circ r_2(w,y,\sigma)
		\frac{z^{M_2}}{(z-\sigma p(w))^{M_1}}\tau(\varepsilon \sigma)m(w) dw  d\sigma dt,
  \end{equation}
		where $r_1(x,w,t)\in S^{m_1}_{\rm cl}(\Omega\times \Omega\times{\R}_+,\End(T^{*0,q}X))$, $r_2(w,y,\sigma)\in S^m_{\rm cl}(\Omega\times \Omega\times{\R}_+,\End(T^{*0,q}X))$, $m_1,m_2\in\mathbb R$, 
		$r_1(x,w,t)$ is properly supported in $(x,w)$, $r_2(w,y,\sigma)$ is properly supported in $(w,y)$, $p(x)\in\cC(X,\dot\R)$, $M_1,M_2\in\N_0$, 
		and $\psi_\mp\in {\Ph}(\mp\Lambda\bm\alpha,\Omega)$ for some $\Lambda\in\cC(X,\R_+)$.
	\end{definition}
	\begin{definition}
		\label{def:L^-infty_z}
		Following Theorem \ref{thm:Main Semi-Classical Expansion}, for $q=n_-$ we define the notation $L^{-\infty}_z(\Omega;T^{*0,q}X)$ by the set collecting all elements of the form $\sum_{j\in J}c_j u_j$, where $|J|<+\infty$, $c_j\in\C$, and
  \begin{equation}
      u_j\in \mathcal E_z(\Omega;T^{*0,q}X)\cup
		\mathcal F_z(\Omega;T^{*0,q}X)\cup \mathcal{G}_z(\Omega;T^{*0,q}X)\cup \mathcal R_z(\Omega;T^{*0,q}X).
  \end{equation}
	\end{definition}
	We are ready to construct the parametrix type Fourier integral operator for the operator $(z-T_{P,\lambda}^{(q)})$. 
	\begin{theorem}
		\label{thm:leading FIO of (z-T_P)^-1 Pi I}
		With the notations and assumptions in Theorem \ref{thm:Main Semi-Classical Expansion}, \eqref{eq:full symbol intro I}, \eqref{eq:approximated Szego kernel intro} and \eqref{eq:system of the principal symbol intro}, we let $q=n_-$, $z\notin{\rm Spec}(T_{P,\lambda}^{(q)})\setminus\{0\}$, $\tau\in\cC(\R_+)$ of \eqref{eq: cut off tau} and take a fixed constant $\varepsilon>0$ so that $\tau(\varepsilon t)\chi(t)=\chi(t)$. Then the Fourier integral operator $A_{z,0}:\cCc(\Omega,T^{*0,q}X)\to\cCc(\Omega,T^{*0,q}X)$, where
	\begin{equation}
	\label{eq:A_z,0}
	A_{z,0}(x,y)
	:=\int_0^{+\infty}e^{it\varphi_-(x,y)}\frac{s_0^-(x,y)}{z-t p_{I_0,I_0}(-\bm\alpha_x)}t^n \tau(\varepsilon t) dt
	+\int_0^{+\infty}e^{it\varphi_+(x,y)}\frac{s_0^+(x,y)}{z-t p_{J_0,J_0}(\bm\alpha_x)}t^n\tau(\varepsilon t) dt,
	\end{equation}
  depends on $z$ smoothly, and  up to a kernel associated by an element in $L^{-\infty}_z(\Omega;T^{*0,q}X)$ we have
		\begin{multline}
		\left((z-T_{P,\lambda}^{(q)})\circ(S_-+S_+)\circ A_{z,0}\circ(S_-+S_+)-\Pi_\lambda^{(q)}\right)(x,y)\label{eq:(z-T_P)(S A_z S)-Pi theorem}\\
	\equiv\int_0^{+\infty}e^{it\Psi_-(x,y)}\frac{r_1^-(x,y,t;z)}{(z-t)^2}\tau(\varepsilon t)dt+\int_0^{+\infty}e^{it\Psi_+(x,y)}\frac{r_1^+(x,y,t;z)}{(z+t)^2}\tau(\varepsilon t)dt,
		\end{multline}
where $\Psi_-\in{\rm Ph}(p_{I_0,I_0}^{-1}(-\bm\alpha)(-\bm\alpha),\Omega)$, $\Psi_+\in{\rm Ph}(p_{J_0,J_0}^{-1}(-\bm\alpha)\bm\alpha,\Omega)$, and we have the following data properly supported in $(x,y)$: $r_1^+(x,y,t;z)=0$ when $n_-=n_+$, $r_1^\mp(x,y,t;z)=\sum_{|\alpha|+|\beta|\leq 2}r^\mp_{1,\alpha,\beta}(x,y,t)t^\alpha z^\beta$, $ r^\mp_{1,\alpha,\beta}(x,y,t)\in S^{n-1}_{\rm cl}(\Omega\times\Omega\times\R_+,\End(T^{*0,q}X))$. 

Also, up to a kernel associated to an element in $L^{-\infty}_z(\Omega;T^{*0,q}X)$ we have
		\begin{align}
		&\Bigr((S_-+S_+)\circ A_{z,0}\circ(S_-+S_+)\Bigr)(x,y)\nonumber\\
		\equiv&\int_0^{+\infty}e^{it\Psi_-(x,y)}\frac{\alpha^{-}(x,y,t)}{z-t}\tau(\varepsilon t) dt
		+\int_0^{+\infty}e^{it\Psi_+(x,y)}\frac{\alpha^{+}(x,y,t)}{z+t}\tau(\varepsilon t) dt,
		\end{align}
where $\Psi_\mp(x,y)$ are the same as we just mentioned, $\alpha^\mp(x,y,t)$ and $\alpha_j^\mp(x,y)$ are properly supported in the variables $(x,y)$, $j\in\N_0$, $\alpha^+(x,y,t)=0$ when $n_-\neq n_+$, $\alpha^\mp(x,y,t)\sim\sum_{j=0}^{+\infty}\alpha_j^\mp(x,y)t^{n-j}$ in $S^{n}_{\rm cl}(\Omega\times\Omega\times\R_+,\End(T^{*0,q}X))$,
    \begin{equation}
    \alpha^{-}_0(x,x)=\frac{|\det\mathcal{L}_x|}{2\pi^{n+1}}\frac{v(x)}{m(x)}~p^{-n-1}_{I_0,I_0}(-\bm\alpha_x),
\end{equation}
  and when $n_-=n_+$ we additionally have
  \begin{equation}
      \alpha^{+}_0(x,x)=\frac{|\det\mathcal{L}_x|}{2\pi^{n+1}}\frac{v(x)}{m(x)}~p^{-n-1}_{J_0,J_0}(-\bm\alpha_x).
  \end{equation}
	\end{theorem}
	\begin{proof}
	For the generality of our argument, we assume $n_-=n_+$, and the case $n_-\neq n_+$ also follows from our proof with some minor change. We first notice that for any $u\in\cCc(\Omega,T^{*0,q}X)$, by Theorem \ref{thm:Hsiao-Marinescu 17 LEF} and \eqref{eq:Pi_lambda^q is smoothing away from diagonal}, after shrinking $\Omega$, there are operators $E,F:\mathscr E'(\Omega;T^{*0,q}X)\to \cC(X;T^{*0,q}X)$ so that
\begin{multline}
    (z-T_{P,\lambda}^{(q)})\circ(S_-+S_+)u
		=(z-T_{P,\lambda}^{(q)})\circ(\Pi_\lambda^{(q)}+E)u
		=(z\Pi_\lambda^{(q)}-T_{P,\lambda}^{(q)})u+(z\Pi_\lambda^{(q)}-T_{P,\lambda}^{(q)})\circ Eu\\
		=\int_0^{+\infty} e^{it\varphi_-(x,y)}\left(z s^-(x,y,t)-t a^-(x,y,t)\right)u(y) m(y) dy dt\\
		+\int_0^{+\infty} e^{it\varphi_+(x,y)}\left(z s^+(x,y,t)-t a^+(x,y,t)\right)u(y) m(y) dy dt
		+(zE-F)u+(z\Pi_\lambda^{(q)}-T_{P,\lambda}^{(q)})\circ Eu.
\end{multline}
		We notice that the operators $zE-F$ and $(z\Pi_\lambda^{(q)}-T_{P,\lambda}^{(q)})\circ E$ are in $L^{-\infty}_z(\Omega;T^{*0,q}X)$. 
		
		On the other hand, we have $A_{z,0}\circ(S_-+S_+)=B_{z,0}^{-,-}+B_{z,0}^{+,+}+B_{z,0}^{-,+}+B_{z,0}^{+,-}$,	where
		\begin{align}
		&B_{z,0}^{-,\mp}(x,y)\\
		=&\int_\Omega \left(\int_0^{+\infty}e^{it\varphi_-(x,w)}\frac{s_0^-(x,w)}{z-t p_{I_0,I_0}(-\bm\alpha_x)}t^n \tau(\varepsilon t) dt\right)\circ S_\mp(w,y) m(w)dw\nonumber\\
		=&\int_0^{+\infty}\frac{\tau(\varepsilon t) 
			t^n}{z-t p_{I_0,I_0}(-\bm\alpha_x)}
		\left(\int_0^{+\infty}\int_\Omega e^{it(\varphi_-(x,w)+\sigma\varphi_\mp(w,y))}s_0^-(x,w)\circ s^\mp(w,y,t\sigma)m(w)dwd\sigma\right) dt,
		\end{align}
		and 
		\begin{align}
		&B_{z,0}^{+,\mp}(x,y)\\
		=&\int_\Omega \left(\int_0^{+\infty}e^{it\varphi_+(x,w)}\frac{s_0^+(x,w)}{z-t p_{J_0,J_0}(\bm\alpha_x)}t^n \tau(\varepsilon t) dt\right)\circ S_\mp(w,y) m(w)dw\nonumber\\
		=&\int_0^{+\infty}\frac{\tau(\varepsilon t) 
			t^n}{z-t p_{J_0,J_0}(\bm\alpha_x)}
		\left(\int_0^{+\infty}\int_\Omega e^{it(\varphi_+(x,w)+\sigma\varphi_\mp(w,y))}s_0^+(x,w)\circ s^\mp(w,y,t\sigma)m(w)dwd\sigma\right) dt.
		\end{align}
		By Definition \ref{def:L^-infty_z}, the operator $B_{z,0}^{\mp,\pm}$ is in $L^{-\infty}_z(\Omega;T^{*0,q}X)$. Also, using Melin--Sj\"ostrand complex stationary phase formula and Theorem \ref{thm:tangential Hessian of phase functions varphi_-}, cf.~also 
\cite{BouSj75}*{\S 4}  or \cite{Hs10}*{pp.~76-77}, we have
		\begin{align}
		&B_{z,0}^{-,-}(x,y)=\int_0^{+\infty}e^{it\varphi_-(x,y)}\frac{\tau(\varepsilon t) 
			t^n}{z-t p_{I_0,I_0}(-\bm \alpha_x)}b^{-,0}(x,y,t) dt+E_z^{-,0}(x,y),\\
		&B_{z,0}^{+,+}(x,y)=\int_0^{+\infty}e^{it\varphi_+(x,y)}\frac{\tau(\varepsilon t) 
			t^n}{z-t p_{J_0,J_0}(\bm \alpha_x)}b^{+,0}(x,y,t)dt+E_z^{+,0}(x,y),
		\end{align}
		where $E_z^{\mp,0}\in L^{-\infty}_z(\Omega;T^{*0,q}X)$ and
		\begin{align}
		&b^{\mp,0}(x,y,t)\sim\sum_{j=0}^{+\infty}b^{\mp,0}_j(x,y)t^{-j}~\text{in}~S^0_{\rm cl}(\Omega\times\Omega\times\R_+,\End(T^{*0,q}X)),\ \ b^{\mp,0}_0(x,x)=s_0^\mp(x,x).
		\end{align}
		
		Combining the calculation before, up to some element in $L^{-\infty}_z(\Omega;T^{*0,q}X)$, we can check that
		\begin{equation}
		\left((z-T_{P,\lambda}^{(q)})\circ (S_-+S_+)\right)\circ \left(A_{z,0}\circ (S_-+S_+)\right)=C_{z,0}^{-,-}+C_{z,0}^{+,+},
		\end{equation}
		where
		\begin{align}
		C_{z,0}^{-,-}(x,y)
		=&\int_0^{+\infty}\int_0^{+\infty}\int_\Omega e^{i\gamma\varphi_-(x,w)+i\beta\varphi_-(w,y)}\tau(\varepsilon \beta)\beta^n\\
		\times&\frac{z s^-(x,w,\gamma)- \gamma a^-(x,w,\gamma)}{z-\beta p_{I_0,I_0}(-\bm\alpha_w)}\circ b^{-,0}(w,y,\beta)m(w)dwd\gamma d\beta\nonumber,
		\end{align}
		and
		\begin{align}
		C_{z,0}^{+,+}(x,y)
		=&\int_0^{+\infty}\int_0^{+\infty}\int_\Omega e^{i\gamma\varphi_+(x,w)+i\beta\varphi_+(w,y)}\tau(\varepsilon \beta)\beta^n\\
		\times&\frac{z s^+(x,w,\gamma)- \gamma a^+(x,w,\gamma)}{z-\beta p_{J_0,J_0}(\bm\alpha_w)}\circ b^{+,0}(w,y,\beta)m(w)dwd\gamma d\beta.\nonumber
		\end{align}
		For $C_{z,0}^{-,-}$, we recall that $p_{I_0,I_0}(-\alpha)>0$ and we can apply the change of variables
		\begin{align}
		&\beta=p_{I_0,I_0}^{-1}(-\bm\alpha_w)t,~t\geq 0;\ \ \gamma=p_{I_0,I_0}^{-1}(-\bm\alpha_w)t\sigma,~\sigma\geq 0,
		\end{align}
		and we have
		\begin{multline}
C_{z,0}^{-,-}(x,y)
		=\int_0^{+\infty}\int_0^{+\infty}\int_\Omega \exp\left(it\cdot p^{-1}_{I_0,I_0}(-\bm\alpha_w)(\sigma\varphi_-(x,w)+\varphi_-(w,y)\right)\label{eq:C_z,0^-,-(x,y) I}\\
		\times\tau(\varepsilon p_{I_0,I_0}^{-1}(-\bm\alpha_w)t)(p_{I_0,I_0}^{-1}(-\bm\alpha_w)t)^n\\   
		\times\frac{z\cdot s^-(x,w,p_{I_0,I_0}^{-1}(-\bm\alpha_w)t\sigma) - p_{I_0,I_0}^{-1}(-\bm\alpha_w)t\sigma\cdot a^-(x,w,p_{I_0,I_0}^{-1}(-\bm\alpha_w)t\sigma)}{z-t}
\circ b^{-,0}(w,y,p_{I_0,I_0}^{-1}(-\bm\alpha_w)t)\\ p_{I_0,I_0}^{-2}(-\bm\alpha_w)t~m(w)dwd\sigma dt.
		\end{multline}
We recall that $s^-,a^-\in S^n_{\rm cl}(\Omega\times\Omega\times\R_+,\End(T^{*0,q}X))$, $b^{-,0}\in S^0_{\rm cl}(\Omega\times\Omega\times\R_+,\End(T^{*0,q}X))$, and the leading coefficients $s^-_0,a^-_0,b^{-,0}_0$ in their symbol asymptotic expansion satisfies
		\begin{align}
		\label{eq:relation between a_0^-, s_0^- and b_0^-}
		a^-_0(x,x)=p_{I_0,I_0}(-\bm\alpha_x)s_0^-(x,x)=p_{I_0,I_0}(-\bm\alpha_x)b_0^{-,0}(x,x).
		\end{align}	
  	We notice that
		\begin{multline}
		\label{eq:w=x=y in eq:leading term of C^-,-_z,0 I}
		\left.\frac{p_{I_0,I_0}^{-n}(-\bm\alpha_w) t^n\sigma^n (z s^-_0-p_{I_0,I_0}^{-1}(-\bm\alpha_w)^{-1} t\sigma a^-_0)(x,w)}{z-t}\circ  b_0^{-,0}(w,y)
  \middle|\right._{\substack{w=x=y\\ \sigma=1}}\\=p_{I_0,I_0}^{-n}(-\alpha_x)t^{n}\sigma^n s_0^-(x,x)\circ s_0^-(x,x).
		\end{multline}
 For the complex-valued function 
		\begin{equation}
		\label{eq:Phi_-}
		\Phi_-(w,\sigma;x,y):=p^{-1}_{I_0,I_0}(-\bm\alpha_w)(\sigma\varphi_-(x,w)+\varphi_-(w,y)),
		\end{equation}
		by \eqref{eq:varphi intro}, we have $(d_w \Phi)(x,1;x,x)_-=(d_\sigma\Phi_-)(x,1;x,x)=0$.	Moreover, by  the local coordinates of Theorem \ref{thm:tangential Hessian of phase functions varphi_-}, we have
\begin{equation}
  \left.\det\frac{1}{2\pi i}\begin{bmatrix}
		\left(\frac{\pr^2\Phi_-}{\pr w_j\pr w_k}\right)_{j,k=1}^{2n+1} & \left(\frac{\pr^2\Phi_-}{\pr w_j\pr \sigma}\right)_{j=1}^{2n+1}\\
		\left(\frac{\pr\Phi_-}{\pr\sigma\pr w_j}\right)_{j=1}^{2n+1} & \frac{\pr\Phi_-}{\pr\sigma^2}
		\end{bmatrix}\middle|\right._{\substack{w=x=y=x_0\\ \sigma=1}}= p^{-2n-2}_{I_0,I_0}(-\bm\alpha_{x_0})\frac{2^{2n-2}}{\pi^{2n+2}}|\mu_1\cdots\mu_n|^2.\label{eq:value of det H(Phi_-)/2 pi i}
\end{equation}
So by Melin--Sj\"ostrand complex stationary phase formula \cite{MeSj74}*{Theorem 2.3}, up to a kernel associated to an element in $L^{-\infty}_z(\Omega;T^{*0,q}X)$, we can write
		\begin{align}
		C_{z,0}^{-,-}(x,y)
		=&\int_0^{+\infty} e^{it\Psi_-(x,y)}\frac{z\mathds S^-_0(x,y)-t\mathds A^-_0(x,y)}{z-t}t^n \tau(\varepsilon t)dt\label{eq:C_z,0^-,-(x,y)}\\
		+&\int_0^{+\infty} e^{it\Psi_-(x,y)}\frac{z\mathds E^-_1(x,y,t)-t\mathds F^-_1(x,y,t)}{z-t}\tau(\varepsilon t)dt,\nonumber
		\end{align}
where $\Psi_-(x,y)$ is the critical value \cite{MeSj74}*{Lemma 2.1} for the almost analytic extension $\Td\Phi_-(\Td w,\Td\sigma;\Td x,\Td y)$ of \eqref{eq:Phi_-}, and $\mathds E_1^-(x,y,t)$ and $\mathds F_1^-(x,y,t)$ are in $S^{n-1}_{\rm cl}(\Omega\times\Omega\times\R_+,\End(T^{*0,q}X))$. Moreover, by \eqref{eq:varphi intro}, we have
\begin{equation}
\Psi_-(x,y)\in{\rm Ph}\left(-p_{I_0,I_0}^{-1}(\bm\alpha)(-\bm\alpha)\right);    
\end{equation}
by \eqref{eq:relation between a_0^-, s_0^- and b_0^-}, \eqref{eq:value of det H(Phi_-)/2 pi i}, and our convention that the volume induced by Hermitian metric satisfies $v(0)=2^n$, 
 we have
\begin{align}
		&\mathds S_0^-(x,x)=\mathds A_0^-(x,x)=p_{I_0,I_0}^{n+1}(-\bm\alpha_x)s_0^-(x,x)\label{eq:mathds S_0^-=A_0^-}.
		\end{align}		

		Similarly, since we assume that $p_{J_0,J_0}(-\bm\alpha)=-p_{J_0,J_0}(\bm\alpha)>0$ when $n_-=n_+$, we can also write
		\begin{multline}
		C_{z,0}^{+,+}(x,y)
		=\int_0^{+\infty}\int_0^{+\infty}\int_\Omega \exp\left(it\cdot p^{-1}_{J_0,J_0}(-\bm\alpha_w)(\sigma\varphi_+(x,w)+\varphi_+(w,y)\right)\\
		\times\tau(\varepsilon p_{J_0,J_0}^{-1}(-\bm\alpha_w)t)(p_{J_0,J_0}^{-1}(-\bm\alpha_w)t)^n\\   
		\times \frac{z\cdot s^+(x,w,p_{J_0,J_0}^{-1}(-\bm\alpha_w)t\sigma) - p_{J_0,J_0}^{-1}(-\bm\alpha_w)t\sigma\cdot a^+(x,w,p_{J_0,J_0}^{-1}(-\bm\alpha_w)t\sigma)}{z+t}
		\circ b^{+,0}(w,y,p_{J_0,J_0}^{-1}(-\bm\alpha_w)t)\\ p_{J_0,J_0}^{-2}(-\bm\alpha_w)t~m(w)dwd\sigma dt.
		\end{multline}
		Then by the same argument for $C_{z,0}^{-,-}$, up to a kernel associated to an element in $L^{-\infty}_z(\Omega;T^{*0,q}X)$, we can write
		\begin{align}
		C_{z,0}^{+,+}(x,y)
		=&\int_0^{+\infty} e^{it\Psi_+(x,y)}\frac{z\mathds S^+_0(x,y)+t\mathds A^+_0(x,y)}{z+t}t^n \tau(\varepsilon t)dt\label{eq:C_z,0^+,+(x,y)}\\
		+&\int_0^{+\infty} e^{it\Psi_+(x,y)}\frac{z\mathds E^+_1(x,y,t)+t\mathds F^+_1(x,y,t)}{z+t}\tau(\varepsilon t)dt,\nonumber
		\end{align}
		where $\mathds E_1^+(x,y,t),~\mathds F_1^+(x,y,t)\in S^{n-1}_{\rm cl}(\Omega\times\Omega\times\R_+,\End(T^{*0,q}X))$, and
		\begin{align}
		&\Psi_+(x,y)\in {\rm Ph}(p_{J_0,J_0}^{-1}(-\bm\alpha)\bm\alpha),\ \ \mathds S_0^+(x,x)=\mathds A_0^+(x,x)=p_{J_0,J_0}^{n+1}(\bm\alpha_x)s_0^+(x,x).\label{eq:mathds S_0^+=A_0^+}
		\end{align}
		
		We notice that in terms of Definition \ref{def:Szego Phase functions} and Theorem \ref{thm:s_0^psi(x,x)}, we can write
		\begin{equation}
		S_\mp(x,y)\equiv\int_0^{+\infty}e^{it\Psi_\mp(x,y)} \left(s^{\Psi_\mp}_0(x,y) t^n+s_1^{\Psi_\mp}(x,y,t)\right)\tau(\varepsilon t)dt,
		\end{equation}
		where $s_1^{\Psi_\mp}(x,y,t)\in S^{n-1}_{\rm cl}(\Omega\times\Omega\times\R_+,\End (T^{*0,q}X))$ and
		\begin{equation}
		\label{eq:s_0^Psi_-}
		s_0^{\Psi_-}(x,x)=p_{I_0,I_0}^{-n-1}(-\bm\alpha_x)s_0^-(x,x).
		\end{equation}
		Then, up to a kernel associated with an element in $L^{-\infty}_z(\Omega;T^{*0,q}X)$ we have
		\begin{align}
		&\left((z-T_{P,\lambda}^{(q)})\circ(S_-+S_+)\circ A_{z,0}\circ(S_-+S_+)-\Pi_\lambda^{(q)}\right)(x,y)\nonumber\\
		\equiv&\int_0^{+\infty}e^{it\Psi_-(x,y)}\left(\frac{z\mathds S_0^-(x,y)-t\mathds A_0^-(x,y)}{z-t}-s^{\Psi_-}_0(x,y)\right)t^n\tau(\varepsilon t)dt\label{eq:z-FIO Psi_- I}\\
		+&\int_0^{+\infty}e^{it\Psi_-(x,y)}\left(\frac{z\mathds S_0^+(x,y)+t\mathds A_0^+(x,y)}{z+t}-s^{\Psi_+}_0(x,y)\right)t^n\tau(\varepsilon t)dt\nonumber\\
		+&\int_0^{+\infty} e^{it\Psi_-(x,y)}\frac{z(\mathds E^-_1-s_1^{\Psi_-})(x,y,t)-t(\mathds F^-_1-s_1^{\Psi_-})(x,y,t)}{z-t}\tau(\varepsilon t)dt\nonumber\\
		+&\int_0^{+\infty} e^{it\Psi_+(x,y)}\frac{z(\mathds E^+_1-s_1^{\Psi_-})(x,y,t)+t(\mathds F^+_1-s_1^{\Psi_-})(x,y,t)}{z+t}\tau(\varepsilon t)dt,\nonumber
		\end{align}
where $\mathds E_1^\mp(x,y,t),~\mathds F_1^\mp(x,y,t),~s_1^{\Psi_\mp}(x,y,t)\in S^{n-1}_{\rm cl}(\Omega\times\Omega\times\R_+,\End(T^{*0,q}X))$.	
	
	The final step of the proof is to apply Theorem \ref{thm:division theorem of principal symbol for Szego FIO} to reduce the above formula. For the purpose we first let 
		\begin{equation}
		\label{eq:mathds I_0 I}
		\mathds I:=\left(\left.(z-T_{P,\lambda}^{(q)})\circ(S_-+S_+)\circ A_{z,0}\circ(S_-+S_+)-\Pi_\lambda^{(q)}\right)\middle|\right._{z=0}.
		\end{equation}
		From the previous calculation, we can check that up to a smoothing kernel on $\Omega\times\Omega$ we have
		\begin{align}
		\mathds I(x,y)
		\equiv&\int_0^{+\infty}e^{it\Psi_-(x,y)}(\mathds A^-_0-s^{\Psi_-}_0)(x,y)t^n\tau(\varepsilon t)dt\\
		+&\int_0^{+\infty}e^{it\Psi_+(x,y)}(\mathds A^+_0-s^{\Psi_+}_0)(x,y)t^n\tau(\varepsilon t)dt\nonumber\\
		+&\int_0^{+\infty} e^{it\Psi_-(x,y)}(\mathds F^-_1-s_1^{\Psi_-})(x,y,t)\tau(\varepsilon t)dt
		+\int_0^{+\infty} e^{it\Psi_+(x,y)}(\mathds F^+_1-s_1^{\Psi_-})(x,y,t)\tau(\varepsilon t)dt.\nonumber
		\end{align}
		By \eqref{eq:mathds I_0 I}, \eqref{eq:mathds S_0^-=A_0^-}, \eqref{eq:mathds S_0^+=A_0^+}, \eqref{eq:s_0^Psi_-} and the above formula, we can see that $I_0$ satisfies all the assumptions in Theorem \ref{thm:division theorem for equivalent Szego phase function} and we have
		\begin{align}
		&(\mathds A^\mp_0-s_0^{\Psi_\mp})(x,y)-f_0^\mp(x,y)\Psi_\mp(x,y)=O(|x-y|^{+\infty}),\label{eq:mathds A^mp_0-s^Psi_mp_0}.
		\end{align}
		On the other hand, if we let
		\begin{equation}
		\label{eq:mathds II}
		\mathds{II}:=\frac{\partial}{\partial z}\left(\left.(z-T_{P,\lambda}^{(q)})\circ(S_-+S_+)\circ A_{z,0}\circ(S_-+S_+)-\Pi_\lambda^{(q)}\right)\middle|\right._{z=0},
		\end{equation}
		then directly from \eqref{eq:z-FIO Psi_- I} we have
		\begin{align}
		&\mathds {II}(x,y)\\
		\equiv&\int_0^{+\infty}e^{it\Psi_-(x,y)}(-\mathds S^-_0+\mathds A^-_0)(x,y)t^{n-1}\tau(\varepsilon t)dt
		+\int_0^{+\infty}e^{it\Psi_+(x,y)}(\mathds S^+_0-\mathds A^+_0)(x,y)t^{n-1}\tau(\varepsilon t)dt\nonumber\\
		+&\int_0^{+\infty} e^{it\Psi_-(x,y)}(-\mathds E^-_1+\mathds F_1^-)(x,y,t)t^{-1}\tau(\varepsilon t)dt+\int_0^{+\infty} e^{it\Psi_+(x,y)}(\mathds E^+_1-\mathds F_1^+)(x,y,t)t^{-1}\tau(\varepsilon t)dt.\nonumber
		\end{align}
		Again by \eqref{eq:mathds I_0 I}, \eqref{eq:mathds S_0^-=A_0^-}, \eqref{eq:mathds S_0^+=A_0^+}, \eqref{eq:s_0^Psi_-} and the above formula, we can see $\mathds{II}$ satisfies all the assumptions in Theorem \ref{thm:division theorem for equivalent Szego phase function} and we have
		\begin{align}
		&(-\mathds S^-_0+\mathds A^-_0)(x,y)-f_1^-(x,y)\Psi_-(x,y)=O(|x-y|^{+\infty}),\label{eq:mathds -S^-_0+A^-_0 I}\\
		&(\mathds S^+_0-\mathds A^+_0)(x,y)-f_1^+(x,y)\Psi_+(x,y)=O(|x-y|^{+\infty}).\label{eq:mathds S^+_0-A^+_0 I}
		\end{align}
		Then by \eqref{eq:mathds A^mp_0-s^Psi_mp_0}, \eqref{eq:mathds -S^-_0+A^-_0 I} and \eqref{eq:mathds S^+_0-A^+_0 I}, up to a kernel associated to $L^{-\infty}_z(\Omega;T^{*0,q}X)$ we can write
		\begin{align}
		&\int_0^{+\infty}e^{it\Psi_-(x,y)}\left(\frac{z\mathds S_0^-(x,y)-t\mathds A_0^-(x,y)}{z-t}-s^{\Psi_-}_0(x,y)\right)t^n\tau(\varepsilon t)dt\nonumber\\
		+&\int_0^{+\infty}e^{it\Psi_-(x,y)}\left(\frac{z\mathds S_0^+(x,y)+t\mathds A_0^+(x,y)}{z+t}-s^{\Psi_+}_0(x,y)\right)t^n\tau(\varepsilon t)dt\nonumber\\
		\equiv&\int_0^{+\infty}e^{it\Psi_-(x,y)}\frac{g_1^-(x,y,t;z)}{(z-t)^2}t^{n-1}\tau(\varepsilon t)dt
		+\int_0^{+\infty}e^{it\Psi_-(x,y)}\frac{g_1^+(x,y,t;z)}{(z+t)^2}t^{n-1}\tau(\varepsilon t)dt,\nonumber
		\end{align}
		where $g_1^\mp(x,y,t;z)=\sum_{|\alpha|+|\beta|\leq 1}g^\mp_{1,\alpha,\beta}(x,y)t^\alpha z^\beta$.
  
  Combining all the calculation above, we can conclude our theorem.
	\end{proof}
	Now we can state and prove the most important result in this section.
	\begin{theorem} 
		\label{thm:expansion of (z-T_P)^-1 Pi} 
		With the assumptions and notations of Theorem \ref{thm:leading FIO of (z-T_P)^-1 Pi I}, for every $N\in\N_0$, we have Fourier integral operators $A_{z,0},A_{z,1},\cdots A_{z,N},R_{z,N+1}:\cCc(\Omega;T^{*0,q}X)\to\cCc(\Omega;T^{*0,q}X)$, which smoothly depend on $z$, such that up to an element in $L^{-\infty}_z(\Omega;T^{*0,q}X)$ we have
		\begin{equation}
		(z-T_{P,\lambda}^{(q)})\circ   \sum_{j=0}^{N}(S_-+S_+)\circ A_{z,j} \circ(S_-+S_+)\equiv \Pi_\lambda^{(q)}+R_{z,N+1}.
		\end{equation}
 In fact, for each $j=0,\cdots,N$, up to a kernel associated to an element in $L^{-\infty}_z(\Omega;T^{*0,q}X)$, we have
\begin{multline}
\label{eq:(S_-+S_+) circ A_z,j circ (S_-+S_+)(x,y)}
    (S_-+S_+)\circ A_{z,j}\circ(S_-+S_+)(x,y)\\
		\equiv\int_0^{+\infty}e^{it\Psi_-(x,y)}\frac{\sum_{|\beta|+|\gamma|\leq 2j}\alpha_{\beta,\gamma}^{-,j}(x,y,t)t^\beta z^\gamma}{\left(z-t\right)^{2j+1}}\tau(\varepsilon t) dt\\
		+\int_0^{+\infty}e^{it\Psi_+(x,y)}\frac{\sum_{|\beta|+|\gamma|\leq 2j} \alpha_{\beta,\gamma}^{+,j}(x,y,t)t^\beta z^\gamma }{\left(z+t \right)^{2j+1}}\tau(\varepsilon t) dt,
\end{multline}
		 where $\alpha_{\beta,\gamma}^\mp(x,y,t)$ is properly supported in the variables $(x,y)$, $\alpha_{\beta,\gamma}^\mp(x,y,t)=0$ when $n_-\neq n_+$, $\alpha_{\beta,\gamma}^\mp(x,y,t)\in S^{n-j}_{\rm cl}\left(\Omega\times\Omega\times\R_+,\End(T^{*0,q}X)\right)$. Also, up to a kernel associated to an element in $L^{-\infty}_z(\Omega;T^{*0,q}X)$, we have
\begin{multline}
    R_{z,N+1}(x,y)
		\equiv\int_0^{+\infty}e^{it\Psi_-(x,y)}\frac{\sum_{|\beta|+|\gamma|\leq 2N+2}R_{\beta,\gamma}^{-,N+1}(x,y,t)t^\beta z^\gamma}{\left(z-t\right)^{2N+2}}\tau(\varepsilon t) dt\\
		+\int_0^{+\infty}e^{it\Psi_+(x,y)}\frac{\sum_{|\beta|+|\gamma|\leq 2N+2} R_{\beta,\gamma}^{+,N+1}(x,y,t)t^\beta z^\gamma }{\left(z+t\right)^{2N+2}}\tau(\varepsilon t) dt,
\end{multline}
where $R_{\beta,\gamma}^{\mp,N+1}(x,y,t)$ is properly supported in the variables $(x,y)$, $ R_{\beta,\gamma}^{+,N+1}(x,y,t)=0$ when $n_-\neq n_+$, $R_{\beta,\gamma}^{\mp,N+1}(x,y,t)\in S^{n-N-1}_{\rm cl}(\Omega\times\Omega\times\R_+,\End(T^{*0,q}X))$.
	\end{theorem} 
	\begin{proof}
	From Theorem \ref{thm:leading FIO of (z-T_P)^-1 Pi I}, we already have Fourier integral operators $A_{z,0}$ and $R_{z,1}$ with all the properties we need. Especially, the properties of \eqref{eq:(S_-+S_+) circ A_z,j circ (S_-+S_+)(x,y)} are verified from the calculation of Melin--Sj\"ostrand stationary phase method applied in \eqref{eq:C_z,0^-,-(x,y)} and \eqref{eq:C_z,0^+,+(x,y)}. This suggests us to use induction to prove our theorem. Now we assume our theorem holds for some $N=N_0\in\N_0$. We denote
		\begin{align}
		R_{z,N_0+1}(x,y)
		:=&\int_0^{+\infty}e^{it\Psi_-(x,y)}\frac{\sum_{|\beta|+|\gamma|\leq 2N_0+2}R_{\beta,\gamma}^{-,N_0+1}(x,y,t)t^\beta z^\gamma}{\left(z-t\right)^{2N_0+2}}\tau(\varepsilon t) dt\\
		+&\int_0^{+\infty}e^{it\Psi_+(x,y)}\frac{\sum_{|\beta|+|\gamma|\leq 2N_0+2} R_{\beta,\gamma}^{+,N_0+1}(x,y,t)t^\beta z^\gamma }{\left(z+t \right)^{2N_0+2}}\tau(\varepsilon t) dt\nonumber,
		\end{align}
	where $R_{\beta,\gamma}^{\mp,N_0+1}(x,y,t)\sim\sum_{\ell=0}^{+\infty}R_{\beta,\gamma,\ell}^{\mp,N_0+1}(x,y)t^{n-N_0-1-\ell}~\text{in $S^{n-N-1}_{1,0}(\Omega\times\Omega\times\R_+,\End(T^{*0,q}X))$}$. Now, for any $z\notin{\rm Spec}(T_{P,\lambda}^{(q)})\setminus\{0\}$ we consider the operator $A_{z,N_0+1}:\cCc(\Omega)\to\cCc(\Omega)$ determined by the oscillatory integral
		\begin{align}
		A_{z,N_0+1}(x,y)
		:=&\int_0^{+\infty}e^{it\varphi_-(x,y)}\frac{\sum_{|\beta|+|\gamma|\leq 2N_0+2}\alpha_{\beta,\gamma}^{-,N_0+1}(x,y,t)t^\beta z^\gamma}{\left(z-t p_{I_0,I_0}(-\alpha_x)\right)^{2N_0+3}}\tau(\varepsilon t) dt\label{eq:A_z,N_0+1(x,y)}\\
		+&\int_0^{+\infty}e^{it\varphi_+(x,y)}\frac{\sum_{|\beta|+|\gamma|\leq 2N_0+2}\alpha_{\beta,\gamma}^{+,N_0+1}(x,y,t)t^\beta z^\gamma }{\left(z-t p_{J_0,J_0}(\alpha_x)\right)^{2N_0+3}}\tau(\varepsilon t) dt,\nonumber
		\end{align}
		where we have the following symbols properly supported in the variables $(x,y)$:
		\begin{align}
		&\alpha_{\beta,\gamma}^{-,N_0+1}(x,y,t):=- R_{\beta,\gamma,0}^{-,N_0+1}(x,y)\cdot p_{I_0,I_0}^{(n-N_0-1+\beta)+1}(-\alpha_x) t^{n-N_0-1},\\
		&\alpha_{\beta,\gamma}^{+,N_0+1}(x,y,t):=-R_{\beta,\gamma,0}^{+,N_0+1}(x,y)\cdot p_{J_0,J_0}^{(n-N_0-1+\beta)+1}(-\alpha_x) t^{n-N_0-1}.
		\end{align}
		From our construction, it is clear that
		\begin{align}
		&\alpha_{\beta,\gamma}^{+,j}(x,y,t)=0~\text{when}~n_-\neq n_+,   
		\end{align}
		and by the same stationary phase method of Melin--Sj\"ostrand applied in \eqref{eq:C_z,0^-,-(x,y)} and \eqref{eq:C_z,0^+,+(x,y)}, up to a kernel associated to an element in $L^{-\infty}_z(\Omega;T^{*0,q}X)$ we can check that
		\begin{multline}
			\left((z-T_{P,\lambda}^{(q)})\circ   (S_-+S_+)\circ A_{z,N_0+1} \circ(S_-+S_+)+
		R_{z,N_0+1}\right)(x,y)\\
		\equiv\int_0^{+\infty}e^{it\Psi_-(x,y)}\tau(\varepsilon t)\\
        \frac{\sum_{|\beta|+|\gamma|\leq 2N_0+2}\left(z(\mathds S^{-,N_0+1}_{\beta,\gamma}+R^{-,N_0+1}_{\beta,\gamma})-t (\mathds A^{-,N_0+1}_{\beta,\gamma}+R^{-,N_0+1}_{\beta,\gamma})\right)(x,y,t)t^\beta z^\gamma}{(z-t)^{2N_0+3}}dt\\
		+\int_0^{+\infty}e^{it\Psi_+(x,y)}\tau(\varepsilon t)\\
        \frac{\sum_{|\beta|+|\gamma|\leq 2N_0+2}\left(z(\mathds S^{+,N_0+1}_{\beta,\gamma}+R^{+,N_0+1}_{\beta,\gamma})+t(\mathds A^{+,N_0+1}_{\beta,\gamma}+R^{+,N_0+1}_{\beta,\gamma})\right)(x,y,t)t^\beta z^\gamma}{(z+t)^{2N_0+3}}dt,
		\end{multline}
  where we have $\mathds S^{\mp,N_0+1}_{\beta,\gamma}\sim\sum_{\ell=0}^{+\infty}\mathds S^{\mp,N_0+1}_{\beta,\gamma,\ell}(x,y)t^{n-N_0-1-\ell}$, $\mathds A^{\mp,N_0+1}_{\beta,\gamma}\sim\sum_{\ell=0}^{+\infty}\mathds A^{\mp,N_0+1}_{\beta,\gamma,\ell}(x,y)t^{n-N_0-1-\ell}$ in $S^{n-N_0-1}_{1,0}(\Omega\times\Omega\times\R_+,\End(T^{*0,q}X))$,
and we have the leading term relation
		\begin{equation}
		\label{eq:mathds S, A and -R}
		\mathds S^{\mp,N_0+1}_{\beta,\gamma,0}(x,x)=\mathds A^{\mp,N_0+1}_{\beta,\gamma,0}(x,x)=-R^{\mp,N_0+1}_{\beta,\gamma,0}(x,x).
		\end{equation}
	This implies that up to a kernel associated to an element in $L^{-\infty}_z(\Omega;T^{*0,q}X)$ we have
		\begin{align}
  &\left((z-T_{P,\lambda}^{(q)})\circ   (S_-+S_+)\circ A_{z,N_0+1} \circ(S_-+S_+)+
		R_{z,N_0+1}\right)(x,y)\\
		\equiv&\int_0^{+\infty}e^{it\Psi_-(x,y)}\tau(\varepsilon t)\frac{\sum_{|\beta'|+|\gamma'|\leq 2N_0+3}\mathds B^{-,N_0+1}_{\beta',\gamma'}(x,y,t)t^{\beta'} z^{\gamma'}}{(z-t)^{2N_0+3}}dt\label{eq:B minus}\\
		+&\int_0^{+\infty}e^{it\Psi_+(x,y)}\tau(\varepsilon t)\frac{\sum_{|\beta'|+|\gamma'|\leq 2N_0+3}\mathds B^{+,N_0+1}_{\beta',\gamma'}(x,y,t)t^{\beta'} z^{\gamma'}}{(z+t)^{2N_0+3}}dt,\nonumber
		\end{align}
		where $\mathds B^{\mp,N_0+1}_{\beta',\gamma'}(x,y,t)\sim\sum_{\ell=0}^{+\infty}\mathds B^{\mp,N_0+1}_{\beta',\gamma',\ell}(x,y)t^{n-N_0-1-\ell}$ in $S^{n-N_0-1}_{1,0}(\Omega\times\Omega\times\R_+,\End(T^{*0,q}X))$. By \eqref{eq:mathds S, A and -R}, we also have the leading term relation
		\begin{equation}
		\mathds B^{\mp,N_0+1}_{\beta',\gamma',0}(x,x)=0,\ \ \text{for all $(\beta',\gamma')\in\N_0^2$ such that $|\beta'|+|\gamma'|\leq 2N_0+3$.}
		\end{equation}
		On the other hand, we notice that by induction hypothesis we have
		\begin{equation}
		\label{eq:R_z^N_0+1}
		(z-T_{P,\lambda}^{(q)})\circ   \sum_{j=0}^{N_0}(S_-+S_+)\circ A_{z,j} \circ(S_-+S_+)- \Pi_\lambda^{(q)}
		\equiv R_{z,N_0+1}
		\end{equation}
		up to an element in $L^{-\infty}_z(\Omega;T^{*0,q}X)$. Thus, we consider the operator
		\begin{equation}
		\label{eq:mathds I_0}
		\mathds I_0:=\left.\left((z-T_{P,\lambda}^{(q)})\circ   (S_-+S_+)\circ A_{z,N_0+1} \circ(S_-+S_+)+
		R_{z,N_0+1}\right)\middle|\right._{z=0}.
		\end{equation}
		We notice that up to a kernel associated with an element in $L^{-\infty}_z(\Omega;T^{*0,q}X)$, we have
		\begin{align}
		\mathds I_0(x,y)
		\equiv&\int_0^{+\infty}e^{it\Psi_-(x,y)}\frac{\sum_{|\beta'|\leq 2N_0+3}\mathds B^{-,N_0+1}_{\beta',0}(x,y,t)t^\beta}{(-t)^{2N_0+3}}\tau(\varepsilon t)dt\\
		+&\int_0^{+\infty}e^{it\Psi_+(x,y)}\frac{\sum_{|\beta'|\leq 2N_0+3}\mathds B^{+,N_0+1}_{\beta',0}(x,y,t)t^\beta}{t^{2N_0+3}}\tau(\varepsilon t)dt,\nonumber
		\end{align}
		and from \eqref{eq:mathds S, A and -R} and \eqref{eq:R_z^N_0+1} we can check that
		$\mathds I_0$ satisfies all the assumptions in Theorem \ref{thm:division theorem of principal symbol for Szego FIO}. This implies that
		\begin{align}
		&\mathds B^{\mp,N_0+1}_{2N_0+3,0,0}(x,y)-f^{\mp,N_0+1}_{2N_0+3,0}(x,y)\Psi_\mp(x,y)=O(|x-y|^{+\infty}),\label{eq:2N_0+3,0 mp}.
		\end{align}
		Next, we consider
		\begin{equation}
		\label{eq:mathds I_1}
		\mathds I_1:=\left.\left(\frac{\pr }{\pr z}(z-T_{P,\lambda}^{(q)})\circ(S_-+S_+)\circ A_{z,N_0+1} \circ(S_-+S_+)+\frac{\pr }{\pr z}
		R_{z,N_0+1}\right)\middle|\right._{z=0}.
		\end{equation}
		we can apply the same argument for $\mathds I_1$ and check that
		\begin{equation}
		\left(\mathds B^{\mp,N_0+1}_{2N_0+2,1,0}+(2N_0+3)\mathds B^{\mp,N_0+1}_{2N_0+3,0,0}\right)(x,y)
		-g^{\mp,N_0+1}_{2N_0+2,1}(x,y)\Psi_\mp(x,y)
  =O(|x-y|^{+\infty}).
		\end{equation}
		From \eqref{eq:2N_0+3,0 mp}, we immediately have
		\begin{align}
		&\mathds B^{\mp,N_0+1}_{2N_0+2,1,0}(x,y)-f^{\mp,N_0+1}_{2N_0+2,1}(x,y)\Psi_\mp(x,y)=O(|x-y|^{+\infty}).
		\end{align}
	Then inductively for $j=2,3,\cdots,2N_0+3$, we also consider
		\begin{equation}
		\label{eq:mathds I_ell}
		\mathds I_\ell:=\left.\left(\frac{\pr^j }{\pr z^j}(z-T_{P,\lambda}^{(q)})\circ(S_-+S_+)\circ A_{z,N_0+1} \circ(S_-+S_+)+\frac{\pr^j}{\pr z^j}
		R_{z,N_0+1}\right)\middle|\right._{z=0}.
		\end{equation}
		We can then use the same method above to recursively verify that for $j=2,3,\cdots,2N_0+3$ we also have
		\begin{align}
		&\mathds B^{\mp,N_0+1}_{2N_0+3-j,j,0}(x,y)-f^{\mp,N_0+1}_{2N_0+3-j,j}(x,y)\Psi_\mp(x,y)=O(|x-y|^{+\infty}).
		\end{align}
		These relations enable us to apply integration by parts in $t$ in \eqref{eq:B minus}, and after some arrangement we can see that
		up to a kernel associated to an element in $L^{-\infty}_z(\Omega;T^{*0,q}X)$ we have
		\begin{align}
		& \left((z-T_{P,\lambda}^{(q)})\circ   \sum_{j=0}^{N_0+1}(S_-+S_+)\circ A_{z,j} \circ(S_-+S_+)- \Pi_\lambda^{(q)}\right)(x,y)\nonumber\\
		\equiv&\int_0^{+\infty}e^{it\Psi_-(x,y)}\frac{\sum_{|\beta|+|\gamma|\leq 2N_0+4}R_{\beta,\gamma}^{-,N_0+2}(x,y,t)t^\beta z^\gamma}{\left(z-t\right)^{2N_0+4}}\tau(\varepsilon t) dt\\
		+&\int_0^{+\infty}e^{it\Psi_+(x,y)}\frac{\sum_{|\beta|+|\gamma|\leq 2N_0+4} R_{\beta,\gamma}^{+,N+1}(x,y,t)t^\beta z^\gamma }{\left(z+t \right)^{2N_0+4}}\tau(\varepsilon t) dt,\nonumber
		\end{align}
		where $R_{\beta,\gamma}^{\mp,N_0+2}(x,y,t)\in S^{n-N_0-2}_{\rm cl}(\Omega\times\Omega\times\R_+,\End(T^{*0,q}X))$ is properly supported in the variables $(x,y)$ and $R_{\beta,\gamma}^{+,N_0+2}(x,y,t)=0$ when $n_-\neq n_+$.
  
		This completes the induction and the proof of our theorem.
	\end{proof}	
\section{Semi-classical asymptotic expansion for the spectral operator}
In this part we prove Theorems \ref{thm:Main Semi-Classical Expansion} and \ref{thm:Main Leading Term}. We recall that we assume $X$ is compact. From Theorem \ref{thm:WF of Szego projection on lower energy forms}, when $q\notin\{n_-,n_+\}$ we have $\Pi_\lambda^{(q)}\in L^{-\infty}(X;T^{*0,q}X)$, and so does $T_{P,\lambda}^{(q)}$. By standard functional analysis, this implies that $T_{P,\lambda}^{(q)}$ is a compact operator on $X$, and it is known that in this situation the spectrum is a bounded set in $\R$ ,~cf.~\cite{Dav95}*{Theorem 4.2.2} for example. As we assume that $\chi\in\cCc(\dot\R)$, when $k\to+\infty$ we can conclude that:
	\begin{equation}
 \label{eq:conclusion for q not being n_- or n_+}
	\text{If}~q\notin\{n_-,n_+\},~\chi(k^{-1}T_{P,\lambda}^{(q)})=0\ \ \text{on}~X.
	\end{equation}
	The main difficulty is the case $q=n_-$. Since $\chi\in\cCc({\dot\R})$, it is known that $\chi(k^{-1}T_{P,\lambda}^{(q)})= \chi(k^{-1}T_{P,\lambda}^{(q)})\circ\Pi_\lambda^{(q)}$.	Our strategy is to apply $\chi(k^{-1}T_{P,\lambda}^{(q)})=\chi(k^{-1}T_{P,\lambda}^{(q)})\circ\Pi_\lambda^{(q)}$ and Helffer--Sj\"ostrand formula
	\begin{equation}    \chi(k^{-1}T_{P,\lambda}^{(q)})
	=\int_\C\frac{\pr\Td\chi(z)}{\pr\ol z}(z-k^{-1}T_{P,\lambda}^{(q)})^{-1}\frac{dz\wedge d\ol z}{2\pi i}
	=\int_\C\frac{\pr\Td\chi(\frac{z}{k})}{\pr\ol z}(z-T_{P,\lambda}^{(q)})^{-1}\frac{dz\wedge d\ol z}{2\pi i},
	\end{equation}
	and solve the full asymptotic expansion of the Schwartz kernel $\chi(k^{-1}T_{P,\lambda}^{(q)})(x,y)$ as $k\to+\infty$.

	\subsection{Helffer--Sj\"ostrand formula and the semi-classical estimates}
	In this section, we establish the semi-classical estimate for the operator
	\begin{align}
\int_\C\frac{\pr\Td\chi(\frac{z}{k})}{\pr\ol z}(z-T_{P,\lambda}^{(q)})^{-1}\circ\Pi_{\lambda}^{(q)}\frac{dz\wedge d\ol z}{2\pi i},\ \ k\to+\infty.
	\end{align}
	in the operator level as $k\to+\infty$. To simplify the discussion we define some notations.
\begin{definition}
\label{def:z-FIO Szego type order m}
With the notations and assumptions in Theorem \ref{thm:expansion of (z-T_P)^-1 Pi}, for $N,\beta\in\N$, we use the notation $S^{-N,\beta}_{\Sigma,z}(\Omega;T^{*0,q}X)$ for the space collecting the finite sum of $z$-dependent Szeg\H{o} type Fourier integral operators $H_z$ with the kernels
\begin{multline}
		H_z(x,y):=\int_0^{+\infty}e^{it\psi_-(x,y)}\frac{\sum_{\alpha+\gamma\leq\beta}z^\gamma h^-_{\alpha,\gamma}(x,y,t)}{(z-t)^\beta}\tau(\varepsilon t)dt\\
		+\int_0^{+\infty}e^{it\psi_+(x,y)}\frac{\sum_{\alpha+\gamma\leq\beta}z^\gamma h^+_{\alpha,\gamma}(x,y,t)}{(z+t)^\beta}\tau(\varepsilon t)dt,
		\end{multline}
  where $h^\mp_{\alpha,\gamma}(x,y,t)\in S^{n-N+\alpha}_{\rm cl}(\Omega\times\Omega\times\R_+,\End(T^{*0,q}X))$, $h^\mp_{\alpha,\gamma}(x,y,t)$ is 
  properly supported in the variables $(x,y)$, $h^+_{\alpha,\gamma}(x,y,t)=0$ when $n_-\neq n_+$, $\psi_\mp\in{\rm Ph}(\mp\Lambda\bm\alpha,\Omega)$ and $\Lambda\in\cC(X,\R_+)$.
\end{definition}
	\begin{definition}
		\label{def:properly supported semi-classical FIO}
		In the situation of Definition \ref{def:z-FIO Szego type order m},
		for any $m\in\mathbb Z$, we let $\mathcal I^{-m}_{\Sigma,k}(\Omega;T^{*0,q}X)$ 
		be the set of all $k$-dependent continuous operators of the form
		$H_{(k)}: \cCc(\Omega,T^{*0,q}X)\to\cC(X,T^{*0,q}X)$ such that 
		the distribution kernel of $H_{(k)}$ satisfies 
		\begin{equation}
		H_{(k)}(x,y)=\int^{+\infty}_0 e^{ikt\psi_-(x,y)}h^-(x,y,t,k)dt
		+\int^{+\infty}_0 e^{ikt\psi_+(x,y)}h^+(x,y,t,k)dt
		+G_k(x,y),
		\end{equation}			
		where $G_k=O(k^{-\infty})~\text{on}~X\times \Omega$, $h^+(x,y,t,k)=0$ if $n_-\neq n_+$,  $h^\mp (x,y,t,k)$ is properly supported in $(x,y)$, $ h^\mp (x,y,t,k)\sim\sum_{j=0}^{+\infty}h_j^\mp(x,y,t,k)$ in $S^{n+1-m}_{\rm loc}\left(1;\Omega\times \Omega\times\R_+,\End(T^{*0,q}X)\right)$, $h_j^\mp (x,y,t,k)$ is properly supported in $(x,y)$ and in $S^{n+1-m-j}_{\rm loc}\left(1;\Omega\times \Omega\times\R_+,\End(T^{*0,q}X)\right)$ for each $j\in\N_0$, and $h_j^\mp(x,y,t,k_0)$ is in $S^{n-m-j}_{1,0}\left(\Omega\times \Omega\times\R_+,\End(T^{*0,q}X)\right)$ for each $k_0>0$ .
\end{definition}

Let us first prove the following result.	
\begin{lemma}
\label{lem:semi-classical integral H_(k,m)}
For $H_z\in S^{-N}_{\Sigma,z}(\Omega;T^{*0,q}X)$ in Definition \ref{def:z-FIO Szego type order m}, we actually have
		\begin{equation}
		\int_{\mathbb C}\frac{\partial\widetilde\chi(\frac{z}{k})}
		{\partial\overline z} H_z\frac{dz\wedge d\overline z}{2\pi i}\in\mathcal{I}^{-(N-1)}_{\Sigma,k}(\Omega;T^{*0,q}X).
		\end{equation}
\end{lemma}
\begin{proof}
Without loss of generality, we take $q=n_-\neq n_+$, and the case $n_-=n_+$ can be deduced from the same argument with some minor changes. By using integration by parts to the variable $t$ in the oscillatory integral $H_z(x,y)$, we can write 
		\begin{multline}
		\label{eq:HS formula for A_{k,m} minus}
		\int_{\mathbb C}\frac{\pr\Td\chi(\frac{z}{k})}{\pr\ol z}\left(\int_0^{+\infty}e^{it\psi_-(x,y)}\frac{\sum_{\alpha+\gamma\leq\beta}z^\gamma h_{\alpha,\gamma}^{-}(x,y,t)}{\left(z-t\right)^\beta}\tau(\varepsilon t) dt\right)\frac{dz\wedge d\overline z}{2\pi i}\\
		=\int_{\mathbb C}\frac{\pr\Td\chi(\frac{z}{k})}{\pr\ol z}
		\frac{(-1)^{\beta-1}}{(\beta-1)!}
		\int_0^{+\infty}\left(\frac{\pr }{\pr t}\right)^{\beta-1}\left(e^{it\psi_-(x,y)}\sum_{\alpha+\gamma\leq \beta}z^\gamma h_{\alpha,\gamma}^{-}(x,y,t)\tau(\varepsilon t)\right)\frac{1}{z-t} dt\frac{dz\wedge d\overline z}{2\pi i}.
		\end{multline}
		Then, by the oscillatory integral version of Fubini theorem, we can write the last integral by
		\begin{equation}
		\label{eq:partial inegration in t for H_{k} minus}
\int_0^{+\infty}\frac{(-1)^{\beta-1}}{(\beta-1)!}
		\int_{\mathbb C}\frac{\pr\Td\chi(\frac{z}{k})}{\pr\ol z}
		\left(\frac{\pr }{\pr t}\right)^{\beta-1}\left(e^{it\psi_-(x,y)}\sum_{\alpha+\gamma\leq \beta}z^\gamma h_{\alpha,\gamma}^{-}(x,y,t)\tau(\varepsilon t)\right)\frac{1}{z-t} \frac{dz\wedge d\overline z}{2\pi i}dt.
		\end{equation}
		By Cauchy--Pompeiu formula, we have
		\begin{multline}
		\int_{\mathbb C}\frac{\pr\Td\chi(\frac{z}{k})}{\pr\ol z}
		\left(\frac{\pr }{\pr t}\right)^{\beta-1}\left(e^{it\psi_-(x,y)}\sum_{\alpha+\gamma\leq \beta}\tau(\varepsilon t)h_{\alpha,\gamma}^{-}(x,y,t) z^\gamma\right)\frac{1}{z-t}\frac{dz\wedge d\overline z}{2\pi i}\\
		=\chi\left(\frac{t}{k}\right)\sum_{\alpha+\gamma\leq\beta}t^\gamma\left(\frac{\pr }{\pr t}\right)^{\beta-1}\left(e^{it\psi_-(x,y)}\tau(\varepsilon t)h_{\alpha,\gamma}^{-}(x,y,t)\right).
		\end{multline}
	By the above calculation and changing $k\mapsto kt$, we know \eqref{eq:partial inegration in t for H_{k} minus} equals to
		\begin{equation}
	\int_0^{+\infty}e^{ikt\psi_-(x,y)}\sum_{\alpha+\gamma\leq \beta-1}\frac{\tau(\varepsilon kt)h_{\alpha,\gamma}^{-}(x,y,kt)}{(\beta-1)!} k^{1+\gamma-(\beta-1)}\frac{\pr^{\beta-1}}{\pr t^{\beta-1}}(t^\gamma\chi(t)) dt.\label{eq:HS formula H_z minus}
		\end{equation}
		
By $h_{\alpha,\gamma}^{-}(x,y,t)\in S^{n-N+\alpha}_{\rm cl}\left(\Omega\times\Omega\times\R_+,\End(T^{*0,q}X)\right)$, we have the asymptotic expansion
\begin{multline}
      k^{1+\gamma-(\beta-1)}h_{\alpha,\gamma}^{-}(x,y,kt)\sim\sum_{j=0}^{+\infty} k^{n+1-(N-1)-j} h_{\alpha,\gamma,j}^{-}(x,y,t)\\
      \text{in $S^{n+1-(N-1)}_{\rm loc}\left(1;\Omega\times\Omega\times\R_+,\End(T^{*0,q}X)\right)$}.
\end{multline}	
		
We recall that $\varepsilon>0$ is a fixed number such that $\tau(\varepsilon t)\chi(t)=\chi(t)$ whenever $t\in\supp\chi$. By our assumption on $\chi$, when $k>0$ is large enough, we can see the products between $\tau(\varepsilon k t)$ and derivatives of $\chi(t)$ are always non-zero. Then for $t>0$, we also have $(\tau(\varepsilon t)-\tau(k\varepsilon t))\chi(t)\in S^{-\infty}_{\rm loc}(1;\R_+)$. By the definition of $\mathcal{I}^{-(N-1)}_{\Sigma,k}(\Omega;T^{*0,q}X)$, we hence complete the proof of our theorem.
\end{proof}
We need the following statement.
\begin{theorem}
\label{thm:negligible estimate for L_z}
For any $L_z\in L^{-\infty}_z(\Omega;T^{*0,q}X)$ in Definition \ref{def:L^-infty_z} we have
\begin{equation}
\label{eq:semi-classical integral of L_z}
    \int_\C\frac{\pr\Td\chi(\frac{z}{k})}{\pr\ol z}L_z\frac{dz\wedge d\ol z}{2\pi i}=\OK\ \ \text{on $X\times\Omega$}.
\end{equation}
\end{theorem}
\begin{proof}
First of all, for the case
\begin{equation}
L_z(x,y)=\int_0^{+\infty} e(x,y,t)\frac{1}{(z-t)^{M_1}}\tau(\varepsilon t)dt,
		\end{equation}	
 where $M_1\in\N$ and the symbol $e(x,y,t)\in 
		S^{-\infty}(\Omega\times \Omega\times{\R}_+,\End(T^{*0,q}X))$ is
		properly supported in the variables $(x,y)$,
from the proof of Lemma \ref{lem:semi-classical integral H_(k,m)} and especially the first line of \eqref{eq:HS formula H_z minus}, we can find some $E(x,y,t)\in S^{-\infty}(\Omega\times\Omega\times\R_+,\End(T^{*0,q}X))$ properly supported in the variables $(x,y)$ such that
\begin{equation}
     \int_\C\frac{\pr\Td\chi(\frac{z}{k})}{\pr\ol z}L_z(x,y)\frac{dz\wedge d\ol z}{2\pi i}=\int_0^{+\infty}E(x,y,t)\chi(k^{-1}t)dt=\OK\ \ \text{on $X\times\Omega$}.
\end{equation}
After applying some minor modification, this method also works for any $L_z\in\mathcal{E}_z(\Omega;T^{*0,q}X)$ and any $L_z\in\mathcal{F}_z(\Omega;T^{*0,q}X)$.

Second, we consider
\begin{equation}
\label{eq:special mathcal G_z}
		L_z(x,y)=\int_0^{+\infty} e^{it\psi(x,y)}g(x,y,t)
		\frac{1}{(z-t)^{M_1}}\tau(\varepsilon t)dt,
		\end{equation}
		where $g(x,y,t)=O(|x-y|^{+\infty})$, $g(x,y,t)\in S^m_{\rm cl}(\Omega\times \Omega\times{\R}_+,\End(T^{*0,q}X))$ for some $m\in\mathbb R$, 
		$g(x,y,t)$ is properly supported in the variables $(x,y)$, $M_1 \in\N_0$, 
		and $\psi\in {\Ph}(-\Lambda\bm\alpha,\Omega)$ for some $\Lambda\in\cC(X,{\R_+})$. Again by the first line of \eqref{eq:HS formula H_z minus}, we can can find a $G(x,y,t)=O(|x-y|^{+\infty})$ properly supported in the variables $(x,y)$ and $G(x,y,t)\in S^{m_1}_{\rm cl}(\Omega\times \Omega\times{\R}_+,\End(T^{*0,q}X))$ for some $m_1\in\R$ such that
\begin{equation}
     \int_\C\frac{\pr\Td\chi(\frac{z}{k})}{\pr\ol z}L_z(x,y)\frac{dz\wedge d\ol z}{2\pi i}=\int_0^{+\infty}e^{it\psi(x,y)}G(x,y,t)\chi(k^{-1}t)dt.
\end{equation}
We notice that for any point $p\in \Omega$, we have $\frac{\partial\psi}{\partial y_{2n+1}}(p,p)>0$ from our assumption. From the Malgrange preparation theorem, we can check that near $(p,p)$ we have
		\begin{equation}
		\label{eq:division of psi in special mathcal G_z}
		\psi(x,y)=f(x,y)(y_{2n+1}+\psi_0(x,y')),
		\end{equation}
		 where $\psi_0$ and $g$ are smooth functions near $(p,p)$, $f(p,p)>0$, $\Im\psi\geq 0$ around $(p,p)$ and $y':=(y_1,\cdots,y_{2n})$. 
		When $\Omega$ is small enough, we may assume that \eqref{eq:division of psi in special mathcal G_z} holds on $\Omega\times \Omega$ and as \eqref{eq:Im varphi>C|z-w|^2} we also have 
		\begin{equation}
		\label{eq: Im psi=O(|x'-y'|^2) in L_z}
		{\rm Im\,}\psi(x,y)\geq C|x'-y'|^2\ \ \mbox{on $\Omega\times \Omega$}, 
		\end{equation}
		where $C>0$ is a constant. We let $\widetilde G(x,y,t)$ be an almost analytic extension of $G(x,y,t)$ in the $y_{2n+1}$ variable. For every $N\in\N$, by using Taylor expansion at $\Td y_{2n+1}=-\psi_0(x,y')$, we have 
		\begin{multline}
  \label{eq:Taylor expansion of G(x,y,t)}
		G(x,y,t)=\widetilde G(x,(y',\Td y_{2n+1}),t)|_{\Td y_{2n+1}=y_{2n+1}}
		=\sum^N_{j=0}G_j(x,y',t)(y_{2n+1}+\psi_0(x,y'))^j\\
  +(y_{2n+1}+\psi_0(x,y'))^{N+1}R_{N+1}(x,y,t),
		\end{multline}
		where $G_j(x,y',t)\in S^{m_1}_{{\rm cl\,}}(\Omega\times \Omega\times\mathbb R_+,\End(T^{*0,q}X))$, $j=1,\cdots,N$, and $R_{N+1}(x,y',t)\in S^{m_1}_{{\rm cl\,}}(\Omega\times \Omega\times\mathbb R_+,\End(T^{*0,q}X))$. On one hand, since $G(x,y,t)=O(|x-y|^{+\infty})$, by taking $(N+1)$-times derivatives of $y_{2n+1}$ in \eqref{eq:Taylor expansion of G(x,y,t)} we can first check that 
  \begin{equation}
      R_{N+1}(x,y,t)=O(|x-y|^{+\infty}).
  \end{equation}
  Then similarly we can check that
            \begin{equation}
		\label{eq:G_j(x,y,t) in mathcal G_z}
		G_j(x,y',t)=O(|x'-y'|^{+\infty}),\ \ j=N,N-1,\cdots,0.
		\end{equation}
  We then let
		\begin{equation}
		\mathds G_N(x,y,t):=e^{it\psi(x,y)}\sum^N_{j=0}(y_{2n+1}+\psi_0(x,y'))^j G_j(x,y',t),    
		\end{equation}
		and consider the operator $\mathds G_{(k,N)}$ defined by kernel
		\begin{equation}
		\mathds G_{(k,N)}(x,y):=\int_0^{+\infty} \mathds G_N(x,y,t)\chi(k^{-1}t)dt.
		\end{equation}
		From \eqref{eq: Im psi=O(|x'-y'|^2) in L_z} and \eqref{eq:G_j(x,y,t) in mathcal G_z}, we can check that 
		\begin{align}
		&e^{it\psi(x,y)}G_j(x,y',t)\in S^{-\infty}(\Omega\times \Omega\times\mathbb R_+,\End(T^{*0,q}X)),~j=1,\cdots,N,\\
		&\mathds G_N(x,y,t)\in S^{-\infty}(\Omega\times \Omega\times\mathbb R_+,\End(T^{*0,q}X)).
		\end{align}
		Then by the result we just have proved in the first step we have
		\begin{equation}
		\mathds G_{(k,N)}=O(k^{-\infty})\ \ \text{on $X\times\Omega$}.
		\end{equation}
Also, we notice that by $(N+1)$-times of partial integration we have
\begin{multline}
    \int_0^{+\infty}e^{it\psi(x,y)}(y_{2n+1}+\psi_0(x,y'))^{N+1}R_{N+1}(x,y,t)\chi(k^{-1}t)dt\\
    = \sum_{j=0}^{N+1}\int_0^{+\infty}e^{it\psi(x,y)}\gamma_{N+1,j}(x,y,t)k^{-(N+1-j)}\frac{\partial^{N+1-j}\chi}{\partial t}(k^{-1}t)dt,
\end{multline}
where $\gamma_{N+1,j}(x,y,t)\in S^{m-N-1+j}_{\rm cl}(\Omega\times\Omega\times\R_+,\End (T^{*0,q}X))$ for each $j=0,\cdots,N+1$. By taking the Taylor expansion \eqref{eq:Taylor expansion of G(x,y,t)} of $G(x,y,t)$ to arbitrary high order $N$ and by the condition that $G(x,y,t)$ is properly supported in $(x,y)$, the above arguments imply that for $L_z$ in the form of \eqref{eq:special mathcal G_z} we have
\begin{equation}
    \int_\C \frac{\pr\Td\chi(\frac{z}{k})}{\pr\ol z}L_z(x,y)\frac{dz\wedge d\ol z}{2\pi i}=\OK\ \ \text{on $X\times\Omega$}.
\end{equation}
 With some minor change of the argument we just used, this method also works for any $L_z\in\mathcal{G}_z(\Omega;T^{*0,q}X)$.

Finally, we notice that for $L_z\in \mathcal{R}_z(\Omega;T^{*0,q}X)$ of the form
		\begin{multline}
		L_z(x,y)=\int_0^{+\infty}\int_0^{+\infty}\int_\Omega e^{it\psi_-(x,w)+i\sigma\psi_+(w,y)}r_1(x,w,t)\circ r_2(w,y,\sigma)
		\frac{z^{M_2}}{(z-t)^{M_1}}\tau(\varepsilon t)m(w) dw d\sigma dt,
		\end{multline}
where  $\psi_\mp\in {\Ph}(\mp\Lambda\bm\alpha,\Omega)$ for some $\Lambda\in\cC(X,{\R_+})$ and $r_1,r_2$ are H\"ormander symbols, by the properties that $\psi(x,w)=0$ when $x=w$, $\psi(w,y)=0$ when $w=y$, $d_w\psi_-(x,w)=d_w\psi_+(w,y)$ at $w=x=y$, $t\geq 0$ and $\sigma\geq 0$, we have the following observation: given a suitably small $\delta>0$, when $|x-w|>\delta$ we can apply arbitrary times of partial integration in $t$;  when $|w-y|>\delta$ we can apply arbitrary times of partial integration in $\sigma$; when $|x-y|<2\delta$ we can apply arbitrary times of partial integration in $w$. Then by this observation and the proof of the previous lemma we can check that 
\begin{equation}
    \int_\C \frac{\pr\Td\chi(\frac{z}{k})}{\pr\ol z}L_z(x,y)\frac{dz\wedge d\ol z}{2\pi i}=\OK\ \ \text{ on $X\times\Omega$},
\end{equation}
and with some minor changes the same argument also works for general $L_z\in\mathcal{R}_z(\Omega;T^{*0,q}X)$.
\end{proof}
The next theorem follows directly from  Theorems  \ref{thm:leading FIO of (z-T_P)^-1 Pi I}, \ref{thm:expansion of (z-T_P)^-1 Pi} and \ref{thm:negligible estimate for L_z}.
	\begin{theorem}
		\label{thm:principal term of chi_k(T_P) I}
		With the same notations and assumptions in Theorem \ref{thm:expansion of (z-T_P)^-1 Pi}, for any $m\in\N_0$ we have
  \begin{equation}
     \int_{\mathbb C}\frac{\partial\widetilde\chi_k}
		{\partial\overline z}\mathscr A_{z,m}\frac{dz\wedge d\overline z}{2\pi i}  :=\int_{\mathbb C}\frac{\partial\widetilde\chi_k}
		{\partial\overline z}\mathscr (S_-+S_+)\circ A_{z,m}\circ (S_-+S_+)\frac{dz\wedge d\overline z}{2\pi i}=A_{(k,m)}\in \mathcal{I}^{-m}_{\Sigma,k}(\Omega;T^{*0,q}X),
  \end{equation}
  and up to an $k$-negligible kernel on $X\times\Omega$ we have
		\begin{equation}
		A_{(k,m)}(x,y)\equiv\int^{+\infty}_0 e^{ikt\Psi_-(x,y)}a^{-,m}(x,y,t,k)dt
		+\int^{+\infty}_0 e^{ikt\Psi_+(x,y)}a^{+,m}(x,y,t,k)dt,
		\end{equation} 
	where
		\begin{align}
		&\Psi_-\in {\Ph}(p_{I_0,I_0}^{-1}(-\bm\alpha)(-\bm\alpha),\Omega),\ \ \Psi_+\in {\Ph}(p_{J_0,J_0}^{-1}(-\bm\alpha)\bm\alpha,\Omega),
		\end{align}
and we have the following data properly supported in $(x,y)$:
\begin{equation}
		\mbox{$a^{\mp,m}(x,y,t,k)\sim\sum^{+\infty}_{j=0}a^{\mp,m}_j(x,y,t)k^{n+1-m-j}$ }
  \mbox{in $S^{n+1-m}_{{\rm loc\,}}(1;\Omega\times \Omega\times\mathbb R_+,\End(T^{*0,q}X))$},
		\end{equation}
		\begin{align}
  &\forall j\in\N_0,~a^{\mp,m}_j(x,y,t)\neq 0\implies t\in\supp\chi;\ \ 
a^{\mp,m}(x,y,t)\neq 0\implies t\in\supp\chi\ ,
		\end{align}
		and
		\begin{equation}
		a^+(x,y,t,k)=0~\text{when}~n_-\neq n_+\ .
		\end{equation}
  Moreover, for $m=0$ we have
\begin{equation}
    a^{-,0}_0(x,x,t)=\frac{|\det\mathcal{L}_x|}{2\pi^{n+1}}\frac{v(x)}{m(x)}~p^{-n-1}_{I_0,I_0}(-\bm\alpha_x)~\chi(t)t^n,
\end{equation}
  and when $n_-=n_+$ we also have
		\begin{align}
		&a^{+,0}_0(x,x,t)=\frac{|\det\mathcal{L}_x|}{2\pi^{n+1}}\frac{v(x)}{m(x)}~p^{-n-1}_{J_0,J_0}(-\bm\alpha_x)~\chi(-t)t^n.
		\end{align}
	\end{theorem}
	From Theorem \ref{thm:principal term of chi_k(T_P) I}, now we have
	\begin{align}
	\chi(k^{-1}T_{P,\lambda}^{(q)})
	=\sum_{m=0}^N A_{(k,m)}+R_{(k,N+1)}+F_{(k,N+1)}\label{eq:operator level expansion of chi_k(T_P)},
	\end{align}
	where $R_{z,N+1}$ is as described in Theorem \ref{thm:expansion of (z-T_P)^-1 Pi} and $F_{z,N+1}\in L^{-\infty}_z(\Omega;T^{*0,q}X)$. We define
	\begin{align}
	&R_{(k,N+1)}:=\int_{\mathbb C}\frac{\partial\widetilde\chi_k}
	{\partial\overline z}(z-T_{P,\lambda}^{(q)})^{-1}\circ R_{z,N+1}\frac{dz\wedge d\overline z}{2\pi i},\\
	&F_{(k,N+1)}:=\int_{\mathbb C}\frac{\partial\widetilde\chi_k}
	{\partial\overline z}(z-T_{P,\lambda}^{(q)})^{-1}\circ F_{z,N+1}\frac{dz\wedge d\overline z}{2\pi i}.
	\end{align}
	We are going to show that for any $N\in\N_0$ we have 
 \begin{align}
	&     F_{(k,N_0+1)}=O(k^{-N})~\text{in}~\mathscr L(H^{-N}_{\rm comp}(\Omega,T^{*0,q}X),H^{N}(X,T^{*0,q}X))
	\end{align}
  and for any $N_1,N_2\in\N_0$ we can find an $N_0>0$ large enough such that
	\begin{align}
 \label{eq:operator level R_{(k,N_0+1)}}
	&     R_{(k,N_0+1)}=O(k^{-N_1})~\text{in}~\mathscr L(H^{-N_2}_{\rm comp}(\Omega,T^{*0,q}X),H^{N_2}(X,T^{*0,q}X)).
	\end{align}
	
	To proceed further, we need the following resolvent estimate.
	\begin{theorem}
		\label{thm:resolvent estimate for T_P}
		In the situation of Theorem \ref{thm:Main Semi-Classical Expansion}, for $q=n_-$ and any $s\in\N_0$ we have
		\begin{equation}
		\Pi_\lambda^{(q)}\circ(z-T_{P,\lambda}^{(q)})^{-1}=O\left(\frac{|z|^s}{|\im z|}\right)~\text{in}~\mathscr L(H^s(X,T^{*0,q}X),H^{s+1}(X,T^{*0,q}X)).
		\end{equation}
	\end{theorem}
	\begin{proof}
By Theorem \ref{thm:parametetrix of Toeplitz operators for lower energy forms}, we have $Q\in L^{-1}_{\rm cl}(X;T^{*0,q}X)$ and $F\in L^{-\infty}(X;T^{*0,q}X)$ such that
		\begin{equation}
		T_{Q,\lambda}^{(q)}\circ (z-T_{P,\lambda}^{(q)})=z T_{Q,\lambda}^{(q)}-\Pi_\lambda^{(q)}+F,
		\end{equation}
		where $T_{Q,\lambda}^{(q)}:=\Pi_\lambda^{(q)}\circ Q\circ \Pi_\lambda^{(q)}$. This implies that
		\begin{equation}
		\label{eq:Pi circ (z-T_P)^-1}
		\Pi_\lambda^{(q)}\circ(z-T_{P,\lambda}^{(q)})^{-1}=-T_{Q,\lambda}^{(q)}+z T_{Q,\lambda}^{(q)}\circ (z-T_{P,\lambda}^{(q)})^{-1}+F\circ (z-T_{P,\lambda}^{(q)})^{-1}.
		\end{equation}
		From  Theorem \ref{thm:T_P is self-adjoint} and the spectral theory of self-adjoint operators, we have
		\begin{equation}
		(z-T_{P,\lambda}^{(q)})^{-1}=O\left(\frac{1}{|\im z|}\right)~\text{in}~\mathscr L(L^2(X,T^{*0,q}X),L^2(X,T^{*0,q}X)).
		\end{equation}
		By the above estimate, \eqref{eq:Sobolev boundedness of T_P} and \eqref{eq:Pi circ (z-T_P)^-1}, we immediately have
		\begin{equation}
		\Pi_\lambda^{(q)}\circ(z-T_{P,\lambda}^{(q)})^{-1}=O\left(\frac{|z|}{|\im z|}\right)~\text{in}~\mathscr L(L^2(X,T^{*0,q}X),H^1(X,T^{*0,q}X)).
		\end{equation}
		We can put this estimate back to \eqref{eq:Pi circ (z-T_P)^-1}, then by $T_{Q,\lambda}^{(q)}\circ (z-T_{P,\lambda}^{(q)})^{-1}=T_{Q,\lambda}^{(q)}\circ \Pi_\lambda^{(q)}\circ(z-T_{P,\lambda}^{(q)})^{-1}$ and the same argument and estimate we just used, we have
		\begin{equation}
		\Pi_\lambda^{(q)}\circ(z-T_{P,\lambda}^{(q)})^{-1}=O\left(\frac{|z|^2}{|\im z|}\right)~\text{in}~\mathscr L(H^1(X,T^{*0,q}X),H^2(X,T^{*0,q}X)).
		\end{equation}
		We can hence bootstrap and get our theorem.
	\end{proof}
 Now we can prove the following.
	\begin{theorem}
		\label{thm:HS formula for mathcal E_z is OK}
		For any operator $E_z\in\mathcal{E}_z(\Omega;T^{*0,q}X)$ of Definition \ref{def:z-depnednt smoothing operator type 1} and $N\in\N_0$ we have
		\begin{equation}
\int_{\mathbb C}\frac{\partial\widetilde\chi_k}
{\partial\overline z}(z-T_{P,\lambda}^{(q)})^{-1}\circ E_{z}\frac{dz\wedge d\overline z}{2\pi i}
=O(k^{-N})~\text{in}~\mathscr L\left(H^{-N}_{\rm comp}(\Omega,T^{*0,q}X),H^{N}(X,T^{*0,q}X)\right).
		\end{equation}
	\end{theorem}
	\begin{proof}
		For simplicity, we assume that the kernel of $E_z$ is given by 
		\begin{equation}
		E_z(x,y)=\int_0^{+\infty} e(x,y,t)
		\frac{z^{M_2}}{(z-t)^{M_1}}\tau(\varepsilon t)dt,
		\end{equation}	
	where $e(x,y,t)\in 
		S^{-\infty}(\Omega\times \Omega\times{\R}_+,\End(T^{*0,q}X)$
		properly supported in the variables $(x,y)$, and $M_1,M_2\in\N_0$. The general situation can also be deduced from some modification of the argument below. 
		
		When $z\notin\Spec(T_{P,\lambda}^{(q)})\setminus\{0\}$, we have
        \begin{equation}
            (z-T_{P,\lambda}^{(q)})^{-1}=z^{-M}{T_{P,\lambda}^{(q),M}}\circ(z-T_{P,\lambda}^{(q)})^{-1}+\sum_{j=0}^{M-1}z^{-1-j}{T_{P,\lambda}^{(q),j}},
        \end{equation}
		where $M\in\N$ is arbitrary, $T_{P,\lambda}^{(q),j}:=T_{P,\lambda}^{(q)}\circ\cdots\circ T_{P,\lambda}^{(q)}$	is the $j$-times composition between $T_{P,\lambda}^{(q)}$ and $T_{P,\lambda}^{(q),0}:=I$. We notice that $z\neq 0$ when $z\in\supp\Td\chi(\frac{\cdot}{k})$. We also notice that from $\Pi_\lambda^{(q)}\circ\Pi_\lambda^{(q)}=\Pi_\lambda^{(q)}$ and $[\Pi_\lambda^{(q)},T_{P,\lambda}^{(q)}]=0$ we have $\left[\Pi_\lambda^{(q)},(z-T_{P,\lambda}^{(q)})^{-1}\right]=0$.
	
Then, on one hand, for the integral
		\begin{equation}
\int_{\mathbb C}\frac{\partial\widetilde\chi(\frac{z}{k})}
{\partial\overline z}z^{-1-j}~T_{P,\lambda}^{(q),j}\circ E_z\frac{dz\wedge d\overline z}{2\pi i}=T_{P,\lambda}^{(q),j}\circ\int_{\mathbb C}\frac{\partial\widetilde\chi(\frac{z}{k})}
{\partial\overline z}z^{-1-j}~ E_z\frac{dz\wedge d\overline z}{2\pi i},
		\end{equation}
		by Fubini theorem and $(M_1-1)$-times of partial integration to $t$ in the sense of oscillatory integrals, we have
		\begin{multline}
\int_{\mathbb C}\frac{\partial\widetilde\chi(\frac{z}{k})}
{\partial\overline z}z^{-1-j}~ E_z(x,y)\frac{dz\wedge d\overline z}{2\pi i}\\
=\int_{\mathbb C}\frac{\partial\widetilde\chi(\frac{z}{k})}
{\partial\overline z}z^{-1-j}\int_0^{+\infty} e(x,y,t)\frac{z^{M_2}}{(z-t)^{M_1}} \tau(\varepsilon t)dt\frac{dz\wedge d\overline z}{2\pi i}\\
=\int_0^{+\infty}\int_{\mathbb C}\frac{\partial}
{\partial\overline z}\left(\widetilde\chi\left(\frac{z}{k}\right)z^{-1-j+M_2}\right)(z-t)^{-1}\frac{dz\wedge d\overline z}{2\pi i}\delta(x,y,t)dt,
		\end{multline}
	where $\delta(x,y,t)$ is in $S^{-\infty}(\Omega\times \Omega\times{\R}_+,\End(T^{*0,q}X))$ and properly supported in $(x,y)$, and $\delta(x,y,t)\neq 0$ if and only if $t>\varepsilon$. So we can apply Cauchy--Pompeiu formula and get
	\begin{multline}
	\int_0^{+\infty}\int_{\mathbb C}\frac{\partial}
{\partial\overline z}\left(\widetilde\chi\left(\frac{z}{k}\right)z^{-1-j+M_2}\right)(z-t)^{-1}\frac{dz\wedge d\overline z}{2\pi i}\delta(x,y,t)dt\\
=\int_0^{+\infty}\chi\left(\frac{t}{k}\right)t^{-1-j+M_2}\delta(x,y,t)dt
=\OK.
	\end{multline}
By the Sobolev-boundedness of $T_{P,\lambda}^{(q)}$, we know that this part of integral satisfies the estimate we want.
	
On the other hand, for arbitrary $M\in\N_0$ such that $M\equiv 0\mod 4$, we have the integral
		\begin{multline}
\int_{\mathbb C}\frac{\partial\widetilde\chi(\frac{z}{k})}
{\partial\overline z} z^{-M}~T_{P,\lambda}^{(q),M}\circ(z-T_{P,\lambda}^{(q)})^{-1} \circ E_z\frac{dz\wedge d\overline z}{2\pi i}\\
=T_{P,\lambda}^{(q),\frac{M}{2}}\circ\int_{\mathbb C}\frac{\partial\widetilde\chi(\frac{z}{k})}
{\partial\overline z}z^{-M}~\Pi_{\lambda}^{(q)}\circ(z-T_{P,\lambda}^{(q)})^{-1}\circ T_{P,\lambda}^{(q),\frac{M}{2}} \circ E_z\frac{dz\wedge d\overline z}{2\pi i}.
\end{multline}
By the Sobolev estimate of $T_{P,\lambda}^{(q)}$, the estimate that $|\frac{\partial\widetilde\chi(\frac{z}{k})}
{\partial\overline z}|=O(k^{-\infty}|\im z|^{+\infty})$, the estimate of Theorem \ref{thm:resolvent estimate for T_P}, again the Sobolev estimate of $T_{P,\lambda}^{(q)}$, and the direct estimate of $E_z$, then for any $M\in\N_0$
 we have
\begin{multline}
T_{P,\lambda}^{(q),\frac{M}{2}}\circ\int_{\mathbb C}\frac{\partial\widetilde\chi(\frac{z}{k})}
{\partial\overline z}z^{-M}~\Pi_{\lambda}^{(q)}\circ(z-T_{P,\lambda}^{(q)})^{-1}\circ T_{P,\lambda}^{(q),\frac{M}{2}} \circ E_z\frac{dz\wedge d\overline z}{2\pi i}\\
=O\left(\sup_{k^{-1}z\in \supp\Td\chi} k^2\cdot \frac{|\im z|^{1+M_1}}{k^{1+M_1}}\cdot |z|^{-M}\cdot\frac{|z|^{\frac{M}{2}+\frac{M}{4}-1}}{|\im z|}\cdot\frac{|z|^{M_2}}{|\im z|^{M_1}}\right)\\
=O(k^{-\frac{M}{4}+M_2-M_1})~\text{in}~\mathscr L\left(H^{-\frac{M}{4}}_{\rm comp}(\Omega,T^{*0,q}X),H^{\frac{M}{4}}(X,T^{*0,q}X)\right).
\end{multline}
Combining all the estimates above we complete the proof.
	\end{proof}
We would like to note that during the proof of the previous theorem, the step where we split $T_{P,
 \lambda}^{M}$ into $T_{P,\lambda}^{\frac{M}{2}}\circ T_{P,\lambda}^{\frac{M}{2}}$ is crucial. This step is designed to prevent the argument from breaking down when we apply Theorem \ref{thm:resolvent estimate for T_P}. Specifically, it helps us avoid a situation where the term $|z|^s$ contributes an excessive power of $k$, which can occur when $s$ is too large.\
 
 With the same proof, we also have the following.
	\begin{theorem}
        \label{thm:HS formula for mathcal F_z is OK}
		For any operator  $F_z\in\mathcal F_z(\Omega;T^{*0,q}X)$ of Definition \ref{def:z-depnednt smoothing operator type 2} and $N\in\N_0$ we have
		\begin{equation}
		\int_{\mathbb C}\frac{\partial\widetilde\chi_k}
		{\partial\overline z}(z-T_{P,\lambda}^{(q)})^{-1}\circ F_{z}\frac{dz\wedge d\overline z}{2\pi i}
		=O(k^{-N})~\text{in}~\mathscr L\left(H^{-N}_{\rm comp}(\Omega,T^{*0,q}X),H^{N}(X,T^{*0,q}X)\right).
		\end{equation}
	\end{theorem}
\begin{theorem}
    \label{thm:HS formula for mathcal R_z is OK}
    	For any operator $\mathscr R_z\in\mathcal{R}_z(\Omega;T^{*0,q}X)$ of Definition \ref{def:z-depnednt smoothing operator type 4} and $N\in\N_0$ we have
		\begin{equation}
		\int_\C\frac{\partial\widetilde\chi_k}
		{\partial\overline z}(z-T_{P,\lambda}^{(q)})^{-1}\circ \mathscr R_z \frac{dz\wedge d\overline z}{2\pi i}
		=O(k^{-N})~\text{in}~\mathscr L\left(H^{-N}_{\rm comp}(\Omega,T^{*0,q}X),H^{N}(X,T^{*0,q}X)\right).
		\end{equation}
\end{theorem}
\begin{proof}
For $\mathscr R_z$ in the form of \eqref{eq:mathcal R_z} and a very small $\epsilon>0$, we notice that for the function $it\psi_\mp(x,w)+i\sigma\psi_\pm(w,y)$, when $|x-w|>\epsilon$, $|w-y|>\epsilon$ and $|x-w|,|w-y|<\epsilon$, we can apply arbitrary times of partial integration in $t$, $\sigma$ and $w$, respectively. Along with the elementary estimate that $|z-t|^{-1}\leq |z|\cdot|\im z|^{-1}t^{-1}$ when $t>0$, we can estimate $\mathscr R_z$ and we can
apply the same argument in Theorem \ref{thm:HS formula for mathcal E_z is OK} to get our theorem.
\end{proof}
 The next kind of remainder estimate needs more work.
	\begin{theorem}
		\label{thm:HS formula for mathcal G_z is OK}
		For any operator $G_z\in\mathcal{G}_z(\Omega;T^{*0,q}X)$ of Definition \ref{def:z-depnednt smoothing operator type 3} and $N\in\N_0$, we have
		\begin{equation}
		\int_\C\frac{\partial\widetilde\chi_k}
		{\partial\overline z}(z-T_{P,\lambda}^{(q)})^{-1}\circ G_z \frac{dz\wedge d\overline z}{2\pi i}
		=O(k^{-N})~\text{in}~\mathscr L\left(H^{-N}_{\rm comp}(\Omega,T^{*0,q}X),H^{N}(X,T^{*0,q}X)\right).
		\end{equation}
	\end{theorem}
	\begin{proof}
	For simplicity, we prove the case for $q=n_-=0$ and
		\begin{equation}
		G_z(x,y)=\int_0^{+\infty} e^{it\psi(x,y)}g(x,y,t)
		\frac{z^{M_2}}{(z-t)^{M_1}}\tau(\varepsilon t)dt,
		\end{equation}
		where $g(x,y,t)=O(|x-y|^{+\infty})$ is in $S^m_{\rm cl}(\Omega\times \Omega\times{\R}_+,\End(T^{*0,q}X))$ for some $m\in\mathbb R$, $g(x,y,t)$ is properly supported in the variables $(x,y)$, 
		$M_1,M_2\in\N_0$, 
		and $\psi\in {\Ph}(-\Lambda\bm\alpha,\Omega)$ for some $\Lambda\in\cC(X,\R_+)$. The general situation can be deduced from some modification of the following argument. 
		
		As in the proof of Theorem \ref{thm:negligible estimate for L_z}, we may assume that
		\begin{equation}
		\label{eq:division of psi in mathcal G_z}
		\psi(x,y)=f(x,y)(y_{2n+1}+\psi_0(x,y'))\ \ \text{on $\Omega\times\Omega$}.
		\end{equation}
	Also, as \eqref{eq:Im varphi>C|z-w|^2} we may assume that 
		\begin{equation}
		\label{eq: Im psi=O(|x'-y'|^2) in mathcal G_z}
		{\rm Im\,}\psi(x,y)\geq C|x'-y'|^2\ \ \mbox{on $\Omega\times \Omega$}, 
		\end{equation}
		where $C>0$ is a constant. We let $\widetilde g(x,y,t)$ be an almost analytic extension of $g(x,y,t)$ in the $y_{2n+1}$ variables. For every $N\in\N$, by Taylor expansion we have 
		\begin{equation}
		g(x,y,t)
		=\sum^N_{j=0}g_j(x,y',t)(y_{2n+1}+\psi_0(x,y'))^j+(y_{2n+1}+\psi_0(x,y'))^{N+1}r_{N+1}(x,y,t),
		\end{equation}
		where $g_j(x,y',t),r_{N+1}(x,y',t)\in S^m_{{\rm cl\,}}(\Omega\times \Omega\times\mathbb R_+,\End(T^{*0,q}X))$, $j=1,\cdots,N$.
		
		On one hand, since $g(x,y,t)=O(|x-y|^{+\infty})$, we can check that 
\begin{equation}
    r_{N+1}(x,y,t)=O(|x-y|^{+\infty})~\text{and}~g_j(x,y',t)=O(|x'-y'|^{+\infty}),\ \ j=N,N-1,\cdots,0.\label{eq:g_j(x,y,t) in mathcal G_z}
\end{equation}
	From \eqref{eq: Im psi=O(|x'-y'|^2) in mathcal G_z} and \eqref{eq:g_j(x,y,t) in mathcal G_z}, we can check that 
		\begin{equation}
		\mathds G_N(x,y,t):=e^{it\psi(x,y)}\sum^N_{j=0}(y_{2n+1}+\psi_0(x,y'))^j g_j(x,y',t)\in S^{-\infty}(\Omega\times \Omega\times\mathbb R_+,\End(T^{*0,q}X)).
		\end{equation}
	We consider the operator $G_{z,N}$ by kernel
		\begin{equation}
		\mathds G_{z,N}(x,y):=\int^{+\infty}_0 e^{it\psi(x,y)}\mathds G_N(x,y,t)\frac{z^{M_2}}{(z-t)^{M_1}}\tau(\varepsilon t)dt.
		\end{equation}
Then we have $\mathds G_{z,N}\in\mathcal{E}_z(\Omega;T^{*0,q}X)$, and Theorem~\ref{thm:HS formula for mathcal E_z is OK} implies that
		\begin{equation}
		\label{eq:HS formula for G_N in mathcal G_z}
		\int\frac{\partial\widetilde\chi(\frac{z}{k})}{\partial\overline z}(z-T_{P,\lambda}^{(q)})^{-1}
		\circ\mathds G_{z,N}\frac{dz\wedge d\overline z}{2\pi i}
		=O(k^{-\infty})\ \ \text{on $X\times\Omega$}.
		\end{equation}
		
		On the other hand, for the operator $\zeta_{z,N}$ associated by the kernel
		\begin{equation}
		\label{eq:zeta_z,N(x,y) I}
		\zeta_{z,N}(x,y):=
		\int^{+\infty}_0 e^{it\psi(x,y)}(y_{2n+1}+\psi_0(x,y'))^{N+1}r_{N+1}(x,y,t)\frac{z^{M_2}}{(z-t)^{M_1}}\tau(\varepsilon t)dt,
		\end{equation}
		by integration by parts in $t$, we can also write
		\begin{equation}
		\label{eq:zeta_z,N(x,y) II}
		\zeta_{z,N}(x,y)=\int_0^{+\infty}e^{it\psi(x,y)}z^{M_2}\frac{\pr^{N+1}}{\pr t^{N+1}}\left((z-t)^{-M_1}\cdot r^f_{N+1}(x,y,t)\cdot\tau(\varepsilon t)\right)dt,
		\end{equation}
		where we have $r^f_{N+1}(x,y,t)\in  S^m_{{\rm cl\,}}(\Omega\times \Omega\times\mathbb R_+,\End(T^{*0,q}X))$. Now, for the operator
		\begin{equation}
		\zeta_{(k,N)}:=\int_{\C\setminus \Spec(T_{P,\lambda}^{(q)})}\frac{\partial\widetilde\chi(\frac{z}{k})}
		{\partial\overline z}(z-T_{P,\lambda}^{(q)})^{-1}\circ\zeta_{z,N} \frac{dz\wedge d\overline z}{2\pi i},
		\end{equation}
		we recall that when $z\notin\Spec(T_{P,\lambda}^{(q)})$ and for any $M\in\N$ we can write
		\begin{multline}
		\label{eq:HS formula zeta_(k,n)}
		\zeta_{(k,N)}=\int_{\C\setminus \Spec(T_{P,\lambda}^{(q)})}\frac{\partial\widetilde\chi(\frac{z}{k})}
		{\partial\overline z}\frac{T_{P,\lambda}^{(q),M}}{z^M}\circ(z-T_{P,\lambda}^{(q)})^{-1}\circ\zeta_{z,N} \frac{dz\wedge d\overline z}{2\pi i}\\
		+\sum_{j=0}^{M-1}{T_{P,\lambda}^{(q),j}}\circ\int_{\C\setminus \Spec(T_{P,\lambda}^{(q)})}\frac{\partial\widetilde\chi(\frac{z}{k})}
		{\partial\overline z} {z^{-1-j}}~\zeta_{z,N} \frac{dz\wedge d\overline z}{2\pi i}.
		\end{multline}
By \eqref{eq:g_j(x,y,t) in mathcal G_z} and Theorem \ref{thm:negligible estimate for L_z}, for all $M\in\N$ we have
\begin{equation}
    \sum_{j=0}^{M-1}{T_{P,\lambda}^{(q),j}}\circ\int_{\C\setminus \Spec(T_{P,\lambda}^{(q)})}\frac{\partial\widetilde\chi(\frac{z}{k})}
		{\partial\overline z} {z^{-1-j}}~\zeta_{z,N} \frac{dz\wedge d\overline z}{2\pi i}=\OK\ \ \text{on $X\times\Omega$}.
\end{equation}

  It remains to handle the estimate for the first part of the integral in \eqref{eq:HS formula zeta_(k,n)} for large $N$. For this purpose we need to choose some suitably large number $M$ in \eqref{eq:HS formula zeta_(k,n)}. When $N\in\N$ is large enough and $N+M_1-m\equiv 0\mod 4$, we take $2M:={N+M_1-m}$. Then, from the formula of $\zeta_{z,N}$ in \eqref{eq:zeta_z,N(x,y) II}, the observation that $\frac{\pr}{\pr t}\tau(\varepsilon t)=O(t^{-\infty})$, and the elementary estimate $|z-t|^{-1}\leq |t\im z|^{-1}|z|$ in our situation, up to a kernel associated by an element in $\mathcal{E}_z(\Omega;T^{*0,q}X)$ we can write 
		\begin{equation}
		\zeta_{z,N}(x,y)=\int_0^{+\infty}e^{it\psi(x,y)}R^f_{N+1}(x,y,t,z)\tau(\varepsilon t)dt,
		\end{equation}
		where for all multi-indices $\alpha,\beta,\gamma$ and for all $x,y\in K\Subset\Omega$ and $t\in\R_+$ such that $\tau(\varepsilon t)>0$, we have some constant $ c_{K,\alpha,\beta,\gamma}>0$ such that
		\begin{equation}
		\label{eq:zeta_z,N(x,y) III}
		|\pr_x^\alpha\pr_y^\beta\pr_t^\gamma R^f_{N+1}(x,y,t,z)|\leq c_{K,\alpha,\beta,\gamma} \frac{|z|^{M_1+M_2+(N+1)}}{|\im z|^{M_1+(N+1)}}t^{m-M_1-(N+1)}.
		\end{equation}
		
Then, by combining all the above estimates and Theorem \ref{thm:HS formula for mathcal E_z is OK}, for any $N\in\N$ which is arbitrarily large enough such that $N+M_1-m\equiv 0\mod 4$, then the number $2M:={N+M_1-m}$ is consequently  arbitrarily large and we have
		\begin{align}
		&\int_{\C\setminus \Spec(T_{P,\lambda}^{(q)})}\frac{\partial\widetilde\chi(\frac{z}{k})}
		{\partial\overline z}z^{-M}~T_{P,\lambda}^{(q),M}\circ(z-T_{P,\lambda}^{(q)})^{-1}\circ\zeta_{z,N}\frac{dz\wedge d\overline z}{2\pi i}\nonumber\\
		=&\int_{\C\setminus \Spec(T_{P,\lambda}^{(q)})}\frac{\partial\widetilde\chi(\frac{z}{k})}
		{\partial\overline z}z^{-M}~T_{P,\lambda}^{(q),\frac{M}{2}}\circ\left(\Pi_\lambda^{(q)}\circ (z-T_{P,\lambda}^{(q)})^{-1}\right)\circ T_{P,\lambda}^{(q),\frac{M}{2}}\circ\zeta_{z,N}\frac{dz\wedge d\overline z}{2\pi i}\nonumber\\
		=&O\left(\sup_{k^{-1}z\in \supp\Td\chi}
		k^2\cdot\frac{|\im z|^{1+M_1+(N+1)}}{k^{1+M_1+(N+1)}}\cdot k^{-M}\cdot \frac{|z|^{
				\frac{M}{4}+\frac{M}{2}-1}}{|\im z|}\cdot\frac{|z|^{M_1+M_2+(N+1)}}{|\im z|^{M_1+(N+1)}}\right)\nonumber\\
		=&O(k^{-\frac{M}{4}+M_2})\ \ \text{in $\mathscr L\left(H^{-\frac{3M}{4}}_{\rm comp}(\Omega,T^{*0,q}X),H^{\frac{M}{2}}(X,T^{*0,q}X)\right)$}.\label{eq:k power decay I}
		\end{align}
		Here we recall that for $\chi\in\cCc({\dot\R})$ we can take 
		$\Td\chi$ such that $\Td\chi\in\cCc(\dot\C)$, so there is a constant $C>0$ such that $\frac{k}{C}<|z|<C k$ when $k^{-1}z\in \supp\Td\chi$.
		
  Combining all the estimates above, we finish our proof.
	\end{proof} 
From Theorems \ref{thm:HS formula for mathcal E_z is OK}, \ref{thm:HS formula for mathcal F_z is OK}, \ref{thm:HS formula for mathcal G_z is OK} and \ref{thm:HS formula for mathcal R_z is OK}, we can conclude the remainder estimates contributing by the elements of $L^{-\infty}_z(\Omega;T^{*0,q}X)$ as follows.
	\begin{theorem}
 \label{thm:F_(k,N+1) is OK}
	   For any operator $L_z\in L^{-\infty}_z(\Omega;T^{*0,q}X)$ of Definition \ref{def:L^-infty_z} and $N\in\N_0$ we have 
 \begin{align}
	&L_{(k)}:=\int_{\mathbb C}\frac{\partial\widetilde\chi_k}
	{\partial\overline z}(z-T_{P,\lambda}^{(q)})^{-1}\circ L_{z}\frac{dz\wedge d\overline z}{2\pi i}=O(k^{-N})~\text{in}~\mathscr L(H^{-N}_{\rm comp}(\Omega,T^{*0,q}X),H^{N}(X,T^{*0,q}X)).
	\end{align}
	\end{theorem}	
 The only remainder estimate remains to be checked is the following.
	\begin{theorem}
 \label{thm:R_(k,N+1) is OK for N large}
		With the same notations and assumptions in Theorem \ref{thm:expansion of (z-T_P)^-1 Pi}, for the operator
		\begin{equation}
		 R_{(k,N+1)}:=\int_\C\frac{\pr\Td\chi(\frac{z}{k})}{\pr\ol z}(z-T_{P,\lambda}^{(q)})^{-1}\circ\Pi_\lambda^{(q)}\circ R_{z,N+1}\frac{dz\wedge d\ol z}{2\pi i},  
		\end{equation}
 and for any $N_1,N_2\in\N_0$, we can find an $N_0>0$ large enough such that
	\begin{align}
	&     R_{(k,N_0+1)}=O(k^{-N_1})~\text{in}~\mathscr L(H^{-N_2}_{\rm comp}(\Omega,T^{*0,q}X),H^{N_2}(X,T^{*0,q}X)).
	\end{align}	
 \end{theorem}
\begin{proof}
For simplicity, we only prove the case for $n_-\neq n_+$, and the situation $n_-=n_+$ can be deduced from the same argument with some minor change. For all $M\in\N$, we recall that we can write
\begin{align}
    R_{(k,N+1)}
    =&\int_{\C\setminus \Spec(T_{P,\lambda}^{(q)})}\frac{\partial\widetilde\chi(\frac{z}{k})}
		{\partial\overline z}z^{-M} T_{P,\lambda}^{(q),M}\circ(z-T_{P,\lambda}^{(q)})^{-1}\circ R_{z,N+1} \frac{dz\wedge d\overline z}{2\pi i}\\
		+&\sum_{j=0}^{M-1}{T_{P,\lambda}^{(q),j}}\circ\int_{\C\setminus \Spec(T_{P,\lambda}^{(q)})}\frac{\partial\widetilde\chi(\frac{z}{k})}
		{\partial\overline z} {z^{-1-j}}~R_{z,N+1} \frac{dz\wedge d\overline z}{2\pi i}.\nonumber
\end{align}
We recall that here we have
\begin{equation}
    R_{z,N+1}(x,y)
		=\int_0^{+\infty}e^{it\Psi_-(x,y)}\frac{\sum_{|\beta|+|\gamma|\leq 2N+2}R_{\beta,\gamma}^{-,N+1}(x,y,t)t^\beta z^\gamma}{\left(z-t\right)^{2N+2}}\tau(\varepsilon t) dt,
\end{equation}
where $\Psi_-\in{\rm Ph}(p_{I_0,I_0}^{-1}(-\bm\alpha)(-\bm\alpha),\Omega)$, $R_{\beta,\gamma}^{-,N+1}(x,y,t)\in S^{n-N-1}_{\rm cl}(\Omega\times\Omega\times\R_+,\End(T^{*0,q}X))$, $R_{\beta,\gamma}^{-,N+1}(x,y,t)$ is properly supported in the variables $(x,y)$.

On one hand, by partial integration in $t$ and Cauchy--Pompieu formula, on $X\times\Omega$ we have
\begin{multline}
     \int_{\C\setminus \Spec(T_{P,\lambda}^{(q)})}\frac{\partial\widetilde\chi(\frac{z}{k})}
		{\partial\overline z} {z^{-1-j}}~R_{z,N+1}(x,y) \frac{dz\wedge d\overline z}{2\pi i}\\-\frac{1}{(2N+1)!}\sum_{\beta+\gamma\leq 2N+2}\int_0^{+\infty}e^{it\Psi_-(x,y)}R_{\beta,\gamma}^{-,N+1}(x,y,t)t^\beta\tau(\varepsilon t)
  \partial_t^{2N+1}\left(\chi\left(\frac{t}{k}\right)t^{-1-j+\gamma}\right)dt=\OK.
\end{multline}
We notice that for any given $N_1,N_2\in\N_0$, when $N$ is large enough, the operator, which is decided by 
\begin{equation}
    \int_0^{+\infty}e^{it\Psi_-(x,y)}R_{\beta,\gamma}^{-,N+1}(x,y,t)t^\beta\tau(\varepsilon t)
  \partial_t^{2N+1}\left(\chi\left(\frac{t}{k}\right)t^{-1-j+\gamma}\right)dt,
\end{equation}
is in $O(k^{-N_1})~\text{in}~\mathscr L(H^{-N_2}_{\rm comp}(\Omega,T^{*0,q}X),H^{N_2}(X,T^{*0,q}X))$.

On the other hand, for the integral
\begin{equation}
    \int_{\C\setminus \Spec(T_{P,\lambda}^{(q)})}\frac{\partial\widetilde\chi(\frac{z}{k})}
		{\partial\overline z}z^{-M} T_{P,\lambda}^{(q),M}\circ(z-T_{P,\lambda}^{(q)})^{-1}\circ R_{z,N+1} \frac{dz\wedge d\overline z}{2\pi i},
\end{equation}
when $M\equiv 0\mod 2$ we can rewrite it by
\begin{equation}
\label{eq:non-trivial part for semi-classical estimate of H_z,N+1}
    \int_{\C\setminus \Spec(T_{P,\lambda}^{(q)})}\frac{\partial\widetilde\chi(\frac{z}{k})}
		{\partial\overline z}z^{-M} T_{P,\lambda}^{(q),\frac{M}{2}}\circ\Pi_\lambda^{(q)}\circ(z-T_{P,\lambda}^{(q)})^{-1}\circ T_{P,\lambda}^{(q),\frac{M}{2}}\circ R_{z,N+1} \frac{dz\wedge d\overline z}{2\pi i}.
\end{equation}
Here the number $M\in\N$ is arbitrary and we can do the same estimate as in \eqref{eq:k power decay I}: by applying the Sobolev continuity estimate in the order of $T_{P,\lambda}^{(q),\frac{M}{2}}$, $\Pi_\lambda^{(q)}\circ (z-T_{P,\lambda}^{(q)})^{-1}$, $T_{P,\lambda}^{(q),\frac{M}{2}}$ and $R_{z,N+1}$,
and using $\left|\frac{\pr\Td\chi(\frac{z}{k})}{\pr \ol z}\right|=O\left( k^{-N}|\im z|^N\right)$ and the estimate that $k<|z|<2k$ when $z\in\supp\Td\chi(k^{-1}z)$ and $k$ is large, we can check that
 for any $N_1,N_2\in\N_0$, we can find a suitable and large $N_0>0$ and another large number $M>0$ depending on $N_0$ such that \eqref{eq:non-trivial part for semi-classical estimate of H_z,N+1} is $  O(k^{-N_1})~\text{in}~\mathscr L(H^{-N_2}_{\rm comp}(\Omega,T^{*0,q}X),H^{N_2}(X,T^{*0,q}X))$.
 
 Combing all the estimates above, we complete the proof of our theorem.
\end{proof}
By Theorems \ref{thm:principal term of chi_k(T_P) I}, \ref{thm:F_(k,N+1) is OK} and \ref{thm:R_(k,N+1) is OK for N large} and taking the asymptotic sum of the symbols of $A_{(k,m)}$, $m\in\N_0$, we can apply the standard semi-classical analysis and get the following result.
\begin{theorem}
\label{eq:operator expansion of chi_k(T_P)}
In the situation of Theorem \ref{thm:Main Semi-Classical Expansion}, for $q=n_-$, we have an ${{A}}_{(k)}\in\mathcal{I}^{0}_{\Sigma,k}(\Omega;T^{*0,q}X)$ such that for any $N_0\in\N$ we have
	\begin{align}
	\chi(k^{-1}T_{P,\lambda}^{(q)})
	-A_{(k)}=O(k^{-N_0})~\text{in}~\mathscr L(H^{-N_0}_{\rm comp}(\Omega,T^{*0,q}X),H^{N_0}(X,T^{*0,q}X)).
	\end{align}
\end{theorem}
By the proof of Sobolev inequality, we can treat the Dirac delta as in the Sobolev space of negative power. Then by the proof of \cite{Hoe03}*{Theorem 5.2.6}, we can immediately deduce the local picture of Theorems \ref{thm:Main Semi-Classical Expansion},~\ref{thm:Main Leading Term} at $q=n_-$ from Theorem~\ref{eq:operator expansion of chi_k(T_P)}. The far away asymptotics of Theorem \ref{thm:Main Semi-Classical Expansion} also follows from the properly supported property of the semi-classical Fourier integral operator in the last theorem. 
\begin{proof}[Proof of Corollary \ref{coro:Semi-Classical Expansion}]
On any coordinate patch $(\Omega_1,x)$, by Theorem \ref{thm:Main Semi-Classical Expansion} and the property that $\varphi_\mp(x,x)=0$, we have
	\begin{equation}
		\chi(k^{-1}T_{P,\lambda}^{(q)})(x,x)\sim\sum_{j=0}^{+\infty}k^{n+1-j} \left(A_j^-(x)+A_j^+(x)\right)
  \text{in}~S^{n+1}_{\rm loc}\left(1;\Omega_1,\End(T^{*0,q}X)\right).
		\end{equation}      
		where $A_j^\mp(x)\in\cC\left(\Omega_1,\End(T^{*0,q}X)\right)$, $A_j^\mp(x)=\int_0^{+\infty}A_j^\mp(x,x,t)dt$, and the local description of $A_0^\mp(x)$ is explicit through Theorem~\ref{thm:Main Leading Term}. We let $(\Omega_2,y)$ be another coordinate patch
with $\Omega_1\cap \Omega_2\neq\emptyset$, and by the same reasoning we have
		\begin{equation}
		\chi(k^{-1}T_{P,\lambda}^{(q)})(y,y)\sim\sum_{j=0}^{+\infty}k^{n+1-j} \left(B_j^-(y)+B_j^+(y)\right)
  \text{in}~S^{n+1}_{\rm loc}\left(1;\Omega_2,\End(T^{*0,q}X)\right).
		\end{equation}      
		where $B_j^\mp(x)\in\cC(\Omega_2,\End(T^{*0,q}X)$. Then on $\Omega_1\cap \Omega_2$ we have
		\begin{equation}
		\label{eq: A_j and B_j}
		\sum_{j=0}^{+\infty} (A_j^- + A_j^+)(\cdot)k^{n+1-j}\sim
		\sum_{j=0}^{+\infty} (B_j^- + B_j^+)(\cdot)k^{n+1-j}~
  \text{in}~S^{n+1}_{\operatorname{loc}}\left(1;\Omega_1\cap \Omega_2,\End(T^{*0,q}X)\right).
		\end{equation}
The relation \eqref{eq: A_j and B_j} shows that 
$A_j^- + A_j^+=B_j^- + B_j^+$ on $\Omega_1\cap \Omega_2$ and  our corollary follows. 
\end{proof}

\begin{proof}[Proof of Corollary \ref{coro:Scaled Spectral Measures}]
   Since measures are distributions of order zero, it suffices to prove that $\mu_k^{(q)}\to\mu_{+\infty}^{(q)}$ in the sense of distribution: for all $\chi\in\cCc(\dot\R)$, as $k\to+\infty$ we have
\begin{multline}
\langle\mu_k^{(q)},\chi\rangle=\frac{\sum_{j}\chi(k^{-1}\lambda_j)}{k^{n+1}}=\frac{\int_X\sum_I'\left\langle\chi(k^{-1}T_{P,\lambda}^{(q)})(x,x)T_I(x)\middle|T_I(x)\right\rangle dm(x)}{k^{n+1}}\\
=\int_X (A_0^-+A_0^+)(x)dm(x)=\mathcal{C}_P^{(q)}\int_{\R} \chi(t)t^n dt=\langle\mu_{+\infty}^{(q)},\chi\rangle.
\end{multline}
\end{proof}
\subsection{An example with the globally free circle action}
We explain the role of Corollary~\ref{coro:Semi-Classical Expansion} through a very important result for quantization introduced by Boutet de Monvel--Guillemin \cite{BG81}*{\S\S 13-14}. We assume that $(M, J)$ is a complex manifold with complex structure $J$ and $L$ is a holomorphic line bundle bundles over $M$ with a smooth Hermitian metric $h^L$. We assume that $\nabla^L$ 
is the holomorphic Hermitian connections, also known as Chern connections, on $(L, h^L)$ and
moreover, with respect to the Chern curvature $R^L:=(\nabla^L)^2$ the two form $\omega:=\frac{i}{2\pi}R^L$ defines a symplectic form on $M$. Therefore under this context the signature $(n_-,n_+)$ of
the curvature $R^L$ (the number of negative and positive eigenvalues) with respect
to any Riemannian metric compatible with $J$ will be the same. We let $g^{TX}$ be any
Riemannian metric on $TX$ compatible with $J$. We let $\ol\partial^{L^m,*}$
be the formal adjoint of the Dolbeault operator $\ol\partial^{L^m}$ on the
Dolbeault complex $\Omega^{0,q}(M,L^m)$, $q=0,\cdots,n:=\dim_\C M$, with the scalar product induced by $g^{TX}$ and $h^L$. We set $\mathbf D_m:=\sqrt{2}  \left(\ol\partial^{L^m}+\ol\partial^{L^m,*}\right)$ and denote by $\Box^{L^m}:=\ol\partial^{L^m,*}\circ\ol\partial^{L^m}+\ol\partial^{L^m}\circ\ol\partial^{L^m,*}
$ the Kodaira Lapalcian. It is clear that $\mathbf D_m^2=2\Box^{L^m}$ is twice the Kodaira Laplacian and preserves the $\Z$-grading
of $\Omega^{0,\cdot}(M,L^m)$. By standard Hodge theory, we know that for any $q,m\in\N$, $\ker \mathbf D_m|_{\Omega^{0,q}(M,L^m)}=\ker\mathbf D_m^2|_{\Omega^{0,q}(M,L^m)}=H^{0,q}(M,L^m)$, where $H^{0,q}(M,L^m)$ is the Dolbeault cohomology, $q=0\cdots,n$. For $\mathbf D_m^2$, from \cite{MaMar06IJM}*{Theorem 1.5} we have the Bochner--Kodaira--Nakano type formula and we have the following vanishing theorem: when $m\in\N$ is large we have $H^{0,q}(M,L^m)=0$ for $q\neq n_-$. The vanishing result above is Andreotti--Grauert’s coarse vanishing
theorem \cite{AnGr62}*{\S 23}, and the original proof is by using the cohomology finiteness theorem for the disc bundle of $L^*$.

When $q=n_-$, the situation is more interesting from the point of view of semi-classical analysis. We let $B_m^{(q)}:L^2_{0,q}(M,L^m)\to H^{0,q}(M,L^m)$ be the Bergman projection for $m$-power of $L$ on $(0,q)$ forms. The Schwartz kernel $B_m^{(q)}(p',p'')$ associated to $B_m^{(q)}$ is called the Bergman kernel, which is a smooth kernel by standard Hodge theory. It is well known that $B_m^{(q)}(p',p'')$ admits the full asymptotic expansion \cites{MaMar06IJM,MM07,HM14}, which is locally uniform. Combining with the identification $\mathscr L(L^{*,m},L^m)$ as $\C$, for $q=n_-$ and $m\in\N$ large enough we have the global expansion
\begin{align}
    &B_m^{(q)}(p')\sim\sum_{j=0}^{+\infty} m^{n-j}b_{j}^-(p')~\text{in}~S^{n}_{\rm loc}(1;M,\End T^{*0,q}M),\label{eq:BKE}\\
    &B_{-m}^{(n-q)}(p')\sim\sum_{j=0}^{+\infty} m^{n-j}b_{j}^+(p')~\text{in}~S^{n}_{\rm loc}(1;M,\End T^{*0,n-q}M),\label{eq:BKE for dual line bundle}
\end{align}
 where we use the convention $L^{-m}:=L^{*,m}$ for high power of dual line bundles and use the fact that $R^{L^*}$ has the constant signature $(n-q,q)$. 
 
 The principal circle bundle $X=\{v\in L^*:|v|_{h^*}=1\}$ is called Grauert tube \cite{Grau62}. Since $M$ is a complex manifold, the circle action is globally free. The connection $1$-form $\bm\alpha$ on $X$ associated to the Chern connection $\nabla^L$ is a contact form on $X$, and $X$ is the CR manifold we consider in this paper. The corresponding Reeb vector field $\mathcal{T}$ is the infinitesimal generator $\partial_\vartheta$ 
		of the $S^1$-action on $X$.	In this case, $\bm\alpha(-i\partial_\vartheta)=1$, $\partial_\vartheta$ 
		commutes with the Szeg\H{o} projection on functions, and ${\rm Spec}(-i\partial_\vartheta)=\Z$. Also, in this context we have the torsion free relation $[-i\partial_\vartheta,W]\subset \Gamma(T^{0,1}X)$ for all $W\in \Gamma(T^{0,1}X)$. Using the Lie derivatives, we can extend $-i\partial_\vartheta$ acting on $(0,q)$-forms. Now, we consider the Toeplitz operator $T_\vartheta^{(q)}:=\Pi^{(q)}\circ (-i\partial_\vartheta)\circ\Pi^{(q)}$ on $X$, which reflect the dynamics of the Reeb vector field. For $\chi\in\cCc(\R)$, by the canonical relation between the holomorphic sections of $M$ and the Fourier components of CR function on $X$, we have
	\begin{equation}
	\label{eq:chi_k(T_P)(x,x) on circle bundle}
	\chi(k^{-1}T_\vartheta^{(q)})(x,x)=\frac{1}{2\pi}\sum_{m\in\Z}\chi\left(\frac{m}{k}\right) B_m^{(q)}\circ\pi_M(x),
	\end{equation}
	where $\pi_M:X\to M$ is the natural projection such that $\pi_M(x)=p'$. For simplicity, from now on we write $\mathbf B_m^{(q)}:=B_m^{(q)}\circ\pi_M$ and $\mathbf b_j^\mp:=b_j^\mp\circ\pi_M$. Because of the term $\chi\left(\frac{m}{k}\right)$, when $k\to+\infty$, we only have to take consideration of $m$ satisfying $|m|\to+\infty$ in \eqref{eq:chi_k(T_P)(x,x) on circle bundle}. By \cite{HM17JDG}*{Theorem 1.12}, we can apply Corollary~\ref{coro:Semi-Classical Expansion} to \eqref{eq:chi_k(T_P)(x,x) on circle bundle}. The following discussion is to give a relatively elementary method to obtain the semi-classical expansion of \eqref{eq:chi_k(T_P)(x,x) on circle bundle} in this specific set-up.
    
    When $q\notin\{n_-,n_+\}$, by the vanishing theorem of Andreotti–-Grauert and \eqref{eq:chi_k(T_P)(x,x) on circle bundle}, we have $\chi(k^{-1}T_\vartheta^{(q)})(x,x)=0$ as $k\to+\infty$.
 
 When $q=n_-$, we have to split the discussion into $n_-\neq n_+$ and $n_-= n_+$. When $n_-\neq n_+$, again from the vanishing theorem of Andreotti–-Grauert we have $ \mathbf B^{(q)}_{-m}=0$, $m\to+\infty$, and when $q=n_-=n_+$, however, we have expansion result \eqref{eq:BKE for dual line bundle} for $\mathbf B^{(n-q)}_{-m}=\mathbf B^{(q)}_{-m}$. 
 
 For the generality of our calculation, from now on we consider $q=n_-=n_+$. In this case we notice that for all $N\in\N$, we have
 \begin{multline}
  \label{eq:Toeplitz and Bergman}
        \sum_{m\in\Z_{\lessgtr 0}}\chi\left(\frac{m}{k}\right)\mathbf B_m^{(q)}(x)
    -\sum_{m\in\Z_{\lessgtr 0}}\chi\left(\frac{m}{k}\right)\sum_{j=0}^N |m|^{n-j} \mathbf b_j^\pm(x)\\
    =\sum_{m\in\Z_{\lessgtr 0}}\chi\left(\frac{m}{k}\right)\mathbf B_m(x)
		-\sum_{j=0}^{N}k^{n+1-j}\sum_{m\in\Z_{\lessgtr 0}}k^{-1}\chi\left(\frac{m}{k}\right)\left(\frac{|m|}{k}\right)^{n-j} \mathbf b_j^\pm(x).
 \end{multline}
 On one hand, as $k\to+\infty$, \eqref{eq:BKE} and  \eqref{eq:BKE for dual line bundle} imply that for any $\ell\in\N$ we have
\begin{multline}
 \label{eq:Toeplitz and Bergman I.1}
    \left\|\sum_{m\in \Z_{\lessgtr 0}}\chi\left(\frac{m}{k}\right)\mathbf B_m^{(q)}(x)
    -\sum_{m\in\Z_{\lessgtr 0}}\chi\left(\frac{m}{k}\right)\sum_{j=0}^N |m|^{n-j} \mathbf b_j^\pm(x)\right\|_{C^\ell}\\
    \leq c^\pm_{\ell,N}\sum_{m\in \Z_{\lessgtr 0}}\left|\chi\left(\frac{m}{k}\right)\right|m^{n-N-1}=O(k^{n-N}).
\end{multline}
 for some constant $c^\pm_{\ell,N}>0$. On the other hand, by the Poisson summation formula, 
		cf.~\cite{Hoe03}*{Theorem 7.2.1} for example, for any $\tau\in\cCc(\R)$ we have
		\begin{equation}
		\label{eq:Possion summation formula}
		k^{-1}\sum_{m\in\Z}\tau\left(\frac{m}{k}\right)=\sum_{m\in\Z}\int_{\R}e^{-it(2\pi k m)}\tau(t)dt.
		\end{equation}
		In the right-hand side of the above equation, when $m\neq 0$ we can apply arbitrary times of integration by parts in $t$, and when $m=0$ we just have a number $\displaystyle\int_{\R}\tau(t)dt$. Accordingly, for any $N\in\N$, we can find a constant $C_N>0$ such that
		\begin{equation}
  \label{eq:estimate of Riemann sum}
		\left|k^{-1}\sum_{m\in\Z}\tau\left(\frac{m}{k}\right)-\int_\R\tau(t)dt\right|<C_N k^{-N}.
		\end{equation}
By \eqref{eq:Toeplitz and Bergman I.1} and \eqref{eq:estimate of Riemann sum} and triangle inequality, we immediately obtain
\begin{multline}
\label{eq:Expansion of chi(k^-1 T^(q))}
    \chi(k^{-1}T_\vartheta^{(q)})(x,x)\sim\sum_{j=0}^{+\infty} k^{n+1-j}\int_0^{+\infty}\chi(t)t^{n-j}\frac{dt}{2\pi}~\mathbf{b}_j^-(x)\\
    +\sum_{j=0}^{+\infty} k^{n+1-j}\int_0^{+\infty}\chi(-t)t^{n-j}\frac{dt}{2\pi}~\mathbf{b}_j^+(x)\ \ \text{in}\ \ S^{n+1}_{\rm loc}(1;M,\End L(T^{*0,q}M)):\ \ q=n_-=n_+.
\end{multline}
When $q=n_-\neq n_+$, from the same argument above we can see that the component contributed by the positive eigenvalues of $R^L$ in \eqref{eq:Expansion of chi(k^-1 T^(q))} will be $\OK$. 

For the related results about spectral asymptotics for Toeplitz operators on strictly pseudoconvex circle bundles, readers can refer to \cites{BPU98,P09,CRa23} for example. 

\subsection{An example with the locally free circle action}
We explain the role of Theorem~\ref{thm:Main Semi-Classical Expansion} on CR manifolds with transversal CR  circle action. An significant example of such CR manifold is the quasi-regular Sasakian manifold, and one important semi-classical expansion in this field is the asymptotic expansion for a weighted Bergman kernel of ample line bundles on orbifolds with cyclic quotient singularities \cite{RT11_1}. Despite the appearance of the singular points, such expansion is locally uniformly on the whole orbifold. We also refer to \cite{DLM11} for the case on general complex orbifolds. Such kind of asymptotics is fundamental in Sasakian geometry \cites{RT11_2,CS18,LP24}.

It is well known, cf.~\cite{BoGa08} for example, that the circle bundle of ample line bundles on orbifolds with cyclic quotient singularities is a quasi-regular Sasakian manifold. In other words, such circle bundle is the CR manifold $(X,T^{1,0}X)$ we consider in this paper with a transversal CR circle action. The related asymptotics result was also solved in \cite{HeHsLi18}.

From now on, we denote the circle action on $(X,T^{1,0}X)$ by $e^{i\theta}$ with the fundamental vector field $T$. The circle action is said to be CR if $[T,\Gamma(T^{1,0}X)]\subset\Gamma(T^{1,0}X)$, and is said to be transversal if we have $\C TX=T^{1,0}X\oplus T^{0,1}X\oplus\C T$. Then $T$ is a Reeb vector field on $X$ preserving $T^{1,0}X$, and by the transversal condition we know that such circle action is locally free. By Lie derivatives, we define $Tu:=\mathscr L_T u$, $\forall u\in\Omega^{0,q}(X)$. We have a rigid Hermitian metric $\langle\,\cdot\,|\,\cdot\,\rangle$ on $\C TX$ such that $T^{1,0}X\perp T^{0,1}X$, $T\perp (T^{1,0}X\oplus T^{0,1}X)$, $\langle\,T\,|\,T\,\rangle=1$ holds. We take the volume form on $X$ by this rigid Hermitian metric and define the $L^2$-inner product $(\,\cdot\,|\,\cdot\,)$ on $\Omega^{0,q}(X)$. The associated standard $L^2$-space is denoted by $L^2_{0,q}(X)$. For $q\in\{0,\cdots,n\}$ and $m\in\mathbb Z$, we put
\begin{equation}
\label{eq:Fourier component}
    \Omega^{0,q}_m(X):=\left\{u\in\Omega^{0,q}(X):\,
Tu=-imu\right\},\ \ H^q_{b,m}(X):=\set{u\in \Omega^{0,q}_m(X):\, \Box_b^{(q)}u=0}.
\end{equation}
Since $\Box_b^{(q)}+T^2$ is elliptic, we have $d_m^{(q)}:={\rm dim\,}H^q_{b,m}(X)<+\infty$. The $m$-th Fourier component of the Szeg\H{o} kernel is given by $S_m^{(q)}(x,y):=\sum_{j=1}^{d_m}f_{j,m}^{(q)}(x)\otimes{f_{j,m}^{(q),*}}(y)$, where $\{f^{(q)}_{1,m},\cdots,f^{(q)}_{d_m,m}\}$ forms an orthonormal basis for $H^q_{b,m}(X)$, and $S_m^{(q)}(x,y)$ is the Schwartz kernel of the orthogonal projection $S_m^{(q)}:L^2_{0,q}(X)\to H^q_{b,m}(X)$. 

For $x\in X$, we say that the period of $x$ is $\frac{2\pi}{r}$, $r\in\mathbb N$, if $e^{i\theta}\circ x\neq x$ for every $0<\theta<\frac{2\pi}{r}$ and $e^{i\frac{2\pi}{r}}\circ x=x$. For each $r\in\mathbb N$, we put $X_r=\set{x\in X:\, \text{the period of}\  x \ \text{is}\  \frac{2\pi}{r}}$ and $p:=\min\set{r\in\mathbb N:\, X_r\neq\emptyset}$. It is well-known from Duistermaat--Heckman~\cite{DH82} that if $X$ is connected, then $X_p$ is an open and dense subset of $X$ and we call $X_{p}$ the regular part of $X$. We may assume that $X=X_{p_1}\bigcup X_{p_2}\bigcup\cdots\bigcup X_{p_t}$, $p=:p_1<p_2<\cdots<p_t$. The semi-classical analysis on the singular part of $X$ is more difficult, and by \cite{HeHsLi18}*{Theorem 1.1} it turns out that the asymptotics of $S_m^{(q)}(x,x)$ holds differently on each $X_r$ when $q=n_-=0$.

We notice that  $\Omega^{0,q}(X)=\bigoplus_{m\in\Z}\Omega^{0,q}_m(X)$. Let $\{u_m\}_{m\in\Z}$, where $u_m\in H^q_{b,m}(X)$, be an orthonormal system of $\Omega^{0,q}(X)$. If we have $\Pi^{(q)}(-iT)\Pi^{(q)} u=\lambda_j u$, where $\lambda_j\neq 0$ and the $L^2$-inner product $(u|u_m)\neq 0$ for some $u_m\in H^{(q)}_{b,m}(X)$, then we can check that $\lambda_j=m$ and $u=u_m$. This implies that the positive spectrum of $\Pi^{(q)}(-iT)\Pi^{(q)}$ is the subset of $\N$. The converse side of the inclusion is almost by definition. So in this case  we actually have
\begin{equation}
\label{eq:eq:chi_k(T_P)(x,y) on quasi-regular}
    \chi(k^{-1}\Pi^{(q)}(-iT)\Pi^{(q)})=\sum_{m\in\Z}\chi(k^{-1}m)S_m^{(q)}.
\end{equation}
By \cite{HM17JDG}*{Theorem 1.12}, we can apply Theorem~\ref{thm:Main Semi-Classical Expansion} to find that \eqref{eq:eq:chi_k(T_P)(x,y) on quasi-regular} has the asymptotic expansion, which is locally uniformly on $X\times X$, as $k\to+\infty$. The following discussion is to give a relatively elementary method to obtain the semi-classical expansion of $\sum_{m\in\Z}\chi(k^{-1}m)S_m^{(q)}(x,y)$ in this specific set-up.

We need the classical result of Baouendi--Rothschild--Treves~\cite{BRT85}: For every point $x_0\in X$, we can find local coordinates $x=(x_1,\cdots,x_{2n+1})$, $z_j=x_{2j-1}+ix_{2j},j=1,\cdots,n$,  defined in some small neighborhood $D=\{(z, x_{2n+1}): \abs{z}<\delta, |x_{2n+1}|<\varepsilon_0\}$ of $x_0$, where $\delta>0$ and $0<\varepsilon_0<\pi$, such that $(z(x_0),x_{2n+1}(x_0))=(0,0)$ and
\begin{equation}
\label{eq:BRT vector fields}
\begin{split}
&T=-\frac{\pr}{\pr x_{2n+1}},\ \ Z_j=\frac{\partial}{\partial z_j}-i\frac{\partial\varphi}{\partial z_j}(z)\frac{\partial}{\partial x_{2n+1}},j=1,\cdots,n,
\end{split}
\end{equation}
where $Z_j(x), j=1,\cdots, n$, form a basis of $T_x^{1,0}X$, for each $x\in D$ and $\varphi(z)\in \mathscr C^\infty(D,\mathbb R)$ independent of $\theta$. We call $(D,(z,x_{2n+1}),\varphi)$ the BRT trivialization, $x=(z,x_{2n+1})$ the canonical coordinates and $\set{Z_j}^n_{j=1}$ the BRT frames.

From now on, we fix $m\in\mathbb{Z}$. We let $B:=(D,(z,\theta),\varphi)$ be a
BRT trivialization. We may assume that $D=U\times]-\varepsilon,\varepsilon[$,
where $\varepsilon>0$ and $U$ is an open set of $\mathbb{C}^n$.  Consider $
L\rightarrow U$ be a trivial line bundle with non-trivial Hermitian fiber
metric $\left\vert 1\right\vert^2_{h^L}=e^{-2\varphi}$. We let $\langle\,\cdot\,,\,\cdot\,
\rangle$ be the Hermitian metric on $\mathbb{C }TU$ given by
\begin{equation}
    \langle\,\frac{\partial}{\partial z_j}\,,\,\frac{\partial}{\partial z_k}
\,\rangle=\langle\,\frac{\partial}{\partial z_j}-i\frac{\partial\varphi}{
\partial z_j}(z)\frac{\partial}{\partial x_{2n+1}}\,|\,\frac{\partial}{\partial
z_k}-i\frac{\partial\varphi}{\partial z_k}(z)\frac{\partial}{\partial x_{2n+1}}
\,\rangle,\ \ 
j,k=1,2,\cdots,n
\end{equation}
The Hermitian metric $\langle\,\cdot\,,\,\cdot\,\rangle$ induces Hermitian metrics on $T^{*p,q}U$
bundle of $(p,q)$ forms on $U$, $p,q=0,1,\cdots,n$, also denoted by $\langle\,\cdot\,,\,\cdot\,\rangle$. Let $dv_U$ be the volume form on $U$ induced by $\langle\,\cdot\,,\,\cdot\,\rangle$. Note that $dv_X=dv_U(x)d\theta$ on $D$. Let $(\,\cdot\,,\,\cdot\,)_m$ be the $L^2$ inner product on $\Omega^{0,q}(U,L^m)$ induced by $\langle\,\cdot\,,\,\cdot\,\rangle$ and $h^{L^m}$, $q=0,1,2,\cdots,n$. We let $\Box_{B,m}^{(q)}$ be the Kodaira Laplacian on acting on $\Omega^{0,q}(U,L^m)$. We need the results of approximate Bergman kernel \cite{HM14}*{Theorems 3.11 \& 3.12}: we have a properly supported smoothing operator $P_{B,m}^{(q)}:\Omega^{0,q}(U,L^m)\To \Omega^{0,q}(U,L^m)$ such that $P_{B,m}^{(q)}\circ e^{-m\varphi}\circ\Box_{B,m}^{(q)}\circ e^{m\varphi}=O(m^{-\infty})$ on $U\times U$ and $e^{m\varphi}P_{B,m}^{(q)}e^{-m\varphi}u=u$, where $u\in \Omega^{0,q}(U)$ and $\Box_{B,m}^{(q)}u=0$. As in the last section, for the generality of our discussion, from now on we assume that $q=n_-=n_+$. We let $P_{B,m}^\pm(z,w)$ to be the distribution kernel of $P_{B,m}^{(q)}(z,w)$ when $m\in\Z_{\lessgtr 0}$, and for $|m|\gg 1$ we have
\begin{equation}\label{eq:BKE_forms}
\begin{split}
&P_{B,m}^\mp(z,w)=e^{i|m|\Psi_B^\mp(z,w)}b_B^\mp(z,w,m)+
O(m^{-\infty}),\\
&\exists\, c>0:\ {\rm Im\,}\Psi_B^\mp\geq c\abs{z-w}^2\,,\ \Psi_B^\mp(z,w)=0\Leftrightarrow z=w,\\
&b_B^\mp(z,w,m)\sim\sum^\infty_{j=0}b^\mp_{B,j}(z, w)|m|^{n-j}\text{ in }S^{n}_{{\rm loc\,}}
(1;U\times U,\End T^{*0,q}X).
\end{split}
\end{equation}
In fact, we have $\Psi_B^\mp(z,w)=i\sum^n_{j=1}\abs{\mu_j}\abs{z_j-w_j}^2+i\sum^n_{j=1}\mu_j(\ol
z_jw_j-z_j\ol w_j)+O(\abs{(z,w)}^3)$, where $\mu_1,\cdots,\mu_n$ are the eigenvalues of the Levi form and $b_B(z,w,m)$ and $b_{B,j}(z,w)$, $j=0,1,2,\cdots$, are all properly supported. 

Now, we write $X=D_1\bigcup D_2\bigcup\cdots\bigcup D_N$, where $B_j:=(D_j,(z,x_{2n+1}),\varphi_j)$ is a
BRT trivialization. We may assume that for each $j$, $
D_j=U_j\times]-4\delta,4\delta[\subset\mathbb{C}^n\times
\mathbb{R}$, $U_j=\left\{z\in\mathbb{C}
^n:\, \left\vert z\right\vert<\gamma_j\right\}$ and $\delta>0$ be a suitably small constant. For each $j$, we put
\begin{equation}
\label{eq:widehat D_j}
\widehat
D_j=\widehat U_j\times(-\frac{\delta}{2},\frac{\delta}{2}),    
\end{equation}
where
$\widehat U_j=\left\{z\in\mathbb{C}^n:\, \left\vert z\right\vert<\widehat\gamma_j\right\}\subset U_j$. We can take $\widehat\gamma_j$ suitably small such that
\begin{equation}
\label{eq:Psi B_j near the diagonal}
    \mbox{$|\Re\Psi_{B_j}(z,w)|<\frac{\delta}{2}$, $|\Im\Psi_{B_j}(z,w)|<\frac{\delta}{2}$, on $\widehat{U}_j\times\widehat{U}_j$.}
\end{equation}
 We may suppose that $X=\widehat D_1\bigcup\widehat
D_2\bigcup\cdots\bigcup\widehat D_N$. We let $\chi_j\in \cCc(\widehat D_j)$, $
j=1,2,\cdots,N$, with $\sum^N_{j=1}\chi_j=1$ on $X$. For a fixed $j\in\{1,2,\cdots,N\}$, we choose
\begin{equation}
\label{eq:cut off sigma_j}
    \mbox{$\sigma_j\in \cCc((-\frac{\delta}{2}
, \frac{\delta}{2}))$ with $\int_\R\sigma_j(\theta)d\theta=1$.}
\end{equation}
For $|m|\gg 1$, we let $
P_{B_j,m}^{(q)}(z,w)$ be the approximated Bergman kernel on $U_j\times U_j$ and take the continuous operator $H_{j,m}^\mp:\Omega^{0,q}(X)\rightarrow \Omega^{0,q}(X)$ given by
\begin{equation}  \label{eq:H_j,m^mp}
\begin{split}
H^\mp_{j,m}(x,y):=
\chi_j(x)e^{\mp imx_{2n+1}}P^\mp_{B_j,m}(z,w)e^{\pm im y_{2n+1}}\sigma_j(\eta):\ \ m\in\Z_{\lessgtr 0}
\end{split}
\end{equation}
where $x=(z,\theta)$, $y=(w,\eta)\in\mathbb{C}^{n}\times\mathbb{R}$. Since $P_{B_j,m}^{(q)}$ is properly supported, $H^\mp_{j,m}$ is well-defined. For the orthogonal projection $Q_m^{(q)}:L^2(X,T^{*0,q}X)\To L^2_m(X,T^{*0,q}X)$ with respect to $(\,\cdot\,|\,\cdot\,)$, we consider $\Gamma_m^\mp:=\sum^N_{j=1}H_{j,m}^\mp\circ
Q_m^{(q)}:\Omega^{0,q}(X)\rightarrow\Omega^{0,q}(X)$, and let $\Gamma_m^\mp(x,y)\in \mathscr C^\infty(X\times X,\End T^{*0,q}X)$ be the distribution
kernel of $\Gamma_m^\mp$. We can check that for $|m|\gg 1$, we have
\begin{equation}\label{eq:Gamma_m integral}
\begin{split}
\Gamma^\mp_m(x,y)=\sum^N_{j=1}\int^\pi_{-\pi}H^\mp_{j,m}(x,e^{iu}\circ
y)e^{imu}\frac{du}{2\pi}=\sum^N_{j=1}\int^\pi_{-\pi}e^{im\widehat
\Psi^\mp_{B_j}(x,e^{iu}\circ y)}\widehat b^\mp_{B_j}(x,e^{iu}\circ y,m)e^{imu}\frac{du}{2\pi},
\end{split}
\end{equation}
where on $D_j$ we have
\begin{equation}
\label{eq:Phase of Gamma_m}
   \widehat\Psi^\mp_{B_j}(x,y)=\mp x_{2n+1}\pm y_{2n+1}+\Psi^\mp_{B_j}(z,w),
\end{equation}
\begin{equation}
\label{eq:symbol of Gamma_m}
     \widehat b^\mp_{B_j}(x,y,m)=\chi_j(x) b^\mp_{B_j}(z,w,m)\sigma_j(\eta).
\end{equation}
By the slight modification of the proof of \cite{HeHsLi18}*{Theorem 4.13}, on $X\times X$ we have
\begin{eqnarray}
    \mbox{$\Gamma_m^{-}=S_m^{(q)}+O(m^{-\infty})$: $m>0$;\ \ $\Gamma_m^{+}=S_m^{(q)}+O(|m|^{-\infty})$: $m<0$.}
\end{eqnarray}

For simplicity, in the following we only discuss the $\mathscr C^0$-estimate of $\Gamma_m^\mp(x,y)$, and the general $\mathscr C^\ell$-estimate follows from the straightforward modification. Since $\chi\in\cCc(\R)$, we have $\chi=\chi_-+\chi_+$, where $\chi_\pm\in\cCc(\R_{\lessgtr 0})$. Combining with \eqref{eq:Gamma_m integral} and \eqref{eq:BKE_forms}, for $k>0$ large enough and any $M\in\N$ we have a constant $C_M>0$ independent of $k$ such that
\begin{equation}
\label{eq:estimate of weighted Gamma_m}
\begin{split}
     &\left|\sum_{m\in\Z_{\lessgtr 0}}\chi_\pm(k^{-1}m)\Gamma_m^\pm(x,y)\right.\\
&\left.-\frac{1}{2\pi}\sum^N_{j=1}\int^{\pi}_{-\pi}\sum_{m\in\Z_{\lessgtr 0}}\chi_\pm(k^{-1}m)e^{i|m|\widehat
\Psi^\pm_{B_j}(x,e^{iu}\circ y)}\sum_{\ell=0}^{M-1}\widehat b^\mp_{B_j,\ell}(x,e^{iu}\circ y)|m|^{n-\ell} e^{imu}du\right|\leq C_M k^{1+n-M}.
\end{split}
\end{equation}
On the other hand, as before, using $\chi_\pm\in\cCc(\R_{\lessgtr 0})$ and Poisson summation formula we have
\begin{equation}
\label{eq:weighted sum of Gamma_m}
    \begin{split}
    &k^{-1}\sum_{m\in\Z}\chi_\mp(\pm k^{-1}m)e^{ik (k^{-1}m)\widehat
\Psi^\mp_{B_j}(x,e^{iu}\circ y)}\sum_{\ell=0}^{M-1}\widehat b^\mp_{B_j,\ell}(x,e^{iu}\circ y)k^{1+n-\ell}(k^{-1}m)^{n-\ell} e^{\pm ik (k^{-1}m)u}\\
&=\sum_{m\in\Z}\int_\R e^{ikt(-2\pi m\pm u+\widehat
\Psi^\mp_{B_j}(x,e^{iu}\circ y))}\chi_\mp(\pm t)\sum_{\ell=1}^{M-1}\widehat b^\mp_{B_j,\ell}(x,e^{iu}\circ y)k^{n+1-\ell}t^{n-\ell}  dt.
\end{split}
\end{equation}
By the properly supported property of the approximated Bergman kernels \eqref{eq:symbol of Gamma_m}, for the latter discussion we may assume that $(x,e^{iu}\circ y)\in \widehat{D}_j\times\widehat{D}_j$. We notice that when $m\neq 0$ and $u\in[-\pi,\pi)$, by \eqref{eq:widehat D_j}, \eqref{eq:Psi B_j near the diagonal} and \eqref{eq:Phase of Gamma_m}, we have $|-2\pi m+u+\widehat{\Psi}|\geq \pi-2\delta$, so we can apply arbitrary times of integration by parts in $t$ and find \eqref{eq:weighted sum of Gamma_m} equals to
\begin{equation}
\int_\R e^{ikt(\pm u+\widehat
\Psi^\mp_{B_j}(x,e^{iu}\circ y))}\sum_{\ell=1}^{M-1}\widehat b^\mp_{B_j,\ell}(x,e^{iu}\circ y)k^{n+1-\ell}t^{n-\ell}\chi_\mp(\pm t)  dt+O(k^{-\infty}).
\end{equation}
Accordingly, to estimate \eqref{eq:estimate of weighted Gamma_m}, it suffices to calculate the sum of integrals
\begin{equation}
\begin{split}
    &\sum_{j=1}^N\int_{2\delta}^{2\pi-2\delta}\int_\R e^{ikt(\widehat
\Psi^\mp_{B_j}(x,e^{iu}\circ y)\pm u)}\sum_{\ell=1}^{M-1}\widehat b^\mp_{B_j,\ell}(x,e^{iu}\circ y)k^{n+1-\ell}t^{n-\ell}\chi_\mp(\pm t) dt \frac{du}{2\pi}\\
&+\sum^N_{j=1}\int^{2\delta}_{-2\delta}\int_\R e^{ikt(\widehat
\Psi^\mp_{B_j}(x,e^{iu}\circ y)\pm u)}\sum_{\ell=1}^{M-1}\widehat b^\mp_{B_j,\ell}(x,e^{iu}\circ y)k^{n+1-\ell}t^{n-\ell}\chi_\mp(\pm t) dt \frac{du}{2\pi}.
\end{split}
\end{equation}
For the first integral above, by \eqref{eq:widehat D_j}, \eqref{eq:Psi B_j near the diagonal} and \eqref{eq:Phase of Gamma_m} again, for $u\in[2\delta,2\pi-2\delta]$ we have
\begin{equation}
    \begin{split}
        &\mbox{$\Re (\widehat\Psi^\mp_{B_j}(x,y)\pm u)=\mp x_{2n+1}\pm y_{2n+1}\pm u+\Re\Psi^\mp_{B_j}(z,w)>\frac{\delta}{2}$}.
    \end{split}
\end{equation}
Hence, we can apply any times of integration by parts in $t$, and the first integral above is $k$-negligible. We remark that this term is quite complicated if we do not take the weighted sum over $m$, and for each fixed and large $m$ it will contribute some exponential error term dues to the periodicity of the set $X_q$, cf.~\cite{HeHsLi18}*{Theorem 1.1}. As for the second integral above, again by the properly supported property of the approximated Bergman kernels \eqref{eq:symbol of Gamma_m}, we may assume that for each $j$ the calculation is applied within $D_j\times D_j$, and by \eqref{eq:BRT vector fields}, \eqref{eq:Phase of Gamma_m}, \eqref{eq:symbol of Gamma_m} and \eqref{eq:cut off sigma_j}, we get
\begin{equation}
\label{eq:FIO_QRS}
    \begin{split}
    &\sum^N_{j=1}\int_\R\int^{2\delta}_{-2\delta} e^{ikt[\mp y_{2n+1}+\widehat
\Psi^\mp_{B_j}(x,(w,-u))\pm u]}\sum_{\ell=1}^{M-1}\widehat b^\mp_{B_j,\ell}(x,(w,-u))k^{n+1-\ell}t^{n-\ell}\chi_\mp(\pm t)  \frac{du}{2\pi} dt\\
=&\sum^N_{j=1}\int_\R\ e^{ikt(\mp x_{2n+1}\pm y_{2n+1}+\Psi^\mp_{B_j}(z,w))}\sum_{\ell=1}^{M-1}k^{n+1-\ell}t^{n-\ell}\chi_\mp(\pm t)\int_{-2\delta}^{2\delta}\widehat b^\mp_{B_j,\ell}(x,(w,-u))\frac{du}{2\pi} dt\\
=&\sum^N_{j=1}\chi_j(x)\sum_{\ell=1}^{M-1}k^{n+1-\ell}\frac{b^\mp_{B_j,\ell}(z,w)}{2\pi}\int_\R\ e^{ikt(\mp x_{2n+1}\pm y_{2n+1}+\Psi^\mp_{B_j}(z,w))}t^{n-\ell}\chi_\mp(\pm t) dt.
    \end{split}
\end{equation}
Combining all the above argument and taking the asymptotic sum for $M\to+\infty$, in view of \eqref{eq:BKE_forms}, when $q=n_-=n_+$ we obtain
\begin{multline}
\label{eq:asymptotic expansion of chi_k(T_P)_QRS}
\sum_{m\in\Z}\chi(k^{-1}m)S_m^{(q)}(x,y)
=\int_0^{+\infty} 
e^{ikt(-x_{2n+1}+y_{2n+1}+\Phi^-(z,w))}{A}^-(z,w,t,k)dt\\
+\int_0^{+\infty} 
e^{ikt(x_{2n+1}-y_{2n+1}+\Phi^+(z,w))}{A}^+(z,w,t,k)dt+O\left(k^{-\infty}\right)~\text{on}~D\times D,
\end{multline}
as a refinement of Theorem \ref{thm:Main Semi-Classical Expansion} in this specific set-up. The coefficients of the expansion of ${A}^\mp(z,w,t,k)$ satisfy $2\pi{A}^\mp_{j}(z,z,t)=b_{B,j}^\mp(z,z)\chi(\pm t)t^{n-j}$, $j\in\N_0$. We refer $b^\mp_{B,1}(z,z)$ to \cites{Hs16,Lu15}.

 We refer to \cite{HHL20} for the on-diagonal expansion on irregular Sasakian manifolds. We also remark that when $(X,T^{1,0}X)$ is an irregular Sasakian manifold, in general the Fourier component of the $L^2$-function on $X$ is represented by the real (not necessarily integer) eigenvalue of $-iT$. Hence the argument here cannot be applied directly to irregular Sasakian cases because the localized Bergman kernel is only defined by high integer power of line bundles and we apply Poisson summation formula.


  \addtocontents{toc}{\protect\setcounter{tocdepth}{0}}
  \section*{Acknowledgement}
  This paper is part of the thesis of the author at Universität zu Köln, and he wishes to express profound gratitude to George Marinescu, Chin-Yu Hsiao and  Hendrik Herrmann for their  invaluable suggestions concerning the problem addressed in this work. 
  \addtocontents{toc}{\protect\setcounter{tocdepth}{2}}
\bibliography{ref}
\end{document}